\documentclass{siamltex}

\usepackage{graphicx}%
\usepackage{multirow}%
\usepackage{amsmath,amssymb,amsfonts}%
\usepackage{mathrsfs}%
\usepackage[title]{appendix}%
\usepackage{xcolor}%
\usepackage{textcomp}%
\usepackage{manyfoot}%
\usepackage{booktabs}%
\usepackage{algorithm}%
\usepackage{algorithmicx}%
\usepackage{algpseudocode}%
\usepackage{listings}%
\usepackage{amsfonts,latexsym}
\usepackage{epsfig}
\usepackage{multirow}
\usepackage{amsmath}
\usepackage{graphicx}
\usepackage{caption}
\usepackage{subcaption}
\usepackage{amsmath}
\usepackage{hyperref}
\usepackage{parskip}
\usepackage{color}
\usepackage{ulem}
\usepackage{epstopdf} 
\usepackage{empheq}
\usepackage[utf8]{inputenc}

\usepackage{amsfonts,latexsym}
\usepackage{epsfig}
\usepackage{graphicx}
\usepackage{amsmath}
\usepackage{color}
\usepackage{mathptmx}  
\usepackage{algpseudocode}%
\usepackage{listings}%

\title{Low-Rank Regularized Convex-Non-Convex Problems for Image Segmentation or
	Completion }
	
	\author{M. El Guide\thanks{Mohammed VI Polytechnic University, Rabat, Morocco. email{mohamed.elguide@um6p.ma}}               
		\and 
		A. El Hachimi\footnotemark[3] \thanks{The UM6P Vanguard center, Mohammed VI Polytechnic University, Rabat, Morocco. email{anaselhachimi1997@gmail.com}  }
		\and 
		K. Jbilou \thanks{Universit\'e du Littoral, C\^ote d'Opale, batiment H. Poincarr\'e,
		50 rue F. Buisson,  F-62280 Calais Cedex, France. email{jbilou@lmpa.univ-littoral.fr}} 
		\and L. Reichel \thanks{Department of Mathematical Sciences, Kent State University, Kent 44242, Ohio USA. email{reichel@math.kent.edu}}
	}

	\date{}
	
\begin{document}
		\maketitle

	\begin{abstract}{This work proposes a novel convex-non-convex formulation of the image 
		segmentation and the image completion problems. The proposed approach is based on the 
		minimization of a functional involving two distinct regularization terms: one promotes 
		low-rank structure in the solution, while the other one enforces smoothness. To solve the
		resulting optimization problem, we employ the alternating direction method of multipliers
		(ADMM). A detailed convergence analysis of the algorithm is provided, and the performance 
		of the methods is demonstrated through a series of numerical experiments.}
	\end{abstract}
	
	
	

\section{Introduction}
Image segmentation and completion are fundamental computational tasks in computer vision 
and data analysis with many applications in medical and hyperspectral imaging, as well as
in remote sensing. These tasks arise when reconstructing or segmenting images with missing
or occluded pixels. Existing approaches, including variational methods \cite{Rudin} and 
deep learning techniques \cite{Long} have been applied to solve these tasks under specific 
circumstances. However, these approaches have limitations, such as sensitivity of the 
computed results to noise in the given data and high computational cost. 

Recovery of meaningful information from corrupted or incomplete data is a long-standing 
problem in image processing. For example, in medical imaging, modalities such as Magnetic 
Resonance Imaging (MRI) and Computerized Tomography (CT) frequently yield data that suffer 
from noise contamination or incomplete acquisition due to hardware limitations or patient 
constraints \cite{MRI_compression,MRI_2}. Similarly, hyperspectral imaging, known for its
ability to capture detailed spectral information, often suffers from missing data values 
because of sensor malfunction \cite{HI_denoising2,HI_denoising1}. These difficulties make
image recovery and segmentation indispensable for accurate data analysis, diagnosis, and 
decision-making. Traditional approaches to deal with these issues include the use of i) 
regularization, e.g., Total Variation (TV) regularization \cite{Amir_denoising,Benchatou,bentbib1},
ii) low-rank matrix or tensor factorization \cite{Completion, HI_denoising2}, and iii) 
variational formulations such as the Mumford-Shah model \cite{Mumford--Shah model}. The 
latter has inspired many variations, including the two-stage Mumford-Shah model by Cai et 
al. \cite{two-stage Mumford--Shah model} and the Convex-Non-Convex (CNC) model proposed 
by Chan et al. \cite{CNC}. Despite their effectiveness, these approaches face limitations, 
such as high computational complexity, sensitivity to parameter settings, and difficulties
in preserving fine details in high-dimensional data.

Image segmentation involves partitioning an image into meaningful regions based on certain
properties of the image such as intensity, texture, or spectral information. Classical
techniques, including active contour methods \cite{T.Chan} and region-growing algorithms 
\cite{Espindola}, have been studied extensively. Variational approaches, such as the 
Mumford-Shah model \cite{Mumford--Shah model} and its convex variants 
\cite{two-stage Mumford--Shah model}, have advanced the field by
providing robust mathematical formulations. More recently, the hybrid convex-non-convex 
framework has gained prominence due to its ability to balance computational feasibility 
with modeling flexibility; see Chan et al. \cite{CNC}. We remark that traditional convex 
formulations, such as those based on the nuclear norm, are robust but can be overly 
restrictive and may give suboptimal solutions. Conversely, purely non-convex approaches, 
while offering greater flexibility, often lack stability and scalability; see, e.g., Lu et
al. \cite{Lu}.

To overcome these limitations, this paper introduces a hybrid convex-non-convex 
optimization framework for image segmentation and completion. By combining the robustness
of convex optimization and the flexibility of non-convex regularization, the proposed 
methods strike a balance between accuracy and computational efficiency. The methods use 
regularization that promotes matrix solutions of low rank. This has emerged as a powerful 
technique for image recovery as well as for image segmentation. Low-rank matrix solutions 
can be represented with reduced degrees of freedom; see Bell et al. \cite{Completion}. We
seek to determine low-rank matrix solutions by regularization with the nuclear matrix 
norm, which is a convex surrogate of the matrix rank function 
\cite{Completion_nuclear_norm}. However, while the 
use of the nuclear norm ensures mathematical tractability, the solution of minimization 
problems that involve the nuclear matrix norm can be very demanding computationally for 
large-scale problems. To address this issue, the Alternating Direction Method of 
Multipliers (ADMM) is frequently employed; see, e.g., \cite{Bu,ADMM1,LPS,ADMM2}. The 
ability of ADMM to decompose complex problems into simple subproblems has made it a 
popular method for solving a variety of optimization problems in image processing. 
Computed examples presented in Section \ref{sec 4} show the proposed methods to be 
competitive with available methods both in terms of accuracy and computing time. 

We formulate image completion and image segmentation within a unified framework.
Let $B\in\mathbb{R}^{n_1\times n_2}$ represent an observed image. For image completion,
the matrix $B$ has zero or blank entries at locations for which no data is available, and
for image segmentation the matrix $B$ represents a fully observed noisy image. The desired
solution $U\in\mathbb{R}^{n_1\times n_2}$ of many image processing problems is a piecewise 
smooth image. Therefore, many solution methods use regularization that promotes piecewise 
smoothness of the computed solution. For instance, this can be achieved with Total 
Variation (TV) regularization; see \cite{Benchatou,elha2,elha1}. We obtain a 
minimization problem of the form
\begin{equation}\label{main2}
	\min_{U\in\mathbb{R}^{n_1\times n_2}}\left\{\dfrac{\lambda}{2}
	\left\Vert U - B\right\Vert_F^2 + R(U)\right\},
\end{equation}
where $R:\mathbb{R}^{n_1\times n_2}\to\mathbb{R}$ denotes a total variation regularization
functional, $\|\cdot\|_F$ stands for the Frobenius norm, and $\lambda>0$ is a 
regularization parameter that balances the influence of the terms in \eqref{main2} on the 
solution. Chan et al. \cite{CNC} proposed the use of the total variation functional
\[
R(U)=\sum_{i=1}^{n_1}\sum_{j=1}^{n_2}\phi(\|\left(\nabla U\right)_{ij}\|_F,a,T),
\]
where
\begin{equation}\label{phifun}
\phi\left( t;T,a \right)=\left\lbrace\begin{array}{ll}
\phi_1\left(t;T,a\right)= \dfrac{a(T_2-T)t^2}{2T}, & t\in [0,T),  \\[2mm]
\phi_2\left(t;T,a\right)= -\dfrac{a}{2}t^2 + aT_2 t - \dfrac{aTT_2}{2}, & t\in [T,T_2),
\\[3mm]
\phi_3\left(t;T,a\right)=\dfrac{aT_2 \left(T_2 - T\right)}{2},~~ & t \in [T_2,\infty).
\end{array} \right.
\end{equation}
The parameters $T>0$ and $a>0$ have to be chosen by a user: $a$ is used to tune the degree
of non-convexity of the regularization functional as discussed below, while $T$ is chosen 
to emphasize interesting features of the image. In image segmentation, $T$ is used to 
indicate which entries of the image should not be considered boundary points of the 
segmented regions of the images. The chosen form behaves like a quadratic smoothing term 
for small gradients but transitions to a flat response for large gradients, thereby 
preserving edges, as proposed by Chan et al. \cite{CNC}. We comment on the choices of $a$,
$T$, and $T_2$ below; see also Chan et al. \cite{CNC} for a discussion on how to choose 
these parameters.

Following Chan et al. \cite{CNC}, we define a gradient-type operator $\nabla U$ by 
\[
\left(\nabla U\right)_{ij} =[\left(D_1 U\right)_{ij}, \left(D_2 U\right)_{ij}]^T,
\]
where the superscript $^T$ stands for transposition,
\[
D_1 U=U C_1,\quad D_2 U=C_2 U,
\]
and $C_1$ and $C_2$ are the circulant matrices
\[
C_1=\begin{bmatrix}
	-1 & 0 & 0 & \ldots & 1\\
	1 & -1 & 0 & \ldots & 0\\
	\vdots & \ddots & \ddots  & \ddots & \vdots\\
	0 & \ldots & 1 & -1 & 0\\
	0 & \ldots & 0 & 1 & -1\\
\end{bmatrix}\in \mathbb{R}^{n_2 \times n_2}, \quad 
C_2=\begin{bmatrix}
	-1 & 1 & 0 & \ldots & 0\\
	0 & -1 & 1 & \ldots & 0\\
	\vdots & \ddots & \ddots  & \ddots & \vdots\\
	0 & \ldots & 0 & -1 & 1\\
	1 & \ldots & 0 & 0 & -1\\
\end{bmatrix}\in \mathbb{R}^{n_1 \times n_1}.
\]

The function $\phi$ in \eqref{phifun} satisfies: 
\begin{enumerate}
	\item $\phi$ is continuously differentiable for $t\in \mathbb{R}_+$.
	\item $\phi$ is twice continuously differentiable for 
	$t\in \mathbb{R}_+\backslash \{T,T_2 \}$.
	\item $\phi$ is convex and monotonically increasing for $t\in [0,T)$.
	\item $\phi$ concave and monotonically non-decreasing for $t\in [T,T_2)$.
	\item $\phi$ is constant for $t\in [T_2,+\infty)$.
	\item $\inf_{t\in \mathbb{R}_+\backslash \{T,T_2\}} \phi''=-a$.
\end{enumerate}
\vskip3pt

We would like to determine an approximate solution of \eqref{main2} of low rank. Recovery
of data so that the computed solution or correction of an available approximate solution 
are of low rank has received considerable attention; see, e.g., 
\cite{HI_denoising2,MRI_compression,LMRS,Rudin,HI_denoising1}. Reasons for this include 
that the imposition of a low-rank constraint may be meaningful in the model considered,
may yield a computed solution of higher quality, or may be easier to interpret. We 
therefore modify the minimization problem \eqref{main2} to obtain
\begin{equation}\label{main3} 
	\min_{U\in\mathbb{R}^{n_1\times n_2}}\left\{ \dfrac{\lambda}{2}
	\left\Vert U - B\right\Vert_F^2 + R(U)\right\}, \text{~~such~that~~~rank}(U)\leq r,
\end{equation}
where $r$ is user-specified positive integer. 

However, the solution of the minimization problem \eqref{main3} is NP-hard. It therefore 
is impractical to solve \eqref{main3}, except when all matrices involved are very small. 
To circumvent this difficulty, we replace the rank function by its convexification, which 
is the matrix nuclear norm; see \cite{Completion,MFazel,Completion_nuclear_norm} for 
discussions on this replacement. We solve the minimization so obtained by ADMM.

In summary, the contributions of this work are: (i) a new unified model for image 
completion and segmentation that promotes both low-rank structure and piecewise 
smoothness, (ii) an efficient ADMM-based algorithm with proven convergence properties, and
(iii) a comprehensive experimental evaluation for a variety of images (grayscale, color, 
hyperspectral) that demonstrate improved accuracy and speed compared to existing methods.

This paper is organized as follows. Section \ref{sec 2} discusses the solution of the 
optimization problem \eqref{main3} with the rank function replaced by the nuclear norm, 
and Section \ref{sec 3} is concerned with the convergence properties of the proposed 
solution method. Numerical examples that illustrate the performance of the solution 
methods are reported in Section \ref{sec 4}. Concluding remarks can be found in Section 
\ref{sec 5}. 

It is a pleasure to dedicate this paper to \AA ke Bj\"orck and Lars Eld\'en, who made
profound contributions to Numerical Linear Algebra and pioneered the development of 
numerical methods for the solution linear discrete ill-posed problems; see, e.g., 
\cite{Bj1,Bj2,Bj3,Bj4,El1,El2,El3,El4}.

\section{The solution method}\label{sec 2}
This section describes the proposed solution method. We convexify the rank-regularized 
problem \eqref{main3} by using the nuclear norm surrogate, and derive an efficient 
algorithm based on ADMM to solve the resulting problem. 

Let the matrix $U\in\mathbb{R}^{n_1\times n_2}$ have the singular values
$\sigma_1\geq\sigma_2\geq\ldots\geq\sigma_{\min\{n_1, n_2\}}\geq 0$. The nuclear norm of
$U$ is defined as
\[
\left\Vert U\right\Vert_* = \sum_{i=1}^{\min\{n_1, n_2\}} \sigma_i.
\]
The matrix nuclear norm is continuous, convex, and coercive. Its subdifferential is given
by 
\[
\partial \left\Vert X \right\Vert_*= \left\{  UV^T+W:\; W\in\mathbb{R}^{n_1\times n_2},\;
U^TW=0,\; W^TV=0,\; \left\Vert W \right\Vert_2\leq 1 \right\}.
\]
Here and throughout this paper $\|\cdot\|_2$ denotes the spectral matrix norm or the
Euclidean vector norm.

Replacing the rank constraint in the optimization problem \eqref{main3} by the matrix 
nuclear norm gives the minimization problem 
\begin{equation}\label{main4}
	\min_{U\in\mathbb{R}^{n_1\times n_2}}\left\{\dfrac{\lambda}{2}
	\left\Vert U- B\right\Vert_F^2 + R(U)+ \left\Vert U \right\Vert_*\right\}.
\end{equation}
We introduce some auxiliary variables to obtain an equivalent minimization problem that 
we will solve:
\begin{equation}\label{main5}
\min_{U,Z,M} \left\{\dfrac{\lambda}{2}\left\Vert U - B\right\Vert_F^2 + 
\sum_{i=1}^{n_1}\sum_{j=1}^{n_2} \phi\left( \left\Vert M_{ij}\right\Vert_2; T,a \right) +
\left\Vert Z\right\Vert_*\right\},
\end{equation}
such that $DU=M$ and $Z=U$. Here, $M$ captures the discrete gradient of $U$ (so that the
regularizer $R(U)$ can be written in terms of $M$), and $Z$ is introduced to facilitate 
the nuclear norm term. These substitutions allow splitting the terms of the objective 
function in the ADMM framework and enable more efficient optimization. In detail, we 
define the differential operator $D=[D_1^T,D_2^T]^T$, and express the auxiliary variable 
$M=[M_{ij}]$, where each matrix entry $M_{ij}=[(D_1U)_{i,j},(D_2U)_{i,j}]^T$ captures the
local gradient information. We propose to solve the resulting optimization problem using 
the ADMM; see, e.g., \cite{Bu,ADMM1,LPS,ADMM2} for discussions of this method. The 
augmented Lagrangian is given by 
\begin{eqnarray*}
	L_{\beta_1,\beta_2}\left(U,Z,M,Q,O\right)=\dfrac{\lambda}{2}\left\Vert U -
	B\right\Vert_F^2 &+& \sum_{i=1}^{n_1}\sum_{j=1}^{n_2} 
	\phi\left( \left\Vert M_{ij}\right\Vert_2; T,a \right) + 
	\left\Vert Z\right\Vert_* \\ 
	&+& \langle DU-M,Q \rangle + \dfrac{\beta_1}{2}\left\Vert DU-M \right\Vert_F^2\\
	&+& \langle Z-U,O \rangle +\dfrac{\beta_1}{2}\left\Vert Z-U \right\Vert_F^2,
\end{eqnarray*}
where $\beta_1>0$ and $\beta_2>0$ are Lagrangian parameters and
$Q\in\mathbb{R}^{2n_1\times n_2}$ and $O\in\mathbb{R}^{n_1\times n_2}$ are Lagrangian 
multipliers. The inner product of $U=[u_{ij}]\in\mathbb{R}^{n_1\times n_2}$ and 
$V=[v_{ij}]\in\mathbb{R}^{n_1\times n_2}$ is defined as
\begin{equation}\label{innerprod}
\langle U,V\rangle={\rm trace}(U^TV)=\sum_{i=1}^{n_1}\sum_{j=1}^{n_2}u_{ij}v_{ij}. 
\end{equation}
In particular, $\|U\|_F=\langle U,U\rangle^{1/2}$.

At the $(k+1)$st iteration of ADMM, we solve the following subproblems
\begin{eqnarray}
	&U^{k+1}=&\arg\min_{U} L_{\beta_1,\beta_2}\left(U,Z^k,M^k,Q^k,O^k\right)\label{subp 1},\\
	&Z^{k+1}=&\arg\min_{Z} L_{\beta_1,\beta_2}\left(U^{k+1},Z,M^k,Q^k,O^k\right), 
	\label{subp 2}\\
	&M^{k+1}=&\arg\min_{M} L_{\beta_1,\beta_2}\left(U^{k+1},Z^{k+1},M,Q^k,O^k\right), 
	\label{subp 3}\\
	&Q^{k+1}=&Q^{k}+\beta_1 \left(DU^{k+1} - M^{k+1}\right), \label{subp 4}\\
	&O^{k+1}=&O^{k}+\beta_2 \left(U^{k+1} - Z^{k+1}\right). \label{subp 5}
\end{eqnarray}
The remainder of this section discusses the solution of these subproblems.

\subsection{Solution of subproblem \eqref{subp 1}} 
We obtain from \eqref{subp 1} that 
\[
U^{k+1}=\arg\min_{U}\dfrac{\lambda}{2}\left\Vert U -B \right\Vert_F^2 + 
\langle DU-M^{k},Q^{k} \rangle +\dfrac{\beta_1}{2}\left\Vert DU-M^{k} \right\Vert_F^2
+ \langle Z^{k}-U,O^{k} \rangle +\dfrac{\beta_2}{2}\left\Vert Z^{k}-U \right\Vert_F^2,
\]
which is equivalent to 
\[
U^{k+1}=\arg\min_{U}\dfrac{\lambda}{2}\left\Vert U -B \right\Vert_F^2 
+\dfrac{\beta_1}{2}\left\Vert DU-M^{k} +\dfrac{Q^{k}}{\beta_1} \right\Vert_F^2
+ \dfrac{\beta_2}{2}\left\Vert Z^{k}-U  +\dfrac{O^{k}}{\beta_2} \right\Vert_F^2.
\]
The solution satisfies 
\begin{equation}\label{u1}
	\lambda U + \beta_1 D^TDU = \lambda B + \beta_1 D^TM^{k} -D^TQ^{k} +\beta_2 Z^{k} -O^{k},
\end{equation}
where $D^T$ is the adjoint of the gradient operator, i.e., the divergence operator. Let
\[
R^{k}=\lambda B + \beta_1 D^TM^{k} -D^TQ^{k} +\beta_2 Z^{k} -O^{k}.
\]
We then can write equation \eqref{u1} as
\[
\lambda U + \beta_1 D_1^TD_1 U + \beta_1 D_2^TD_2U=R^k,
\]
which is equivalent to 
\[
\lambda U + \beta_1  UC_1C_1^T + \beta_1 UC_2C_2^T=R^k.
\]
Since the matrices $C_1$ and $C_2$ are circulants, they are diagonalizable by Discrete
Fourier Transform (DFT) matrices, i.e.,
\[
C_1=F_1^H \Lambda_1 F_1,\quad C_2=F_2^H \Lambda_2 F_2,
\]
where $F_1$ and $F_2$ denote DFT matrices of sizes $n_2\times n_2$ and $n_1\times n_1$,
respectively, and the superscript $^H$ stands for transposition and complex conjugation. 
The diagonal entries of the diagonal matrices $\Lambda_1$ and $\Lambda_2$ are the 
eigenvalues of $C_1$ and $C_2$, respectively. Let $\otimes$ denote the Kronecker product 
and let the operator $\texttt{vec}$ stack all entries of a matrix column by column to
give a vector. We obtain
\[
\left(F_1^H \otimes F_2^H\right) \left(\lambda I \otimes I + \beta_1 \Lambda_1 \otimes I 
+\beta_1 I \otimes \Lambda_2\right)\left(F_1 \otimes F_2\right)\texttt{vec}
\left( U \right)=\texttt{vec}\left( R^k \right).
\]
The parameters $\lambda$ and $\beta_1$ are positive and the matrices $\Lambda_1$ and 
$\Lambda_2$ are invertible. Therefore, the matrix 
$\left(\lambda I\otimes I+\beta_1\Lambda_1\otimes I+\beta_1 I\otimes\Lambda_2\right)$ 
is invertible. It follows that $U$ can be expressed as
\begin{equation}\label{sol u1}
\texttt{vec}\left( U \right)=\left(F_1^H \otimes F_2^H\right) \left(\lambda I \otimes I + 
\beta_1 \Lambda_1 \otimes I +\beta_1 I \otimes \Lambda_2\right)^{-1}
\left(F_1 \otimes F_2\right)\texttt{vec}\left( R^k \right).
\end{equation}

\subsection{Solving subproblem \eqref{subp 2}} 
The matrix $Z^{k+1}$ satisfies
\[
Z^{k+1}=\arg\min_{Z}\left\{\left\Vert Z\right\Vert_* + \langle Z-U^{k+1},O^{k} \rangle +
\dfrac{\beta_2}{2}\left\Vert Z-U^{k+1} \right\Vert_F^2\right\},
\]
which can be written as
\[
Z^{k+1}=\arg\min_{Z} \left\{\left\Vert Z\right\Vert_* + \dfrac{\beta_2}{2}\left\Vert Z-
U^{k+1} +\dfrac{O^{k}}{\beta_2} \right\Vert_F^2\right\}.
\]
The right-hand side defines a proximal operator of the matrix nuclear norm. It can be 
shown that since $\beta_2>0$, the iterates $Z^{k+1}$ converge to a unique solution 
$Z^\infty$ as $k\to\infty$; see \cite{nuclear_norm_prox}. The iterate $Z^{k+1}$ can be 
expressed with the aid of Singular Value Thresholding (SVT) of the matrix 
\begin{equation}\label{Ushft}
	U^{k+1} -\dfrac{O^{k}}{\beta_1}.
\end{equation}
Let $USV^T$ denotes the singular value decomposition of the matrix \eqref{Ushft}. Thus, $U$
and $V$ are orthogonal matrices and $S$ is a diagonal matrix, whose nontrivial entries are 
the singular values. Then
\begin{equation}\label{sol z1}
	Z^{k+1}=U S_{\frac{1}{\beta_2}}  V^T,
\end{equation}
where, elementwise, $S_{\frac{1}{\beta_2}}=\max\left(S-\dfrac{1}{\beta_2},0\right)$.

\subsection{Solving subproblem \eqref{subp 3}}
Let $M=[M_{ij}]\in\mathbb{R}^{n_1\times n_2}$. Then this subproblem can be written as 
\begin{equation*}
M^{k+1}=\arg\min_{M} \left\{\sum_{i=1}^{n_1}\sum_{j=1}^{n_2} \phi\left( \left\Vert M_{ij}
\right \Vert_2; T,a \right) + \left\Vert Z^{k+1}\right\Vert_* + 
\langle DU^{k+1}-M,Q^{k} \rangle 
+ \dfrac{\beta_1}{2}\left\Vert DU^{k+1}-M\right\Vert_F^2\right\}.
\end{equation*}
This expression can can be simplified to
\begin{equation}\label{sub2.5}
M^{k+1}=\arg\min_{M}\left\{ \sum_{i=1}^{n_1}\sum_{j=1}^{n_2} \phi\left( \left\Vert M_{ij}
\right \Vert_2; T,a \right)+\dfrac{\beta_1}{2}\left\Vert DU^{k+1}-M+\dfrac{Q^k}{\beta_1} 
\right\Vert_F^2\right\}.
\end{equation}
Following the discussion by Chan et al. \cite{CNC}, we find that the minimizer is given by 
\begin{equation}\label{sol M1}
	M_{ij}=\zeta_{ij}R_{ij},
\end{equation}
where 
\[
R_{ij}=\left[\left(D_1U^{k+1}\right)_{ij}-\dfrac{Q^{k,(1)}_{i,j}}{\beta_1}, 
\left(D_2U^{k+1}\right)_{ij}-\dfrac{Q^{k,(2)}_{i,j}}{\beta_1}\right]
\]
and 
\[
\zeta_{ij}=\left\lbrace\begin{array}{ll}
	\kappa_1, &~~~~\text{if } \left\Vert R_{ij}\right\Vert_2\in [0,\kappa_0), \\
	\kappa_2 -\dfrac{\kappa_3}{\left\Vert R_{ij}\right\Vert_2}, & ~~~~\text{if } 
	\left\Vert R_{ij}\right\Vert_2\in [\kappa_0,T_2), \\
	1, & ~~~~\text{if } \left\Vert R_{ij}\right\Vert_2\in [T_2,+\infty)
\end{array}\right.
\]
with 
\[
\kappa_0=T+\dfrac{a}{\beta_1}(T_2-T),\quad \kappa_1=\dfrac{T}{\kappa_0},\quad 
\kappa_2=\dfrac{\beta_1}{\beta_1-a},\quad \kappa_3=\dfrac{aT_2}{\beta_1-a}.
\]
Chan et al. \cite{CNC} show that the cost function for the minimization problem 
\eqref{sub2.5} is strongly convex (convex) if and only if $\beta_1>a$ ($\beta_1\geq a$). 

We are in a position to discuss the solution of the minimization problem \eqref{main4}.
The following algorithm outlines the solution process. 

\begin{algorithm}[H]
\caption{Low-rank convex-non-convex (LR-CNC) algorithm}\label{Alg 1}
\textbf{Input:} $B\in \mathbb{R}^{n_1\times n_2}$ observed data.\\
\textbf{Output:} $U\in \mathbb{R}^{n_1\times n_2}$ recovered data.\\
\textbf{Initialize:} $\lambda$, $\beta_1$, $\beta_2$, $a$, $T_2$.
\begin{algorithmic}[1]
	\While{not converged}
	\State Update $U$ from \eqref{sol u1}.
	\State Update $Z$ from \eqref{sol z1}.
\State Update $M$ from \eqref{sol M1}.
\State Update $Q$ and $Q$ from \eqref{subp 4} and \eqref{subp 5}, respectively.
\EndWhile
\end{algorithmic}
\end{algorithm}

\noindent ADMM is guaranteed to converge to a global minimum of the convex problem \eqref{main5}; in
our non-convex formulation, we show in Section \ref{sec 3} that every limit point of the 
sequence generated by Algorithm~\ref{Alg 1} is a stationary point of \eqref{main5}.
\\
A numerically practical way to compute the solution $U$ is to first apply the MATLAB 
function \texttt{psf2otf} to the matrix 
\[
L=\frac{1}{4}\begin{bmatrix}
	0& -1& 0\\
	-1& 4& -1\\
	0& -1& 0
\end{bmatrix}.
\]
This MATLAB function is used to compute the fast Fourier transform of the matrix $L$ 
(considered as a Point-Spread Function (PSF) array) and creates the Optical Transfer
Function (OTF) array. Using MATLAB notation, $U$ can be computed as 
\[
U=\texttt{ifft2}( \texttt{fft2}(R^k)./(1+((\beta_1/\lambda)*\texttt{psf2otf}(L)))).
\]
Generally, the matrix $M$ is easy to compute. The main computational work required by 
Algorithm \ref{Alg 1} is the computation of the matrix $Z$ as this requires the evaluation
of a singular value decomposition. The matrices in the computed examples of Section 
\ref{sec 4} are small enough to make this feasible. For very large matrices, we can 
compute a partial singular value decomposition that is made up of all singular triplets 
with singular values larger than $1/\beta_2$. These singular triples can be computed, 
e.g., with the MATLAB function \texttt{svds}, which implements the method described in 
\cite{BR}.

\section{Convergence analysis}\label{sec 3}
This section investigates the convergence of Algorithm~\ref{Alg 1}. We establish that the
sequence of the generated iterates has limit points and any accumulation point is a 
solution (or stationary point) of the minimization problem. Let 
\begin{equation}\label{Jfun}
	J(U;\lambda,a,T)=F(U)+L(U) + R(U),
\end{equation}
where 
\[
F(U)=\dfrac{\lambda}{2}\left\Vert U-B \right\Vert_F^2,\quad  
L(U)=\left\Vert U\right\Vert_*, \quad R(U)=
\sum_{i=1}^{n_1}\sum_{j=1}^{n_2}\phi\left(\left\Vert \left(\nabla U\right)_{ij} 
\right\Vert_2;T,a\right).
\]
The functionals $F$, $L$, and $R$ are referred to as the fidelity, low-rank, and
regularization functionals, respectively. Consider the saddle point problem of determining
a quintuple $\left( U^*, Z^*, W^*, Q^*, O^* \right)$ such that
\begin{equation}
	\label{saddle}
	\begin{split}
		L_{\beta_1,\beta_2}\left(U^*,Z^*,M^*,Q,O\right)\leq L_{\beta_1,\beta_2}
		\left(U^*,Z^*,M^*,Q^*,O^*\right)\leq L_{\beta_1,\beta_2}\left(U,Z,M,Q^*,O^*\right),\\ 
		\forall \left(U,Z,M,Q,O\right)\in \mathbb{R}^{n_1\times n_2} \times 
		\mathbb{R}^{n_1\times n_2} \times \mathbb{R}^{2n_1 \times n_2} \times 
		\mathbb{R}^{2n_1\times n_2}\times \mathbb{R}^{n_1 \times n_2}.
	\end{split}
\end{equation}
Chan et al. \cite[Lemma 3.1]{CNC} show the following result.

\begin{lemma}
	The functional defined by 
	\[
	\begin{split}
		J_1(\cdot;\lambda,T,a):\mathbb{R}^{n_1\times n_2}& \to \mathbb{R},\\
		U& \mapsto \dfrac{\lambda}{2}\left\Vert U-B \right\Vert_F^2 +
		\sum_{i=1}^{n_1}\sum_{j=1}^{n_2}\phi\left(\left\Vert \left(\nabla U\right)_{ij} 
		\right\Vert_2;T,a\right)
	\end{split}
	\]
	is strictly convex if and only if the function 
	$h(\cdot;\lambda,T,a):\mathbb{R}\to\mathbb{R}$ given by 
	\[
	h(t;\lambda,T,a)=\dfrac{\lambda}{18}t^2 + \phi(\left\vert t \right\vert;T,a)
	\]
	is strictly convex.
\end{lemma}

\begin{lemma}
	The functional $J(\cdot;\lambda,T,a):\mathbb{R}^{n_1\times n_2}\to \mathbb{R}$ is strictly 
	convex if and only if the function $J_1:\mathbb{R}^{n_1\times n_2}\to \mathbb{R}$ is 
	strictly convex.
\end{lemma}

\begin{proof}
	The functional $J(\cdot;\lambda,T,a)$ can be written as
	\begin{equation}\label{Jrep}
		J(U;\lambda,T,a)=J_1(U,;\lambda,T,a) + \left\Vert U\right\Vert_*.
	\end{equation}
	The nuclear norm function is convex. Therefore, it does not affect the strict convexity. 
	This shows the lemma.
\end{proof}
\vskip2pt

\begin{corollary}\label{cor 1}
	The functional $J$ is strictly convex if and only if $h$ is strictly convex.
\end{corollary}
\vskip2pt

The next theorem furnishes conditions that secure strict convexity of 
$J(\cdot;\lambda,T,a)$.
\vskip2pt

\begin{theorem}
	A sufficient condition for the functional $J(\cdot;\lambda,T,a)$ to be strictly convex is 
	that the pair $\left(\lambda,a\right)\in \mathbb{R}^*_+\times \mathbb{R}^{*}_+$ satisfies
	\begin{equation}\label{condition}
		\lambda>9a \quad \left( \text{equivalently}, \; \lambda=\tau\, 9a
		\text{~~for some constant~~} \tau>1 \right).
	\end{equation}
\end{theorem}

\begin{proof}
	The theorem can be shown similarly as \cite[Theorem 3.5]{CNC} by using Corollary
	\ref{cor 1}. 
\end{proof}
\vskip2pt

A function $G:\mathbb{R}^n\to\mathbb{R}$ is said to be $\mu$-strongly convex if there is
a constant $\mu>0$ such that $G(x)-\frac{\mu}{2}\|x\|_2^2$ is convex.
\vskip2pt

\begin{proposition}
	The functional $J(\cdot;\lambda,T,a)$ is proper, continuous, bounded from below by zero, 
	coercive, and $\mu$-strongly convex, where 
	\begin{equation}\label{muval}
		\mu=\lambda -9a.
	\end{equation}
\end{proposition}
\vskip2pt

\begin{proof}
	Chan et al. \cite[Proposition 3.7]{CNC} show that the functional $J_1(U;\lambda,T,a)$ is
	proper, continuous, bounded from below by zero, coercive, and $\mu$-strongly convex, with 
	$\mu$ given by \eqref{muval}. The nuclear norm function satisfies
	\begin{equation}\label{normineq}
		\left\Vert U \right\Vert_F \leq \left\Vert U\right\Vert_*.
	\end{equation}
	The representations \eqref{Jrep} and \eqref{normineq} show that the functional $J$ is 
	continuous, bounded below by zero, and strongly $\mu$-convex. 
\end{proof}
\vskip2pt

\begin{lemma}\label{liplemma}
	Let the pair of parameters $\left(\lambda,a\right)$ satisfy  \eqref{condition}. Then the 
	functional $J$, the regularization term $R$, the quadratic fidelity term $F$, and the
	functional $L$ in \eqref{Jfun} are locally Lipschitz continuous functions. 
\end{lemma}
\vskip2pt

\begin{proof}
	Chan et al. \cite{CNC} show that the functions $R$ and $F$ are locally Lipschitz continuous. 
	The lemma follows from the observation that the function $L$ also is locally Lipschitz 
	continuous.
\end{proof}
\vskip2pt

\begin{proposition}
	For any pair of parameters $\left(\lambda,a\right)$ that satisfies \eqref{condition}, the 
	functional $J:\mathbb{R}^{n_1\times n_2}\to\mathbb{R}$ has a unique (global) minimizer 
	$U^*$ that satisfies
	\begin{equation}\label{rel1}
		0\in \partial_u [J](U^*),
	\end{equation}
	where $0$ denotes the null matrix of size $n_1\times n_2$ and 
	$\partial_u [J](U^*)\subset \mathbb{R}^{n_1\times n_2}$ stands for the subdifferential of 
	$J$. Moreover, 
	\[
	0\in D^T\bar{\partial}_u[R](DU^*) + \lambda\left( U^* -B \right) + 
	\partial_U [\left\Vert U \right\Vert_*]\left(U^{*}\right),
	\]
	where $\bar{\partial}_u[R](DU^*)$ denotes the Clarke generalized gradient \cite{Clarke} of
	the (non-convex and non-smooth) regularization function $R$, and 
	$\partial_U [\left\Vert U \right\Vert_*]\left(U^{*}\right)$ is the subdifferential of the 
	nuclear norm function at $U^*$.
\end{proposition}
\vskip2pt

\begin{proof}
	The functional $J$ is strongly convex when \eqref{condition} holds. Therefore, $J$ has a 
	unique minimizer $U^*$ and \eqref{rel1} follows from the generalized Fermat's rule 
	\cite{Fermat}. 
	
	The concept of a generalized gradient extends the concept of a subdifferential for 
	non-smooth convex functions to non-smooth non-convex functions that are locally Lipschitz.
	According to Lemma \ref{liplemma}, the functionals $J$, $\phi$, $L$, and $F$ are locally 
	Lipschitz. Therefore, the generalized gradient is defined.
	
	For non-smooth and convex functions, the Clarke generalized gradient equals the 
	subdifferential, i.e.,
	\[
	\partial_U [J](U^*)=\bar{\partial}_U [J](U^*).
	\]
	Hence, 
	\[
	\bar{\partial}_U [J](U^*)\subset \bar{\partial}_U [\phi](U^*) +\bar{\partial}_U 
	[\dfrac{\lambda}{2}\left\Vert U-B \right\Vert_F^2](U^*)+\bar{\partial}_U 
	[\left\Vert U-\right\Vert_*](U^*),
	\]
	which gives
	\[
	\bar{\partial}_U [J](U^*)\subset \bar{\partial}_U [\phi](U^*) +
	\bar{\partial}_U [\lambda\left( U-B \right)](U^*)+\partial_U 
	[\left\Vert U-\right\Vert_*](U^*),
	\]
	because in the case of a continuously differentiable function, the generalized gradient 
	reduces to the gradient.
\end{proof}
\vskip2pt

\begin{proposition}
	Let $F:\mathbb{R}^{n_1\times n_2}\to\mathbb{R}$ be the fidelity function, and let
	$B\in\mathbb{R}^{n_1\times n_2}$, $Q\in\mathbb{R}^{2n_1\times n_2}$, 
	$O\in \mathbb{R}^{n_1\times n_2}$, $\lambda\in\mathbb{R}^{*}_+$, $\gamma_1\in \mathbb{R}$, 
	and $\gamma_2\in \mathbb{R}$. Then the function
	\[
	F(U)- \dfrac{\gamma_1}{2}\left\Vert DU-Q \right\Vert_F^2 -\dfrac{\gamma_2}{2}
	\left\Vert U-O \right\Vert_F^2
	\]
	is convex if 
	\begin{equation}\label{convineq}
		\gamma_1\leq \dfrac{\lambda-\gamma_2}{8}.
	\end{equation}
	This inequality also is necessary if the inequality is to hold for all $n_1,n_2\geq 1$.
\end{proposition}
\vskip2pt

\begin{proof}
	In order for the function $F$ to be convex, its Hessian $H$ has to be positive 
	semidefinite. Let $n=\min\{n_1, n_2\}$. The Hessian is given by 
	\[
	H=\lambda I_n -\gamma_1 D^TDI_{n} -\gamma_2 I_n,
	\]
	where $I_n$ denotes the identity matrix of order $n$. The matrix $H$ can be written as
	\[
	H=\lambda I_{n} -\gamma_1 C_1C_1^T - \gamma_1 C_2^TC_2 -\gamma_2 I_{n},
	\]
	where $C_1,C_2\in\mathbb{R}^{n\times n}$. It can be shown by a simple calculation that 
	$C_1C_1^T=C_2^TC_2$. Let $A=C_1C_1^T=C_2^TC_2$. Then the Hessian can be expressed as
	\[
	H=\lambda I_{n} -2\gamma_1 A -\gamma_2 I_{n}.
	\]
	Since the matrix $A$ is symmetric and positive semidefinite, it has the spectral 
	factorization 
	\[
	A=V\Lambda V^T,\qquad \Lambda={\rm diag}[\lambda_1(A),\ldots,\lambda_{n}(A)],
	\qquad VV^T=I,
	\]
	with eigenvalues $\lambda_{i}(A)\geq 0$ and orthogonal eigenvector matrix $V$. In order 
	for 
	\[
	H=V(\left(\lambda-\gamma_2\right)I_n - 2\gamma_1 \Lambda)V^T
	\]
	to be positive semidefinite, $\lambda$, $\gamma_1$, and $\gamma_2$ must satisfy
	\[
	\lambda-\gamma_2 - 2\gamma_1 \max_{i=1,2,\ldots,n}
	\left(\lambda_i\left(A\right)\right)\geq 0.
	\]
	Since 
	\[
	A=\begin{bmatrix}
		2 &-1& 0&\ldots & 0 & -1\\
		-1 &2& -1& 0& \ldots & 0\\
		0 & \ddots & \ddots & \ddots & \ldots & 0\\
		\ldots & \ddots & \ddots & \ddots & -1& 0\\
		\ldots & \ldots & \ddots &-1& 2& -1\\
		-1 & 0& \ldots &0& -1& 2
	\end{bmatrix}\in\mathbb{R}^{n\times n},
	\]
	we obtain by using Gershgorin discs that $\lambda_i(A)\in [0,4]$ for all $i$. It follows 
	that convexity is secured if \eqref{convineq} holds. The largest eigenvalue converges to
	$4$ as $n$ increases. Therefore \eqref{convineq} is necessary if we would like to 
	determine $\gamma_1$ independently of $n_1$ and $n_2$.
\end{proof}
\vskip2pt

\begin{proposition}\label{lagmul}
	For any Lagrangian multipliers $Q$ and $O$, we have that the augmented Lagrangian 
	functional $L_{\beta_1,\beta_2}(U,Z,M,Q,O)$ is proper, continuous, and coercive jointly in
	the primal variables $(U,Z,M)$. Moreover, 
	$L_{\beta_1,\beta_2}(U,Z,M,Q,O)$ is jointly convex in $(U,Z,M)$ if the penalty parameters 
	$\beta_1$ and $\beta_2$ satisfy 
	\begin{equation}\label{condition beta}
		\beta_1=\dfrac{\gamma_1\gamma_3}{\gamma_1 -\gamma_3},\quad
		\beta_2=\dfrac{\gamma_2\gamma_4}{\gamma_2 -\gamma_4},
	\end{equation}
	where
	\begin{equation}\label{condition gamma}
		(\gamma_1,\gamma_2,\gamma_3,\gamma_4)\in\mathbb{T}=\left\{ 
		(\gamma_1,\gamma_2,\gamma_3,\gamma_4)\in \mathbb{R}:\,\gamma_1\geq\gamma_3,\, \gamma_1\leq 
		\dfrac{\lambda-\gamma_2}{8},\, \gamma_2\geq\gamma_4\geq 0,\, \gamma_3\geq a \right\}.
	\end{equation}
	The relations \eqref{condition beta} are meaningful if $\gamma_1>\gamma_3$ and 
	$\gamma_2>\gamma_4$.
\end{proposition}
\vskip2pt

\begin{proof}
	It is clear that the function $L_{\beta_1,\beta_2}$ is proper and continuous. To show that
	$L_{\beta_1,\beta_2}$ is coercive jointly in $(U,Z,M)$, we note that
	\begin{eqnarray*}
		L_{\beta_1,\beta_2}(U,Z,M,Q,O)&=&F(U) + R(M) + \left\Vert Z\right\Vert_* + 
		\langle DU-M,Q \rangle + \dfrac{\beta_1}{2}\left\Vert DU-M\right\Vert_F^2 + 
		\langle U-Z,O \rangle \\
		&+& \dfrac{\beta_2}{2}\left\Vert U-Z\right\Vert_F^2\\
		&=& F(U) + R(M) + \left\Vert Z\right\Vert_* + \dfrac{\beta_1}{2}
		\left\Vert DU-M +\dfrac{Q}{\beta_1}\right\Vert_F^2 - 
		\dfrac{1}{2\beta_1}\left\Vert Q\right\Vert_F^2 \\
		&+& \dfrac{\beta_2}{2}\left\Vert U-Z+ \dfrac{O}{\beta_2}\right\Vert_F^2 - 
		\dfrac{1}{2\beta_2}\left\Vert O\right\Vert_F^2.
	\end{eqnarray*}
	Since $R(M)$ is bounded and the terms $\dfrac{1}{2\beta_2}\left\Vert O\right\Vert_F^2$ and 
	$\dfrac{1}{2\beta_1}\left\Vert Q\right\Vert_F^2$ are independent of $(U,Z,M)$, it follows
	that $R(M)$ and these terms do not affect the coercivity. The expression
	\[
	F(U)+\dfrac{\beta_1}{2}\left\Vert DU-W +\dfrac{Q}{\beta_1}\right\Vert_F^2 + 
	\dfrac{\beta_2}{2}\left\Vert U-Z+ \dfrac{O}{\beta_2}\right\Vert_F^2 
	\]
	is coercive with respect to $U$, the expression 
	\[
	\left\Vert Z\right\Vert_*+\dfrac{\beta_2}{2}\left\Vert U-Z+ 
	\dfrac{O}{\beta_2}\right\Vert_F^2 
	\]
	is coercive with respect to $Z$, and the expression
	\[
	\dfrac{\beta_1}{2}\left\Vert DU-W +\dfrac{Q}{\beta_1}\right\Vert_F^2 
	\]
	is coercive with respect to $W$. Therefore, $L_{\beta_1,\beta_2}(U,Z,M,Q,O)$ is jointly 
	coercive with respect to $(U,Z,M)$. 
	
	To show convexity, we express $L$ as
	\begin{eqnarray*}
		L(U,Z,M,Q,O)&=&F(U)-\dfrac{\gamma_1}{2}\left\Vert DU \right\Vert_F^2 - 
		\dfrac{\gamma_2}{2}\left\Vert U \right\Vert_F^2\\ 
		&+& R(M) + \dfrac{\gamma_3}{2}\left\Vert M\right\Vert_F^2\\
		&+& \left\Vert Z\right\Vert_* + \dfrac{\gamma_4}{2}\left\Vert Z\right\Vert_F^2\\
		&+& \langle DU-M,Q \rangle + \langle U-Z,O \rangle  \\
		&+&   \dfrac{\beta_2}{2}\left\Vert U-Z\right\Vert_F^2+\dfrac{\beta_1}{2}
		\left\Vert DU-M\right\Vert_F^2+ \dfrac{\gamma_1}{2}\left\Vert DU \right\Vert_F^2 + 
		\dfrac{\gamma_2}{2}\left\Vert U \right\Vert_F^2 -\dfrac{\gamma_3}{2}
		\left\Vert M\right\Vert_F^2-\dfrac{\gamma_4}{2}\left\Vert Z\right\Vert_F^2.
	\end{eqnarray*}
	Define the functions
	\[
	\begin{split}
		&L_1(U)=F(U)-\dfrac{\gamma_1}{2}\left\Vert DU \right\Vert_F^2 - 
		\dfrac{\gamma_2}{2}\left\Vert U \right\Vert_F^2,\\
		&L_2(M)=R(M) + \dfrac{\gamma_3}{2}\left\Vert M\right\Vert_F^2,\\
		&L_3(Z)=\left\Vert Z\right\Vert_* + \dfrac{\gamma_4}{2}\left\Vert Z\right\Vert_F^2,\\
		&L_4(U,M,Z)=\langle DU-M,Q \rangle + \langle U-Z,O \rangle,\\
		&L_5(U,Z,M)=\dfrac{\beta_1}{2}\left\Vert DU-M\right\Vert_F^2+ 
		\dfrac{\beta_2}{2}\left\Vert U-Z\right\Vert_F^2+
		\dfrac{\gamma_1}{2}\left\Vert DU \right\Vert_F^2 + 
		\dfrac{\gamma_2}{2}\left\Vert U \right\Vert_F^2-
		\dfrac{\gamma_3}{2}\left\Vert M\right\Vert_F^2-
		\dfrac{\gamma_4}{2}\left\Vert Z\right\Vert_F^2
	\end{split}
	\]
	and note that $L_5$ can be written as
	\[
	L_5(U,Z,M)=\dfrac{\beta_1 +\gamma_1}{2}\left\Vert DU\right\Vert_F^2 -
	\beta_1\langle DU,M\rangle + \dfrac{\beta_1-\gamma_3}{2}\left\Vert M \right\Vert_F^2 +
	\dfrac{\beta_2+\gamma_2}{2}\left\Vert U \right\Vert_F^2 + 
	\dfrac{\beta_2-\gamma_4}{2}\left\Vert Z\right\Vert_F^2 - \beta_2\langle U,Z\rangle.
	\]
	The functional $U\to L_1(U)$ is convex if
	\[
	\gamma_1 \leq \dfrac{\lambda-\gamma_2}{8}
	\]
	and $L_2$ is convex if the inequality
	\begin{equation}\label{abd}
		\gamma_3\geq a
	\end{equation}
	holds. Since the nuclear norm $\left\Vert Z\right\Vert_*$ is convex, the functional $L_3$
	is convex when $\gamma_4\geq 0$. In view of that the functional $L_4$ is affine with 
	respect to $U$, $M$, and $Z$, it does not affect the convexity of $L$. Finally, the 
	functional $L_5$ is jointly convex if it can be reduced to the form
	\[
	\dfrac{1}{2}\left\Vert c_1 DU -c_2 M\right\Vert_F^2 + 
	\dfrac{1}{2}\left\Vert c_3 U -c_4 Z\right\Vert_F^2
	\mbox{~~~for some coefficients~~~} c_1,c_2,c_3,c_4\geq 0.
	\]
	Thus, we require the coefficients of $\left\Vert DU\right\Vert_F^2$,
	$\left\Vert M \right\Vert_F^2$, $\left\Vert U \right\Vert_F^2$, and 
	$\left\Vert Z\right\Vert_F^2$ be positive, that the product of the square roots of 
	$c_1$ and $c_2$ equal the coefficient of $-\langle DU,M \rangle$, and that the product of 
	the square roots of $c_3$ and $c_4$ equal the coefficient of $-\langle U,Z \rangle$. This 
	implies that 
	\[
	\beta_1\geq-\gamma_1,~~\beta_1\geq\gamma_3,~~\beta_2\geq-\gamma_2,
	~~\beta_2\geq\gamma_4\geq 0,~~ \gamma_1\gamma_3=\beta_1(\gamma_1-\gamma_3),~~
	\gamma_2\gamma_4=\beta_2(\gamma_2-\gamma_4).
	\]
	Thus, $\beta_1$ and $\beta_2$ should satisfy
	\[
	\beta_1=\dfrac{\gamma_1\gamma_3}{\gamma_1 -\gamma_3},\qquad 
	\beta_2=\dfrac{\gamma_2\gamma_4}{\gamma_2 -\gamma_4},
	\]
	and the quadruple $(\gamma_1,\gamma_2,\gamma_3,\gamma_4)$ should satisfy 
	\eqref{condition gamma}. Moreover, the set $\mathbb{T}$ is non-empty.
\end{proof}
\vskip2pt

Proposition \ref{lagmul} shows that we can determine parameters $\beta_1$ and $\beta_2$ 
such that the functional $L_{\beta_1,\beta_2}$ is jointly convex in $(U,Z,M)$. Since 
\[
\gamma_1\geq\gamma_3,~~\gamma_2\geq\gamma_4,~~\gamma_1,\gamma_2,\gamma_3,\gamma_4>0,
\]
it follows from \eqref{condition beta} and \eqref{condition gamma} that there are 
parameters $\rho_1,\rho_2\geq 1$ such that 
\[
\gamma_1=\rho_1\gamma_3,\quad \gamma_2=\rho_2\gamma_4.
\]
Assume that $\rho_1,\rho_2>1$. Then condition \eqref{abd} yields
\[
\beta_1\geq\dfrac{a}{\rho_1-1}
\]
and 
\[
a\leq\gamma_1\leq\dfrac{\lambda-\gamma_2}{8}~~\Leftrightarrow~~\gamma_2\leq \lambda-8a.
\]
Thus,
\[
0< \beta_2\leq  \dfrac{\gamma_4\left(\lambda-8a \right)}{\gamma_2-\gamma_4}=
\dfrac{\rho_2\left(\lambda-8a \right)}{\rho_2-1}.
\]
Therefore, $\beta_1$ and $\beta_2$ satisfy the inequalities 
\begin{equation}\label{cond 3.9}
	\beta_1\geq\dfrac{a}{\rho_1-1},\quad 0< \beta_2\leq  
	\dfrac{\rho_2\left(\lambda-8a \right)}{\rho_2-1}.
\end{equation}
\vskip2pt

\begin{lemma}\label{lastlemma}
	Assume that $f=g+h$, where $g$ and $h$ are lower semi-continuous convex functions from 
	$\mathbb{R}^{n_1 \times n_2 }$ to $\mathbb{R}$, and $h$ is Gateau differentiable with 
	derivative $h'$. If $p^* \in \mathbb{R}^{n_1 \times n_2}$, then the following conditions 
	are equivalent:
	\begin{enumerate}
		\item $p^*\in \arg\inf_{p\in \mathbb{R}^{n_1 \times n_2}} f(p)$,
		\item $g(p)-g(p^*)+\langle h'(p^*), p-p^* \rangle\geq 0$, $\forall p \in 
		\mathbb{R}^{n_1 \times n_2}$.
	\end{enumerate}
	Moreover, when the function $g$ has a separable structure of the type 
	\[
	g(p)=g(p_1, p_2, p_3)=Q_1(p_1)+Q_2(p_2)+Q_3(p_3),
	\]
	where $p_1,\,p_2,\,p_3$ are independent variables, conditions $1.$ and $2.$ are 
	equivalent to 
	\[
	\left\lbrace
	\begin{array}{rcl}
		p_1^{*}&=&\arg\min_{p_1} f(p_1,p_2^*,p_3^*), \\
		p_2^{*}&=&\arg\min_{p_2} f(p_1^*,p_2,p_3^*), \\
		p_3^{*}&=&\arg\min_{p_3} f(p_1^*,p_2^*,p_3).
	\end{array}\right.
	\]
\end{lemma}
\vskip2pt

\begin{proof}
	The proposition can be shown similarly as in \cite{CNC}.
\end{proof}
\vskip2pt

\begin{theorem}
	For any pair of parameters $(\lambda,a)$ that satisfies condition \eqref{condition}, any 
	penalty parameters $\beta_1$ and $\beta_2$ that satisfy \eqref{condition beta}, and any
	Lagrange multipliers $Q$ and $O$, the augmented Lagrangian functional $L$ satisfies
	\begin{equation}
		(U^*,Z^*,M^*)\in \arg\min_{U,Z,M} L_{\beta_1,\beta_2}(U,Z,M,Q,O) \Leftrightarrow 
		\left\lbrace\begin{array}{cc}
			U^{*}=\arg\min_{U} L(U,Z^*,M^*,Q,O), &  \\
			Z^{*}=\arg\min_{Z} L(U^*,Z,M^*,Q,O), &   \\
			M^{*}=\arg\min_{M} L(U^*,Z^*,M,Q,O).
		\end{array}\right.
	\end{equation}
The saddle point problem \eqref{saddle} has at least one solution, and all solutions are 
of the form $(U^*,U^*,DU^*,Q^*,O^*)$, where $U^*$ is the unique global minimizer of the 
functional $J$.
\end{theorem}
\vskip2pt

\begin{proof}
	It follows from the proof of Proposition \ref{lagmul} that there is a quadruple 
	$(\gamma_1, \gamma_2, \gamma_3, \gamma_4)$ such that the functional
	\begin{eqnarray*}
		L_{\beta_1,\beta_2}\left(U,Z,M,Q,O\right)=F(U)-\dfrac{\gamma_1}{2}
		\left\Vert DU \right\Vert_F^2 -\dfrac{\gamma_2}{2}\left\Vert U\right\Vert_F^2 +
		R\left(M\right)+ \dfrac{\gamma_3}{2}\left\Vert M\right\Vert_F^2 + 
		\left\Vert Z\right\Vert_* + \dfrac{\gamma_4}{2}\left\Vert Z\right\Vert_F^2\\ + 
		\langle DU-M,Q \rangle + \langle U-Z,O \rangle + 
		\dfrac{1}{2}\left\Vert c_1 DU-c_2 M\right\Vert_F^2 + 
		\dfrac{1}{2}\left\Vert c_3 U -c_4 Z \right\Vert_F^2,
	\end{eqnarray*}
	is proper, continuous, coercive, and convex jointly in the variables $(U,Z,M)$. We can 
	express $L$ as
	\[
	L_{\beta_1,\beta_2}\left(U,Z,M,Q,O\right)=g(U,Z,M) +h(U,Z,M),
	\]
	where
	\[
	g(U,Z,M)=g_1(U)+g_2(Z)+g_3(M),\quad
	h(U,Z,M)=\dfrac{1}{2}\left\Vert c_1 DU-c_2 M\right\Vert_F^2 + 
	\dfrac{1}{2}\left\Vert c_3 U -c_4 Z \right\Vert_F^2
	\]
	with 
	\begin{eqnarray*}
		&g_1(U)&=F(U)-\dfrac{\gamma_1}{2}\left\Vert DU \right\Vert_F^2 -
		\dfrac{\gamma_2}{2}\left\Vert U\right\Vert_F^2+\langle DU,Q\rangle +\langle U,O\rangle,\\
		&g_2(Z)&=\left\Vert Z\right\Vert_*+\dfrac{\gamma_4}{2}\left\Vert Z\right\Vert_F^2-
		\langle Z,O \rangle,\\
		&g_3(M)&=R(M)+\dfrac{\gamma_3}{2}\left\Vert M\right\Vert_F^2 -\langle M,Q \rangle.
	\end{eqnarray*}
	The functions $g$ and $h$ are semi-continuous and convex. Moreover, $h$ is Gatteau 
	differentiable. Therefore, Lemma \ref{lastlemma} yields the desired result.
\end{proof}
\vskip2pt

\begin{theorem}\label{thm3.12}
	Let the parameters $\left(\lambda,a\right)$ satisfy \eqref{condition} and let the 
	Lagrangian parameters $(\beta_1,\beta_2)$ satisfy \eqref{condition beta}. Then the saddle 
	point problem \eqref{saddle} admits at least one solution, and all the solutions have the 
	form $(U^*,U^*, DU^*,Q^*,O^*)$, where $U^*$ denotes the unique minimizer of \eqref{main3}. 
	Furthermore, $DU^*=M^*$ and $U^*=Z^*$.
\end{theorem}
\vskip2pt

\begin{proof}
	The proof is similar to the proof presented in \cite{CNC}.  
\end{proof}
\vskip2pt

Let $(U^*,Z^*,M^*,Q^*,O^*)$ solve the saddle point problem \eqref{saddle}. Then
\begin{eqnarray*}
	U^*&=&\arg\min_{U}L_{\beta_1,\beta_2}\left(U,Z^*,M^*,Q^*,O^*\right)=
	L_{\beta_1,\beta_2|U}(U),\\
	Z^*&=&\arg\min_{Z}L_{\beta_1,\beta_2}\left(U^*,Z,M^*,Q^*,O^*\right)=
	L_{\beta_1,\beta_2|Z}(Z),\\
	M^*&=&\arg\min_{M}L_{\beta_1,\beta_2}\left(U^*,Z^*,M,Q^*,O^*\right)=
	L_{\beta_1,\beta_2|M}(M),
\end{eqnarray*}
where $L_{\beta_1,\beta_2|U}$, $L_{\beta_1,\beta_2|Z}$, and $L_{\beta_1,\beta_2|M}$ are 
restrictions of the functional $L_{\beta_1,\beta_2}$ to $U$, $Z$, and $M$, respectively.
These functionals can be written as 
\begin{eqnarray*}
	L_{\beta_1,\beta_2|U}\left(U\right)&=&\left[F(U)-
	\dfrac{\beta_1'}{2}\left\Vert DU-M^* \right\Vert_F^2 - 
	\dfrac{\beta_2'}{2}\left\Vert Z^* -U \right\Vert_F^2\right]\\
	&+& \left[\langle DU,Q^* \rangle + \langle U,O^* \rangle
	+ \dfrac{\beta_1+\beta_1'}{2} \left\Vert DU-M^* \right\Vert_F^2 + 
	\dfrac{\beta_2+\beta_2'}{2} \left\Vert Z^*-U \right\Vert_F^2\right],\\
	L_{\beta_1,\beta_2|Z}\left(Z\right)&=&\left[\left\Vert Z\right\Vert_* + 
	\dfrac{\beta_2''}{2}\left\Vert Z -U^* \right\Vert_F^2 \right] + 
	\left[-\langle Z,O^* \rangle + \dfrac{\beta_2-\beta_2''}{2} 
	\left\Vert Z-U^* \right\Vert_F^2\right], \\
	L_{\beta_1,\beta_2|M}\left(M\right)&=&\left[R\left(M\right) + 
	\dfrac{\beta_1''}{2}\left\Vert DU^*-M \right\Vert_F^2\right] +\left[-\langle M,Q^* \rangle
	+ \dfrac{\beta_1-\beta_1''}{2} \left\Vert DU^*-M\right\Vert_F^2\right],
\end{eqnarray*}

In order to apply Lemma \ref{lastlemma} to the functionals $L_{\beta_1,\beta_2|U}$, 
$L_{\beta_1,\beta_2|Z}$, and $L_{\beta_1,\beta_2|M}$, we have to verify that the first 
and second parts of each functional are semi-continuous and convex, and that the second 
part is Gateau differentiable. For this to hold, the parameters 
$\beta_1',\beta_1'',\beta_2',\beta_2''$ have to satisfy the conditions 
\begin{equation}\label{cond 3.10}
	-8\beta_1 -\beta_2 \leq 8\beta_1'+\beta_2' \leq \lambda,~~ 0\leq \beta_2'' \leq \beta_2,~~ a\leq \beta_1'' \leq \beta_1.
\end{equation}
We obtain
\begin{eqnarray}
	\label{condition optimality U}
	&&F(U)-F(U^*) -\dfrac{\beta_1'}{2}\left\Vert DU -M^*\right\Vert_F^2 + 
	\dfrac{\beta_1'}{2}\left\Vert DU^* -M^*\right\Vert_F^2 - \dfrac{\beta_2'}{2}\left\Vert Z^*
	-U \right\Vert_F^2 + \dfrac{\beta_2'}{2}\left\Vert Z^* -U^* \right\Vert_F^2  \\
	&&+ \langle D^TQ^* + O^* + (\beta_1+\beta_1')(D^TDU^* -D^TM^*) + (\beta_2 + \beta_2')
	(U^* -Z^*), U-U^*\rangle \geq 0,\; \forall U\in \mathbb{R}^{n_1 \times n_2}, \nonumber \\[2mm]
	\label{condition optimality Z}
	&&\left\Vert Z \right\Vert_* - \left\Vert Z^* \right\Vert_* + \dfrac{\beta_2''}{2}\left\Vert Z-U^* \right\Vert_F^2 - \dfrac{\beta_2''}{2}\left\Vert Z^*-U^* \right\Vert_F^2\\
	&&+ \langle -O^* +(\beta_2 -\beta_2'')\left(Z^*-U^*\right),Z-Z^*\rangle \geq 0,\; \forall Z\in \mathbb{R}^{n_1 \times n_2},\nonumber \\[2mm]
	\label{condition optimality M}
	& & R\left( M \right) - R\left( M^* \right) + \dfrac{\beta_1''}{2}\left\Vert DU^*-M \right\Vert_F^2 - \dfrac{\beta_1''}{2}\left\Vert DU^*-M^* \right\Vert_F^2\\
	&&+ \langle -Q^* +(\beta_1 -\beta_1'')\left(M^*-DU^*\right),M-M^*\rangle \geq 0,\; \forall U\in \mathbb{R}^{2n_1 \times n_2}.\nonumber
\end{eqnarray}
\vskip2pt

In the following, let $\left(U^*,Z^*,DU^*,Q^*,O^*\right)$ be a solution of the saddle 
point problem \eqref{saddle}. Introduce the variables
\[
\widetilde{U}^k=U^k -U^*,~~ \widetilde{Z}^k=Z^k -Z^*,~~ \widetilde{M}^k=M^k - M^*,~~ 
\widetilde{Q}^k=Q^k -Q^*,~~ \widetilde{O}^k=O^k -O^*, 
\]
where $\left(U^k,Z^k,DU^k,Q^k,O^k\right)$, $k=1,2,\ldots~$, is the sequence generated by 
Algorithm \ref{Alg 1}. The following propositions are important for showing convergence of
this sequence to the solution of the saddle point of \eqref{saddle}.
\vskip3pt
\begin{proposition}\label{penultprop}
	For any parameter pairs $(\lambda,a)$ that satisfies \eqref{condition}, and parameter 
	pairs $\left(\beta_1,\beta_2\right)$ such that
	\begin{equation}\label{eq beta1beta2}
		\beta_1>\max\{ \dfrac{a}{\rho_1-1}, \dfrac{2a\rho_1}{(\rho_1-1)^2} \},\;  0<\beta_2<\min\{ \dfrac{\rho_2(\lambda-8a)}{(\rho_2-1)}, \dfrac{2\rho_2(\lambda-8a)}{(\rho_2-1)^2}   \}  \text{ with } 
		\rho_1>1 \text{ and } \rho_2>3,
	\end{equation}
	we have 
	\begin{eqnarray}
		& & \dfrac{1}{2\beta_1} \left(\left\Vert \widetilde{Q}^{k} \right\Vert_F^2 - \left\Vert \widetilde{Q}^{k+1} \right\Vert_F^2\right)+\dfrac{1}{2\beta_2} \left(\left\Vert \widetilde{O}^{k} \right\Vert_F^2 - \left\Vert \widetilde{O}^{k+1} \right\Vert_F^2\right)\\
		&\geq&\left\Vert c_1 D\widetilde{U}^{k+1}-c_2\widetilde{M}^{k+1} \right\Vert_F^2+\left\Vert c_3 \widetilde{U}^{k+1}-c_4 \widetilde{Z}^{k+1} \right\Vert_F^2 -\beta_1\langle D\widetilde{U}^{k+1},\widetilde{M}^{k}-\widetilde{M}^{k+1} \rangle - \beta_2 \langle \widetilde{U}^{k+1},\widetilde{Z}^{k}-\widetilde{Z}^{k+1} \rangle,\nonumber\\
		&+& \dfrac{\beta_3}{2}\left\Vert D\widetilde{U}^{k+1}-\widetilde{M}^{k+1}\right\Vert_F^2 + \dfrac{\beta_4}{2}\left\Vert \widetilde{U}^{k+1}-\widetilde{Z}^{k+1}\right\Vert_F^2,\nonumber
	\end{eqnarray}
	for certain values of $\beta_3$ and $\beta_4$, and for specific values of $c_1,c_2,c_3,c_4$ 
	to be determined. 
\end{proposition}

\medskip

See Appendix \ref{app A} for a proof. 

\medskip

\begin{proposition}\label{lastprop}
	Let $\left(\beta_1,\beta_2\right)$ satisfy \eqref{eq beta1beta2} and let 
	$\left(\beta_3,\beta_4\right)$ and $\left(\beta_1'',\beta_2''\right)$ satisfy 
	\eqref{eq beta3beta4} and \eqref{eq beta''}, respectively. Assume that the coefficients
	$c_1,c_2,c_3,c_4$ satisfy \eqref{eq c}. Then 
	\begin{eqnarray*}
		& & \dfrac{1}{2\beta_1} \left(\left\Vert \widetilde{Q}^{k} \right\Vert_F - 
		\left\Vert \widetilde{Q}^{k+1} \right\Vert_F\right)+\dfrac{1}{2\beta_2} 
		\left(\left\Vert \widetilde{O}^{k} \right\Vert_F - \left\Vert \widetilde{O}^{k+1}
		\right\Vert_F\right)\\
		&\geq&\left\Vert c_1 D\widetilde{U}^{k+1}-c_2\widetilde{M}^{k+1} \right\Vert_F^2+
		\left\Vert c_3 \widetilde{U}^{k+1}-c_4 \widetilde{Z}^{k+1} \right\Vert_F^2 + 
		\dfrac{\beta_3}{2}\left\Vert D\widetilde{U}^{k+1}-\widetilde{M}^{k+1}\right\Vert_F^2 + 
		\dfrac{\beta_4}{2}\left\Vert \widetilde{U}^{k+1}-\widetilde{Z}^{k+1}\right\Vert_F^2 \\
		&+& \dfrac{\beta_1}{2}\left( -\left\Vert\widetilde{M}^{k} \right\Vert_F^2 + 
		\left\Vert \widetilde{M}^{k+1}\right\Vert_F^2\right) +(\dfrac{3\beta_1}{2}-\beta_1'')
		\left\Vert \widetilde{M}^{k+1}-\widetilde{M}^k\right\Vert_F^2 \nonumber\\
		&+& \dfrac{\beta_2}{2}\left( -\left\Vert\widetilde{Z}^{k} \right\Vert_F^2 + 
		\left\Vert \widetilde{Z}^{k+1}\right\Vert_F^2\right) + (\dfrac{3\beta_2}{2}-\beta_2'')
		\left\Vert \widetilde{Z}^{k+1}-\widetilde{Z}^{k} \right\Vert_F^2.  
	\end{eqnarray*}
\end{proposition}
\medskip

A proof is provided in Appendix \ref{app B}.

\medskip

The following theorem gives a condition on $\beta_1$ and $\beta_2$ so that the sequence 
generated by Algorithm \eqref{Alg 1} converges to the solution of the saddle point problem
\eqref{saddle}.

\begin{theorem}\label{thmcnv}
	Assume that $(U^*,Z^*,DU^*,Q^*,O^*)$ is a solution of a saddle point problem and let the
	parameter pair $(\lambda,a)$ satisfy \eqref{condition}. Then for any parameters such that 
	\eqref{eq beta1beta2} holds, the sequence\hfill\\ $\{(U^k,Z^k,M^k,Q^k,O^k)\}_{k=1}^{\infty}$ 
	generated by Algorithm \ref{Alg 1} satisfies
	\[
	\lim_{k\to \infty} U^{k}=U^*, \; \lim_{k\to \infty} Z^{k}=Z^*,\; 
	\lim_{k\to \infty} DU^{k}=DU^*=M^*,\, \lim_{k\to \infty} M^{k}=M^*.
	\]
\end{theorem}

\medskip

See Appendix \ref{app C} for a proof.

Theorem~\ref{thmcnv} implies that the iterates $(U^k,Z^k,M^k,Q^k,O^k)$ generated by 
Algorithm \ref{Alg 1} converge to a stationary solution of \eqref{main5} when $k$ 
increases. In other words, the proposed method is guaranteed to converge under the stated 
conditions. The following section illustrates the effectiveness of this algorithm when
applied to various image completion and segmentation tasks. 

\medskip
\section{Numerical examples}\label{sec 4}
We compare the LR-CNC algorithm (Algorithm~\ref{Alg 1}) to several available methods for 
image completion and segmentation, including the Convex-Non-Convex (CNC) segmentation 
method by Chan et al. \cite{CNC}, the ATCG-TV image restoration algorithm proposed by 
Benchettou et al. \cite{Benchatou} (a total variation-based image reconstruction method 
using an alternating conditional gradient scheme), and the classical Chan-Vese 
segmentation method \cite{T.Chan}. We also apply a standard K-means clustering \cite{Iko}
to determine segmentations of reconstructed images. 

The iterations with Algorithm~\ref{Alg 1} are terminated as soon as two
consecutive approximations $U^k$ are sufficiently close. Specifically, we terminate the
iterations when
\begin{equation}\label{reldiff}
	\dfrac{\left\Vert U^{k+1}-U^{k} \right\Vert_F}{\left\Vert U^{k}\right\Vert_F}\leq 
	\text{tol},
\end{equation}
where $\text{tol}$ is a user-chosen tolerance. In all the experiments, we set 
$\text{tol}=10^{-4}$. We report the Peak Signal-to-Noise Ratio (PSNR) 
\[
\text{PSNR}=10\log_{10} \dfrac{\max U(:)^2}{\left\Vert U-U_{reco}\right\Vert_F^2},
\]
where $U$ and $U_{reco}$ denote arrays that represent the original and the recovered data,
respectively, and $\max U(:)^2$ is the maximum squared entry (pixel-value) of the array 
$U$.

When segmenting images, we also compute the Structural Similarity Index Measure (SSIM) to 
determine the closeness of the recovered and the original image. The definition of this 
index is somewhat involved, see \cite{SSIM}, and we do not provide it here. We just recall
that a larger SSIM-value corresponds to a more accurate reconstruction and the maximum 
value is 1. All computations were carried out on a laptop computer equipped with a 2.3 GHz
Intel Core i5 processor and 8 GB of memory using MATLAB 2023.

For the numerical experiments, we need to fix several parameters: $a$, $\lambda$, $T$,
$\beta_1$, and $\beta_2$. Based on the equations \eqref{condition}, we let 
$\lambda=\tau_1 9a$, and \eqref{eq beta1beta2} shows that
\[
\beta_1=\tau_2\max\left\{ \dfrac{a}{\rho_1-1}, \dfrac{2a\rho_1}{(\rho_1-1)^2} \right\},
\quad \beta_2=\tau_3 \min\left\{ \dfrac{\rho_2(\lambda-8a)}{(\rho_2-1)},
\dfrac{2\rho_2(\lambda-8a)}{(\rho_2-1)^2}\right\}
\]
with $\rho_1>1$ and $\rho_2>3$. Thus, we only need to set the parameters $a$, $\tau_1$,
$\tau_2$, $\tau_3$, $\rho_1$, $\rho_2$, and $T$.

\subsection{Image completion}
This subsection compares the performance of Algorithm \ref{Alg 1} to that of the recently
proposed algorithms ATCG-TV by Benchettou et al. \cite{Benchatou} and CNC by Chan et al. 
\cite{CNC}. This subsection is divided into two parts: the first part is concerned with 
the gray scale images, while the second part considers color images.

In this subsection we use the sampling rate, which is defined as 
\[
\text{SR}=\dfrac{n_1n_2-p}{n_1n_2},
\]
where $n_1n_2$ is the number of the pixels of the data, and $p$ is the number of missed 
entries. For all the experiments in this subsection, we set $a=0.1$, $T=10^{-6}$, 
$\rho_1=2.5$, $\rho_2=3.001$, $\tau_1=\rho_1$, $\tau_2=\rho_1$, and $\tau_3=1.0001$.

\subsubsection{Gray scale images}
We illustrate the performance of Algorithm \ref{Alg 1} when applied to the well-known MRI
and cameraman images. Figure \ref{fig 1} shows results of completion for the MRI image 
determined by Algorithm \ref{Alg 1}, the CNC algorithm \cite{CNC}, and the ATCG-TV 
algorithm \cite{Benchatou} for $90\%$ missing data (SR$=0.1$). Table \ref{tab 1} reports 
PSNR-values, CPU-times (in seconds), and the number of iterations (Iter) required by each 
algorithm to satisfy the stopping criteria. The stopping criteria for the ATCG-TV and CNC
algorithms were the default choices provided in \cite{Benchatou} and \cite{CNC}, 
respectively.

\begin{figure}[H]
\centering
\begin{tabular}{ccccc}
\includegraphics[width=0.15\linewidth]{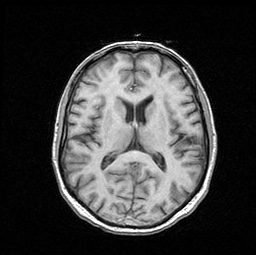}&
\includegraphics[width=0.15\linewidth]{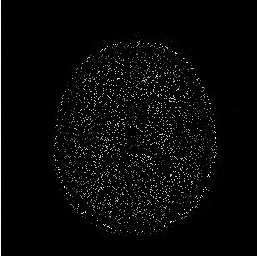}&
\includegraphics[width=0.15\linewidth]{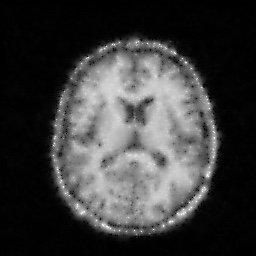}&
\includegraphics[width=0.15\linewidth]{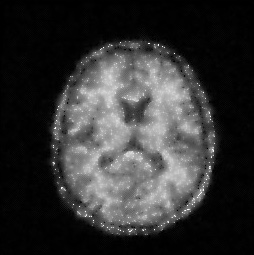}&
\includegraphics[width=0.15\linewidth]{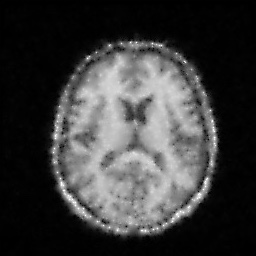}\\
\textbf{Original} &\textbf{Observed}  &\textbf{CNC} &\textbf{ATCG-TV} &\textbf{LR-CNC}
\end{tabular}
\caption{Recovered MRI images by the CNC, ATCG-TV, and LR-CNC algorithms for SR$=0.1$.}
\label{fig 1}
\end{figure}

\begin{table}[h]
\centering
\resizebox{0.9\textwidth}{!}{
\begin{tabular}{@{}l|l|l|l|l|l|l|l|l|l|l@{}}
\hline Data & Method & \multicolumn{3}{|c|}{CNC} &\multicolumn{3}{|c|}{ATCG-TV} & 
\multicolumn{3}{|c}{LR-CNC} \\
\hline & SR & PSNR & CPU-time & Iter & PSNR & CPU-time & Iter & PSNR & CPU-time & Iter  \\
\hline \multirow{3}{*}{MRI} & 0.1 & 21.22 & 20.22& 730&18.44&8.20&968 & \bf{21.36} & 
\bf{4.54} & \bf{174}  \\
 & 0.2 & 23.14& 5.68 &355 &21.96&7.13&782& \bf{23.37}& \bf{2.69} & \bf{98}  \\
 & 0.3 & 25.01&4.00 &262 &24.37&8.56&889 & \bf{25.32}  & \bf{1.65} & \bf{60}  \\
\hline \multirow{3}{*}{Cameraman} & 0.1 &  21.34 & 11.53 &  757 & 19.31&6.58&770 & 
\bf{21.52} & \bf{5.09} & \bf{184}  \\ 
 & 0.2 & 23.11 & 5.52& 361 & 22.11&7.31&785 & \bf{23.27}& \bf{2.64} & \bf{91} \\ 
 & 0.3 & 24.33 & 3.35 &  220 & 24.08&6.90&755 & \bf{24.70}  & \bf{1.66} & \bf{58}\\
\hline
\end{tabular}}
\caption{PSNR-values of the restored images, CPU-times, and number of iterations by the 
CNC, ATCG-TV, LR-CNC methods for various values of SR.} \label{tab 1} 
\end{table}

Figure \ref{fig 1} and Table \ref{tab 1} show that for small SR-values, the LR-CNC 
algorithm outperforms the ATCG-TV algorithm. However, for larger SR-values, the 
difference in performance between the two algorithms is less significant. The LR-CNC 
algorithm demonstrates a clear advantage in terms of CPU time and the number of iterations
required in comparison with both the CNC and ATCG-TV algorithms. Figure \ref{fig 2} 
displays the graphs of the logarithm of the mean square error versus the number of 
iterations for the LR-CNC, CNC, and ATCG-TV algorithms, as well as the values of 
\begin{equation}\label{logdiff}
\ln\left(\dfrac{\left\Vert U^{k+1}-U^k \right\Vert_F}{\left\Vert U^k\right\Vert_F}\right)
\end{equation}
and the evolution of the PSNR-values as a function of the number of iterations when these 
algorithms are applied to the MRI image for SR$=0.1$.

\begin{figure}[H]
	\centering
	\begin{tabular}{cc}
		\includegraphics[width=0.4\linewidth]{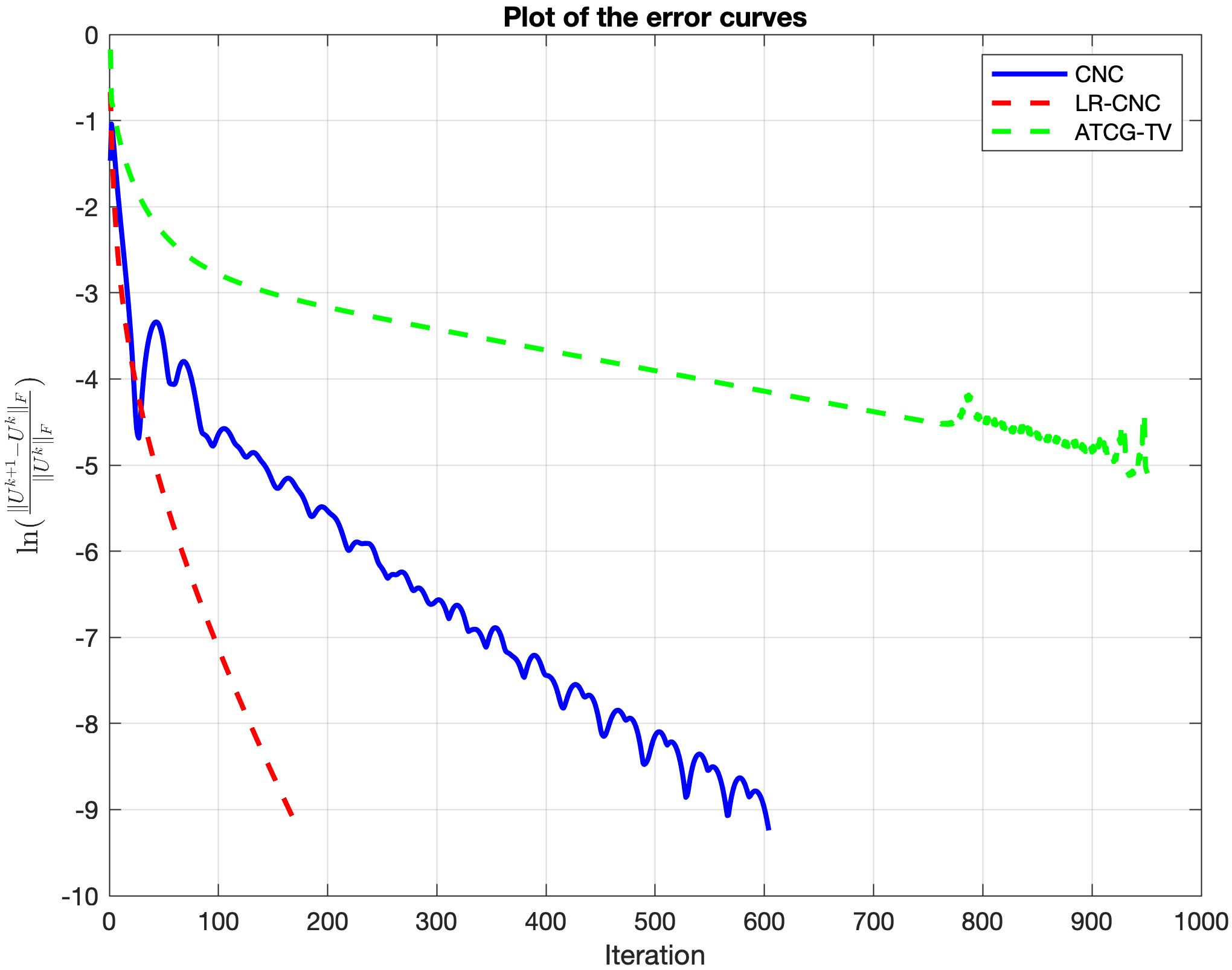}&  
		\includegraphics[width=0.4\linewidth]{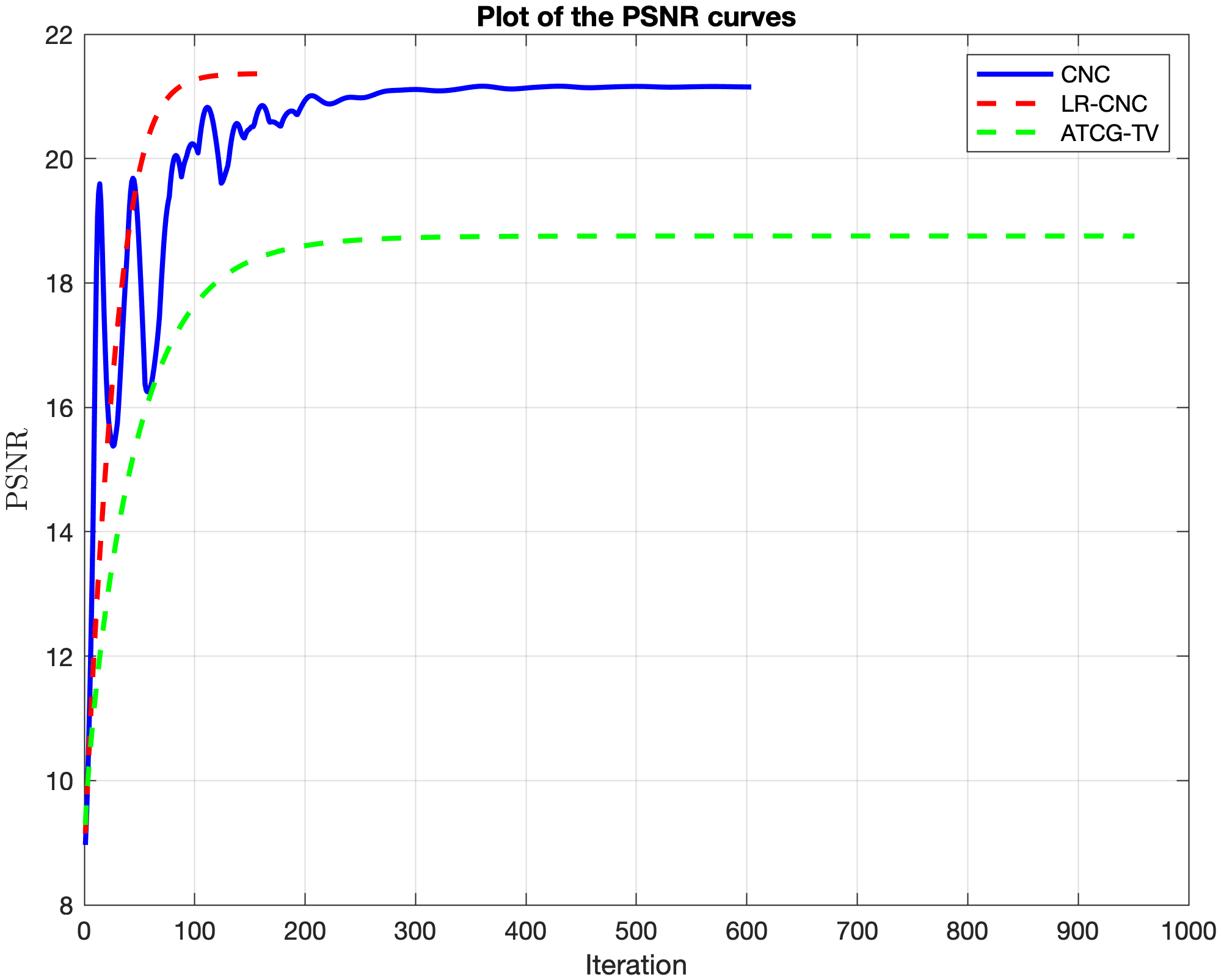}
	\end{tabular}
\caption{Logarithm of the differences \eqref{logdiff} and the PSNR-values as a function of
the number of iterations for the CNC, ATCG-TV, and LR-CNC algorithms applied to the MRI 
image with SR$=0.1$.}\label{fig 2}
\end{figure}

Figure \ref{fig 2} shows the relative differences \eqref{reldiff} obtained with the 
LR-CNC algorithm to decrease faster than the corresponding differences for the CNC and 
ATCG-TV algorithms. Moreover, the PSNR-values of the restorations determined by the LR-CNC
algorithm increase quickly and smoothly with the iteration number. In fact, they increase
much faster and in a smoother manner with the iteration number than for the CNC and ATCG-TV 
algorithms. 

\subsubsection{Color images}
We apply the LR-CNC, CNC, and ATCG-TV algorithms to the restoration of color images. In
our experiments, we use the well-known Airplane and Barbara images, both of which are 
represented by arrays of size $256 \times 256 \times 3$. Thus, the red, green, and blue 
channels are each represented by a tensor slice. These images are reshaped into matrices 
of size $256 \times (256*3)$ by using the MATLAB command
$\texttt{reshape}\left(\cdot,[256,256*3]\right)$.

Figure \ref{fig 3} displays results for the color images Airplane and Barbara with $90\%$ 
missing data when reconstructed by the CNC, ATCG-TV, and LR-CNC algorithms.

\begin{figure}[H]
\centering
\begin{tabular}{ccccc}
\includegraphics[width=0.15\linewidth]{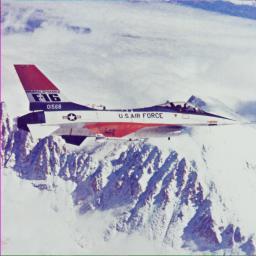}&
\includegraphics[width=0.15\linewidth]{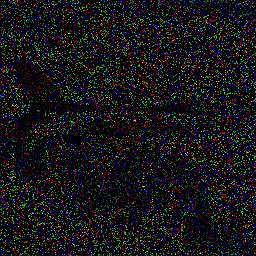}&
\includegraphics[width=0.15\linewidth]{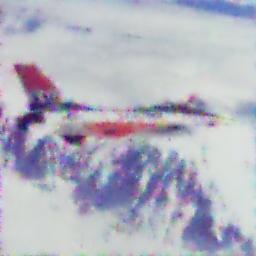}&
\includegraphics[width=0.15\linewidth]{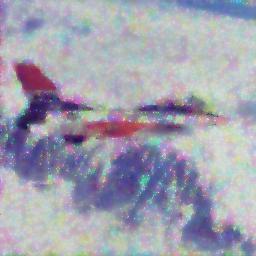}&
\includegraphics[width=0.15\linewidth]{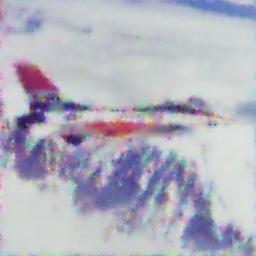}\\
\includegraphics[width=0.15\linewidth]{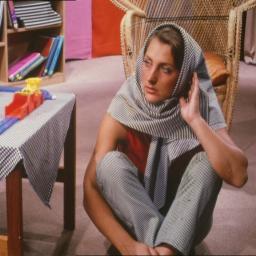}&
\includegraphics[width=0.15\linewidth]{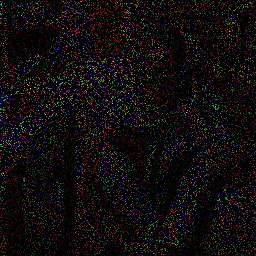}&
\includegraphics[width=0.15\linewidth]{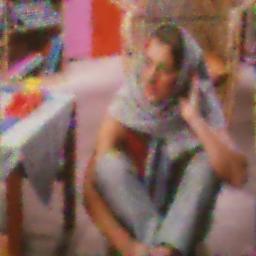}&
\includegraphics[width=0.15\linewidth]{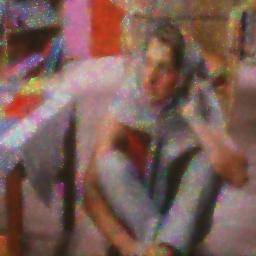}&
\includegraphics[width=0.15\linewidth]{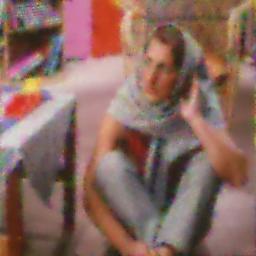}\\
\textbf{Original} &\textbf{Observed}  &\textbf{CNC} &\textbf{ATCG-TV} &\textbf{LR-CNC}
\end{tabular}
\caption{Recovered Airplane and Barbara images determined by the CNC, ATCG-TV, and LR-CNC
algorithms with SR$=0.1$. The observed images are dark because they consist of only $10\%$
of the pixel values of the original images.}\label{fig 3}
\end{figure}

Table \ref{tab 2} displays the PSNR-values, CPU-time, and the number of iterations 
required by the CNC, ATCG-TV, and LR-CNC algorithms to satisfy the stopping criterion.
The computations are carried out for the Airplane and Barbara images for several 
SR-values. Figure \ref{fig 4} depicts the relative differences \eqref{logdiff} and the
PSNR-values as functions of the iteration number for the CNC, ATCG-TV, and LR-CNC 
algorithms.

\begin{table}[h]
\centering
\begin{tabular}{@{}l|l|l|l|l|l|l|l|l|l|l@{}}
\hline Data & Method & \multicolumn{3}{|c|}{CNC} &\multicolumn{3}{|c|}{ATCG-TV} & 
\multicolumn{3}{|c}{LR-CNC} \\
\hline & SR & PSNR & CPU-time & Iter & PSNR & CPU-time & Iter & PSNR & CPU-time & Iter  \\
\hline \multirow{3}{*}{Airplane} & 0.1 & 22.26 &56.64& 714 & 18.64 &19.24 &736 & 
\bf{22.47} &\bf{4.14} &\bf{78} \\
 & 0.2 & 24.39&13.48 &355 & 23.04&20.90 &736  &\bf{24.53}&\bf{2.11}&\bf{38} \\
& 0.3 & 25.86 & 9.56 &212 & 25.21 & 22.42 &753 & \bf{26.13}&\bf{2.05}&\bf{37} \\
\hline \multirow{3}{*}{Barbara} & 0.1 & 23.11 & 24.30 &722 & 20.73 &20.78 &849 & 
\bf{23.24}&\bf{3.63}&\bf{71} \\
 & 0.2 & 25.19 & 13.14 & 346 & 24.42 & 24.59 &814 &	\bf{25.37}&\bf{2.39}&\bf{41} \\
 & 0.3 & 26.76 & 8.47 & 217 & 26.63 &22.32 &769   &	\bf{26.89}&\bf{2.00}&\bf{36} \\
\hline
\end{tabular}
\caption{PSNR-values, CPU-time, and the number of iterations required by the CNC, ATCG-TV,
and LR-CNC algorithms for several SR-values.}\label{tab 2}
\end{table}

\begin{figure}[H]
\centering
\begin{tabular}{cc}
\includegraphics[width=0.4\linewidth]{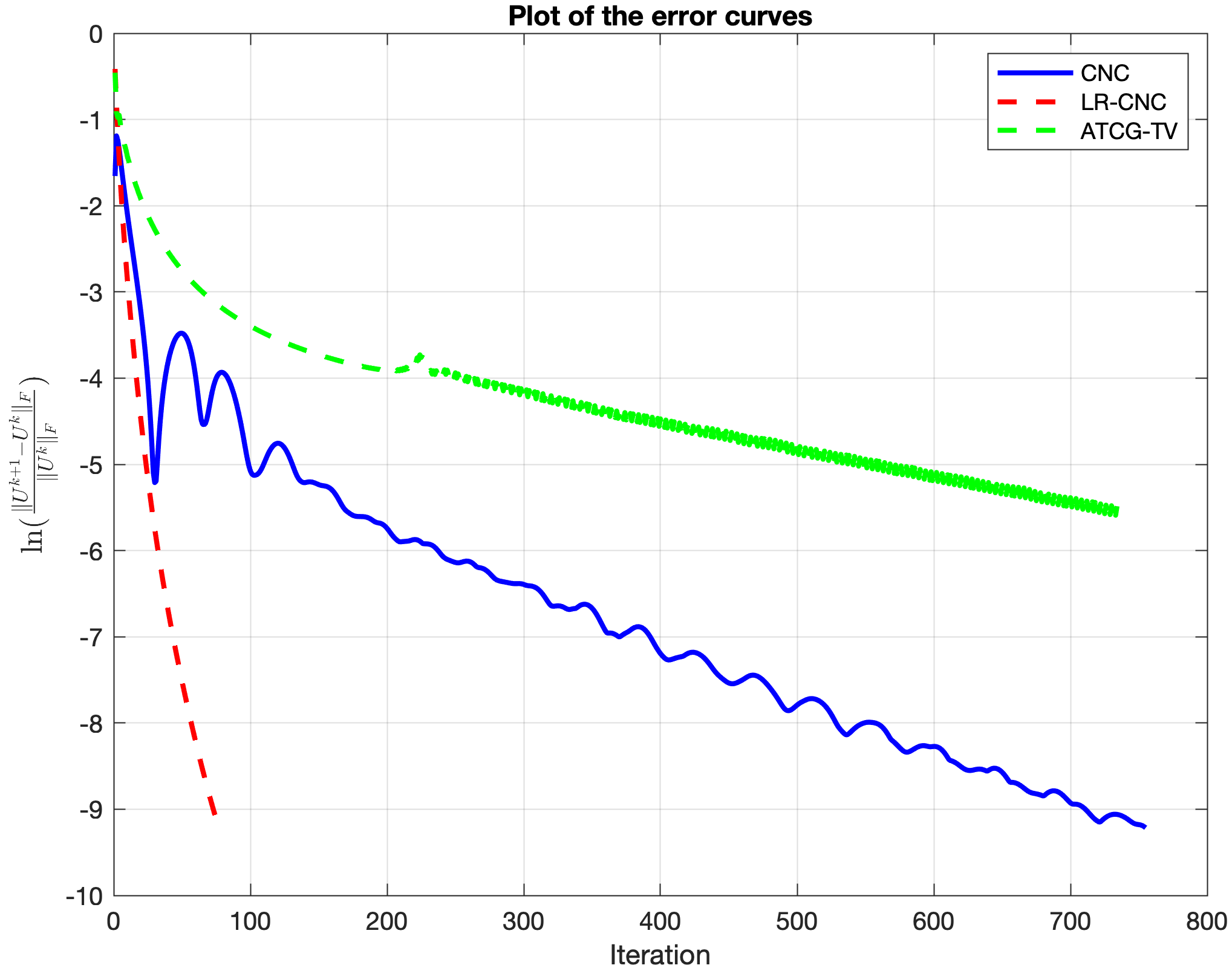}&  
\includegraphics[width=0.4\linewidth]{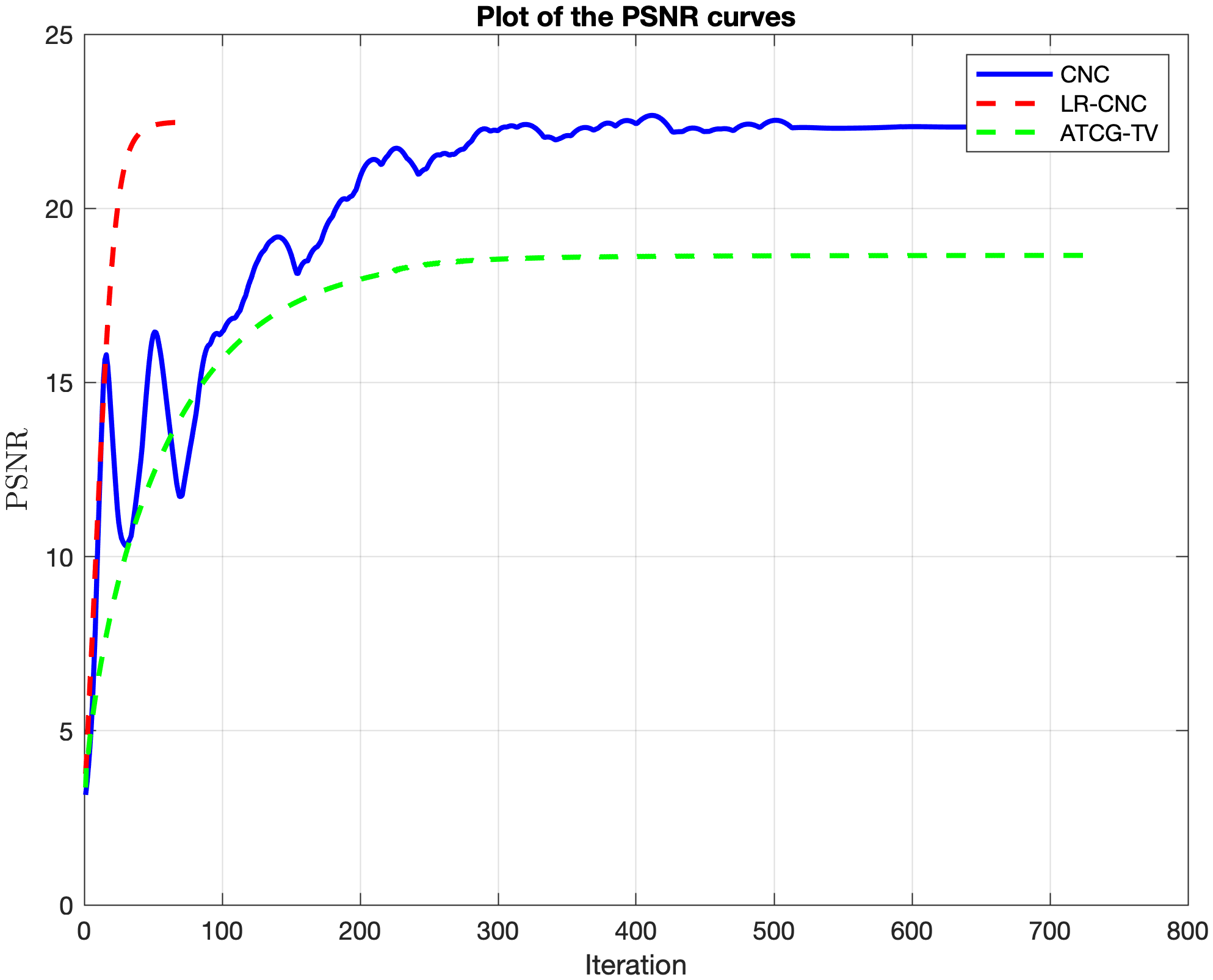}
\end{tabular}
\caption{The relative differences \eqref{logdiff} and PSNR-values as functions of the 
number of iterations for the CNC, ATCG-TV, and LR-CNC algorithms for SR$=0.1$.}
\label{fig 4}
\end{figure}

The above results show the LR-CNC algorithm to consistently produce images with higher 
PSNR/SSIM values than the other algorithms in our comparison, especially in challenging 
scenarios (very low sampling ratio or high noise levels). Qualitatively, the LR-CNC 
algorithm gives reconstructions that retain details and segments images into regions 
more distinctly than the other algorithms; see Figures \ref{fig 3} and \ref{fig 4}. An 
advantage of the LR-CNC algorithm is its fast convergence: the algorithm achieves accurate
results in far fewer iterations (\ref{tab 2}) than the other algorithms due to the 
efficiency of low-rank regularization and our ADMM scheme.

\subsection{Image segmentation}
This subsection applies the LR-CNC algorithm (Algorithm \ref{Alg 1}) to image 
segmentation. We compare the performance of this algorithm to the CNC and ATCG-TV 
algorithms, as well as to an algorithm proposed by Chan and Vese \cite{T.Chan}. In all
experiments, we apply these algorithms to images that have been contaminated by Gaussian 
noise and then use the K-means algorithm \cite{Iko} to segment the resulting image, with 
$K$ equal to the number of regions in the ground truth. Given a noise-free image 'Im', we 
generate a noise-contaminated image $B$ with the MATLAB command 
\[
B = \texttt{imnoise}(\text{Im}, '\text{gaussian}', L, 0.01); 
\]
where the parameter $L$ specifies the mean of the Gaussian noise; $0.01$ is the variance 
of the noise. The variance is the default value. We will refer to the value of the 
parameter $L$ as the Gaussian noise level. For all segmentation experiments, we let 
$a\in\{0.1,0.2\}$. If $L$ is relatively large, then we set $a=0.2$, otherwise we set 
$a=0.1$. Moreover, we let $T=10^{-6}$, $\rho_1=2.5$, $\rho_2=10.001$, $\tau_1=\rho_1$, 
$\tau_2=\rho_1$, and $\tau_3=1.0001$. Code for the algorithm by Chan \& Vese \cite{T.Chan}
is available at \footnote{\url{https://fr.mathworks.com/matlabcentral/fileexchange/34548-active-contour-without-edge}}.

Figure \ref{fig 5} displays results obtained with the CNC, ATCG-TV, Chan \& Vese, and 
LR-CNC algorithms for Gaussian noise level $L=0.3$. These tests are applied to the 
MRI image with the cluster number set to $K=3$, and to a Millennium simulation 
\footnote{\url{https://wwwmpa.mpa-garching.mpg.de/galform/virgo/millennium/}} with cluster
number $K=2$.

\begin{figure}[H]
\centering
\begin{tabular}{cccccc}
\includegraphics[width=0.12\linewidth]{completion_results/original_seg.jpg}&
\includegraphics[width=0.12\linewidth]{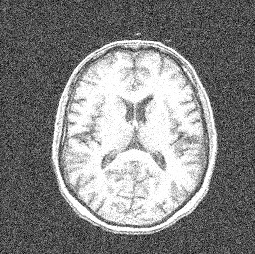}&
\includegraphics[width=0.12\linewidth]{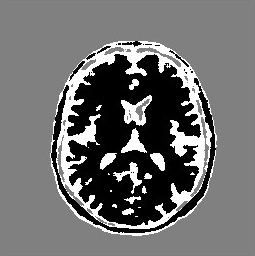}&
\includegraphics[width=0.12\linewidth]{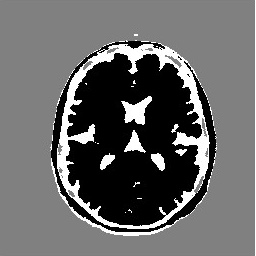}&
\includegraphics[width=0.12\linewidth]{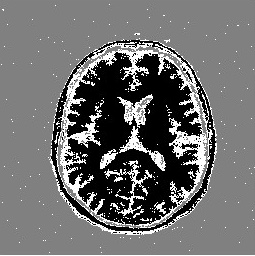}&
\includegraphics[width=0.12\linewidth]{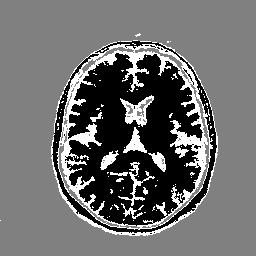}\\
\includegraphics[width=0.12\linewidth]{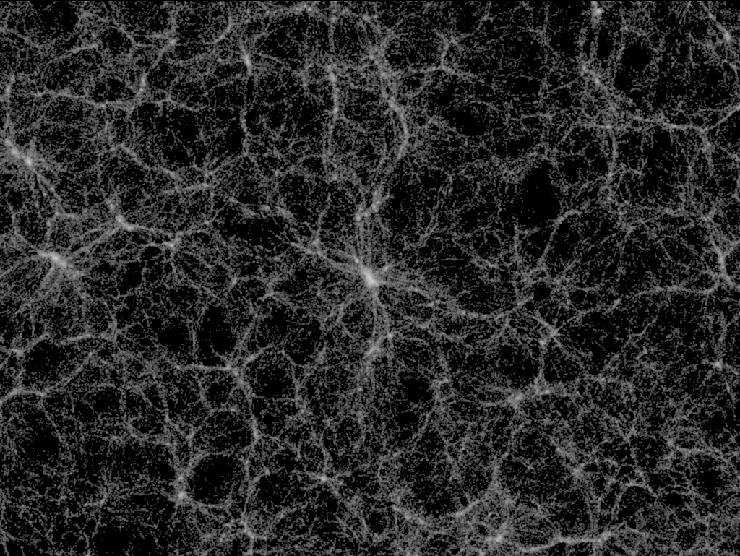}&
\includegraphics[width=0.12\linewidth]{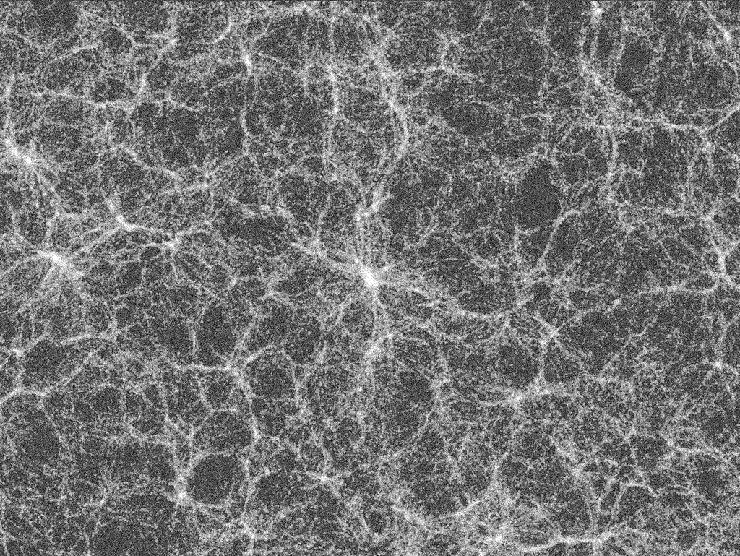}&
\includegraphics[width=0.12\linewidth]{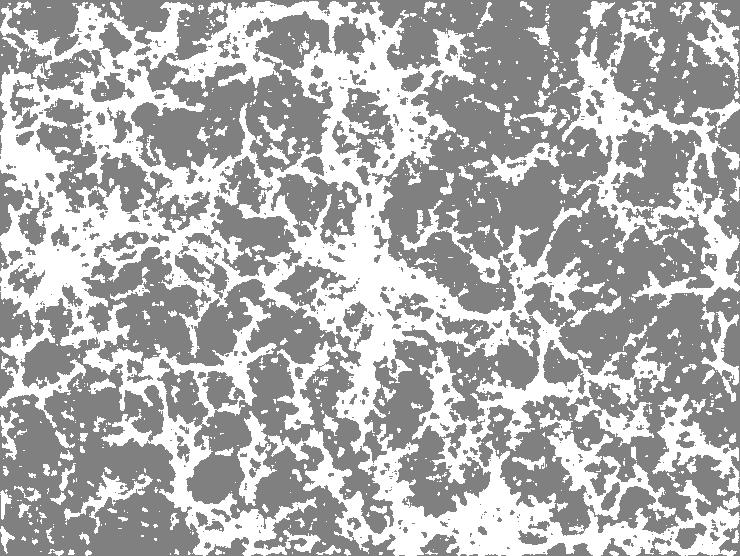}&
\includegraphics[width=0.12\linewidth]{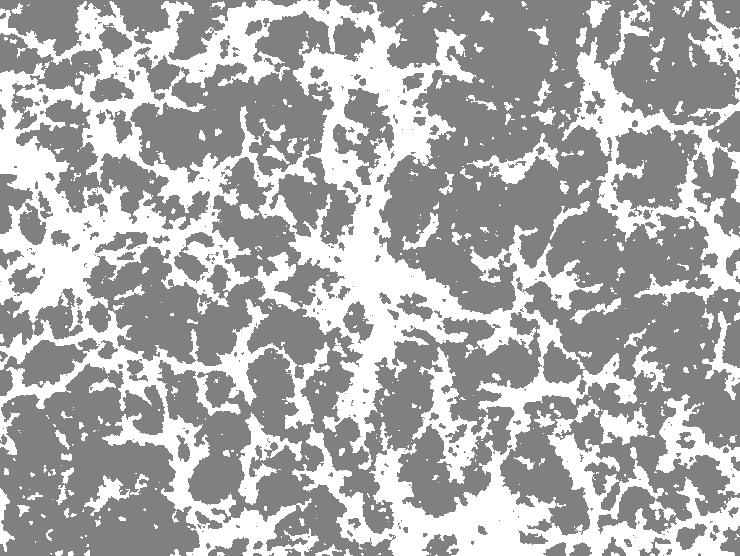}&
\includegraphics[width=0.12\linewidth]{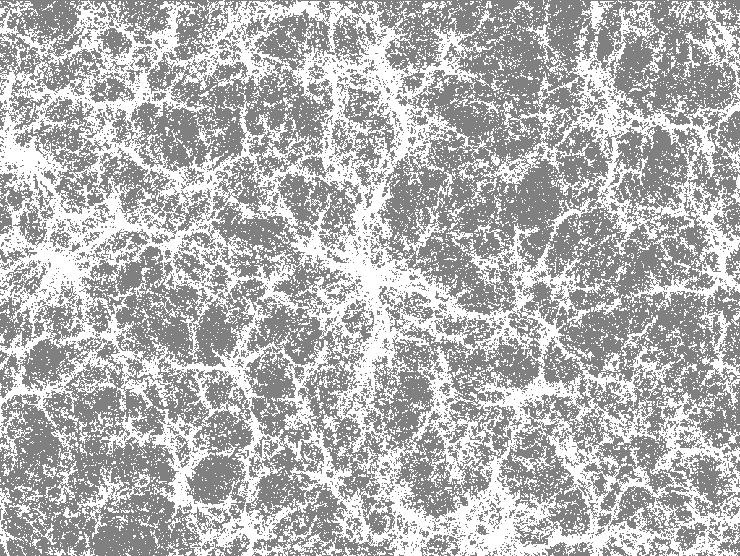}&
\includegraphics[width=0.12\linewidth]{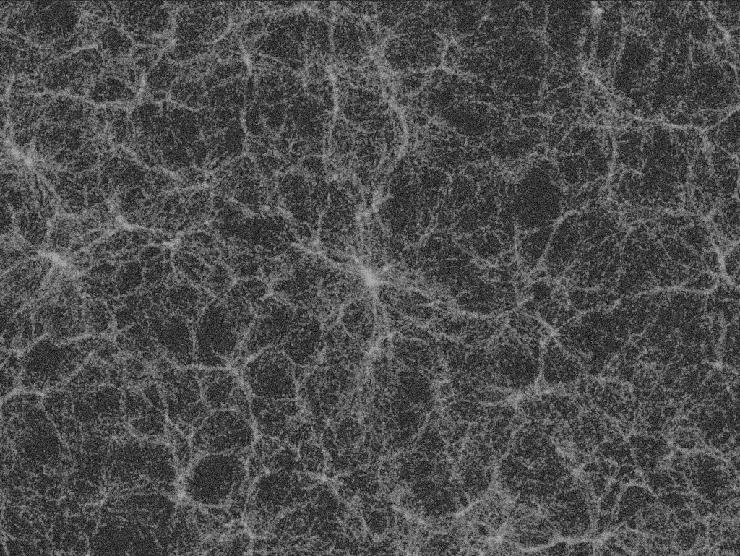}\\
\textbf{Original} &\textbf{Observed} & \textbf{Chan $\&$ Vese} &  \textbf{ATCG-TV} &\textbf{CNC} &\textbf{LR-CNC}
\end{tabular}
\caption{Segmented MRI and Millenium images for Gaussian noise with $L=0.3$, and for $K=3$
and $K=2$, respectively.}\label{fig 5}
\end{figure}

Table \ref{tab 3} shows PSNR and SSIM values for segmented images determined by the CNC, 
ATCG-TV, Chan \& Vese, and LR-CNC algorithms for several Gaussian noise levels $L$ when 
applied to the MRI and Millennium images. Figures \ref{fig 6} and \ref{fig 7} display the
relative differences \eqref{logdiff} and PSNR-values of images determined at each 
iteration of the CNC, ATCG-TV, and LR-CNC algorithms as functions of the iteration number 
when these algorithms are applied to the MRI and Millennium images. The Gaussian noise 
level is $L=0.3$ and we seek to determine $K = 3$ and $K = 2$ clusters in the MRI and 
Millenium images, respectively.

\begin{table}[h]
\centering
\begin{tabular}{@{}l|l|l|l|l|l|l|l|l|l@{}}
\hline method  & & \multicolumn{2}{c|}{CNC} & \multicolumn{2}{|c|}{ATCG-TV} & \multicolumn{2}{|c}{Chan $\&$ Vese \cite{T.Chan}} & \multicolumn{2}{|c}{LR-CNC}  \\
\hline & $L$ & PSNR & SSIM & PSNR & SSIM & PSNR & SSIM & PSNR & SSIM \\
\hline \multirow{4}{*}{MRI}& 0.01 & 24.94 & 0.42 & 24.37 & 0.49&24.17 & 0.64 & \bf{25.12} & \bf{0.49}\\
 &  0.1 & 18.73 & 0.33 &18.11 & 0.37 &18.71 & 0.37 & \bf{20.50} & \bf{0.38}  \\
  & 0.2 & 13.78 & 0.29 &13.80 & 0.31 & 13.71 & 0.32 & \bf{15.43} & \bf{0.35} \\
 & 0.3 &10.62 & 0.27 & 10.63 & 0.26 & 10.54 & 0.27 & \bf{12.03} & \bf{0.31} \\
\hline \multirow{4}{*}{Millennium } & 0.01  & 22.74 & 0.71 & 22.40 & 0.69 & 18.94 & 0.37 & 
\bf{22.40} & \bf{0.69} \\
 & 0.1& 18.14 & 0.62 & 16.20 & 0.30 & 16.66 & 0.30 & \bf{20.30} & \bf{0.65} \\
  & 0.2 & 13.34 & 0.49 & 12.21 & 0.26 &  13.09 & 0.24 & \bf{18.29} & \bf{0.59}\\
&  0.3&10.10 & 0.40 & 9.15 & 0.22 &10.11 & 0.20 & \bf{13.60} & \bf{0.51}\\
\hline
\end{tabular}
\caption{PSNR- and SSIM-values for images determined by the CNC, ATCG-TV, Chan \& Vese, 
and LR-CNC algorithms for Gaussian noise with several noise levels $L$.}\label{tab 3}
\end{table}

\begin{figure}[H]
\centering
\begin{tabular}{cc}
\includegraphics[width=0.4\linewidth]{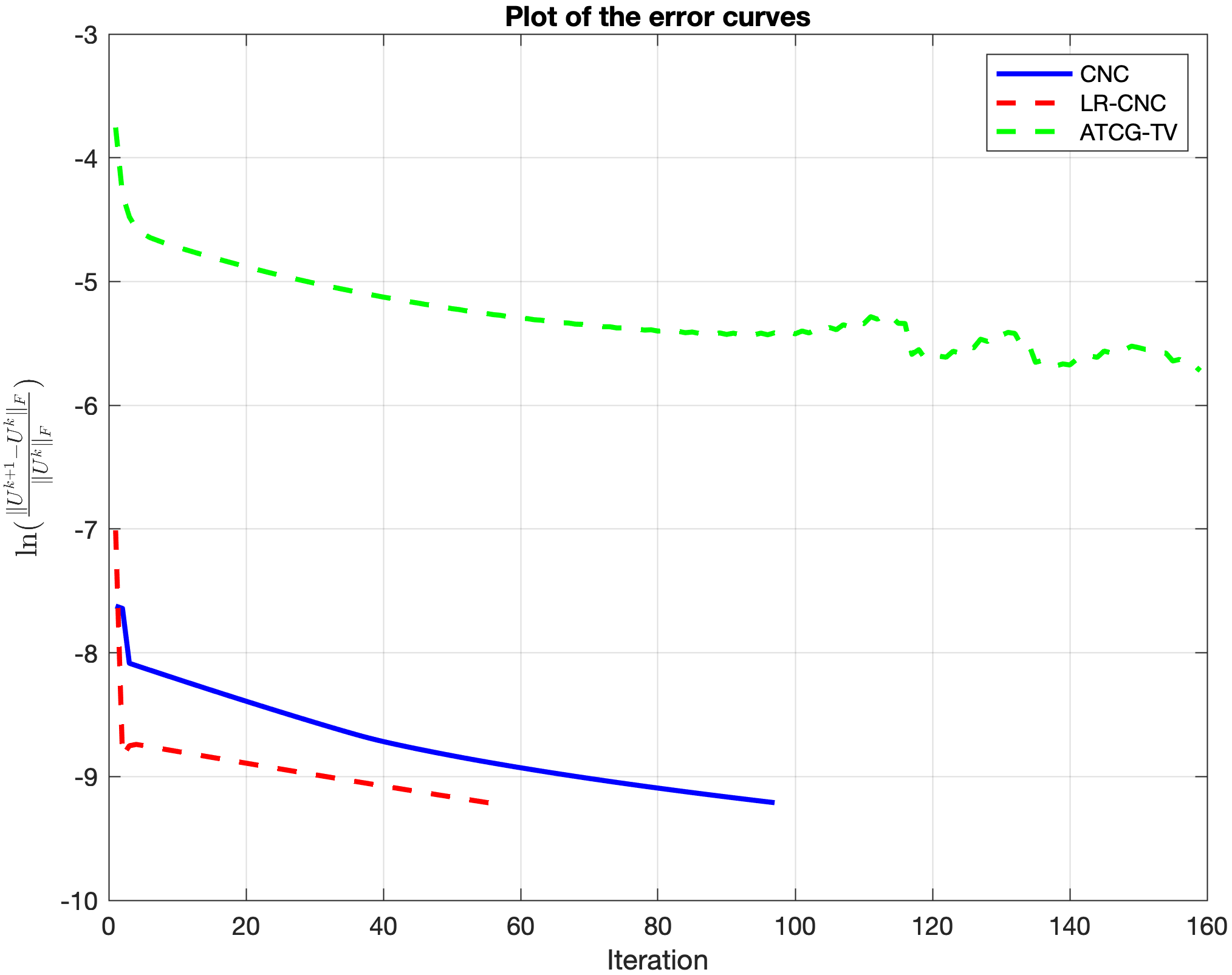}   & 
\includegraphics[width=0.4\linewidth]{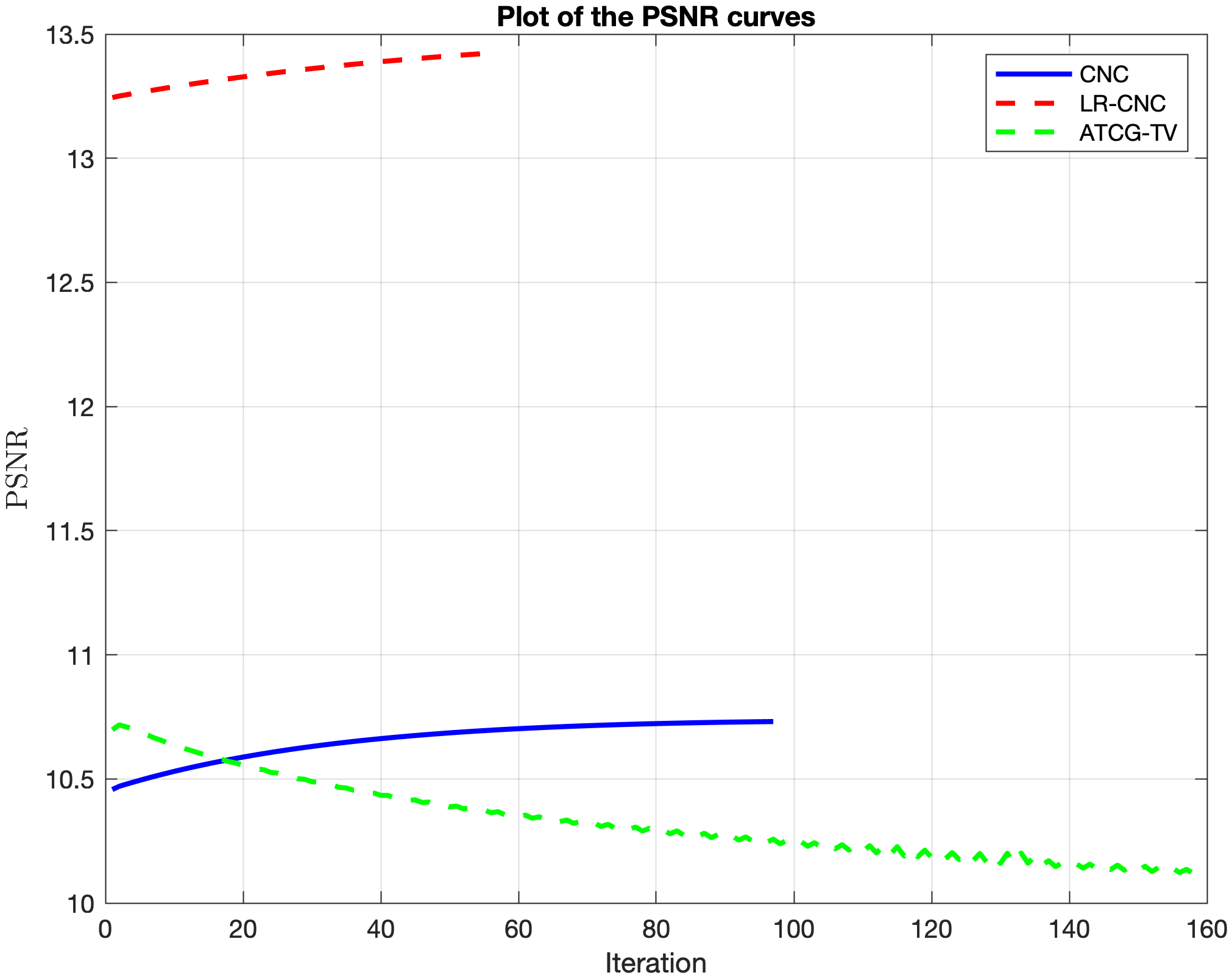} 
\end{tabular}
\caption{The relative differences \eqref{logdiff} and PSNR-values as functions of the
iteration number obtained with CNC, ATCG-TV, and LR-CNC algorithms applied to the MRI 
image with \textcolor{blue}{Gaussian noise level} $L=0.3$ and $K=3$.}\label{fig 6}
\end{figure}

\begin{figure}[H]
\centering
\begin{tabular}{cc}
\includegraphics[width=0.4\linewidth]{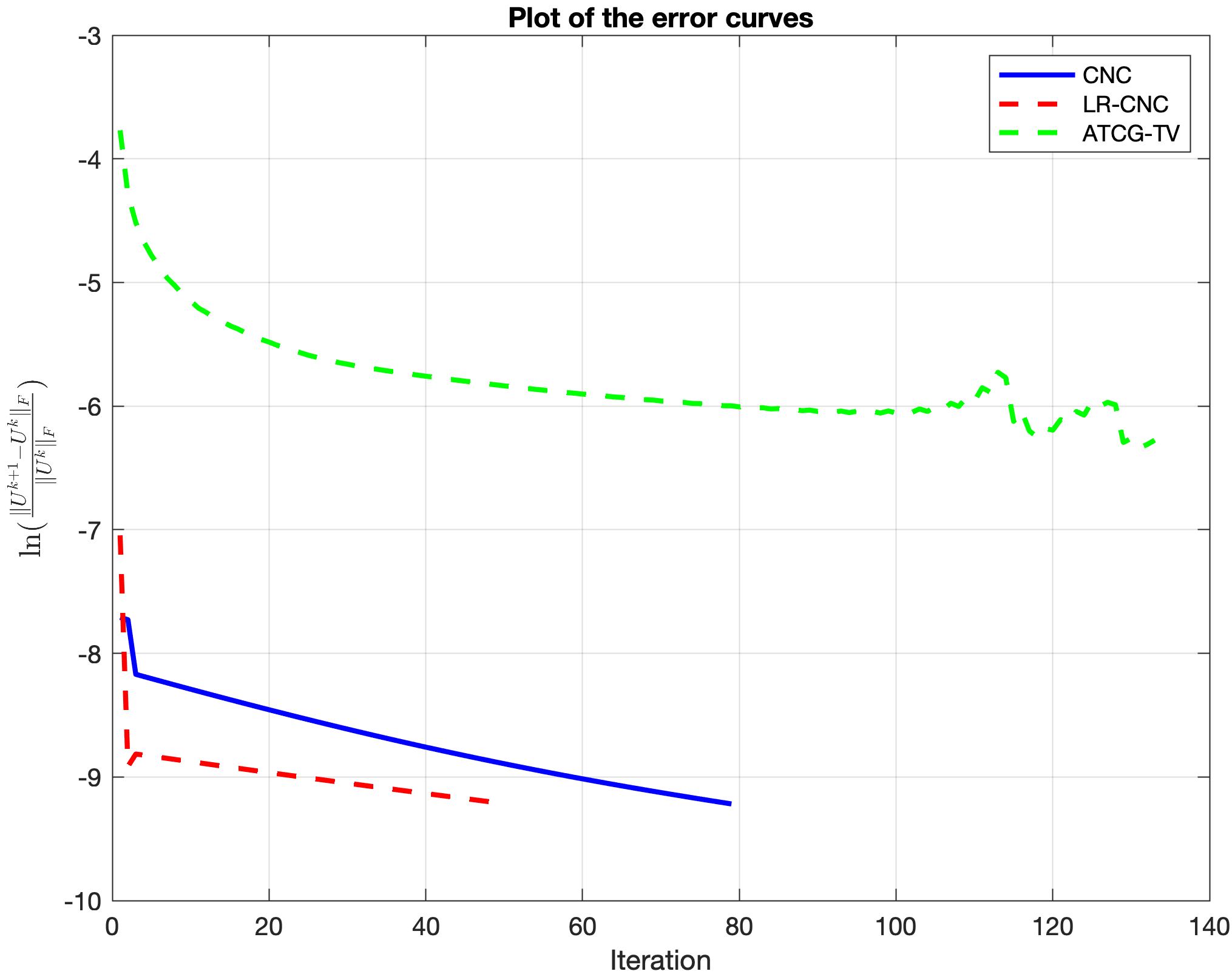}   & 
\includegraphics[width=0.4\linewidth]{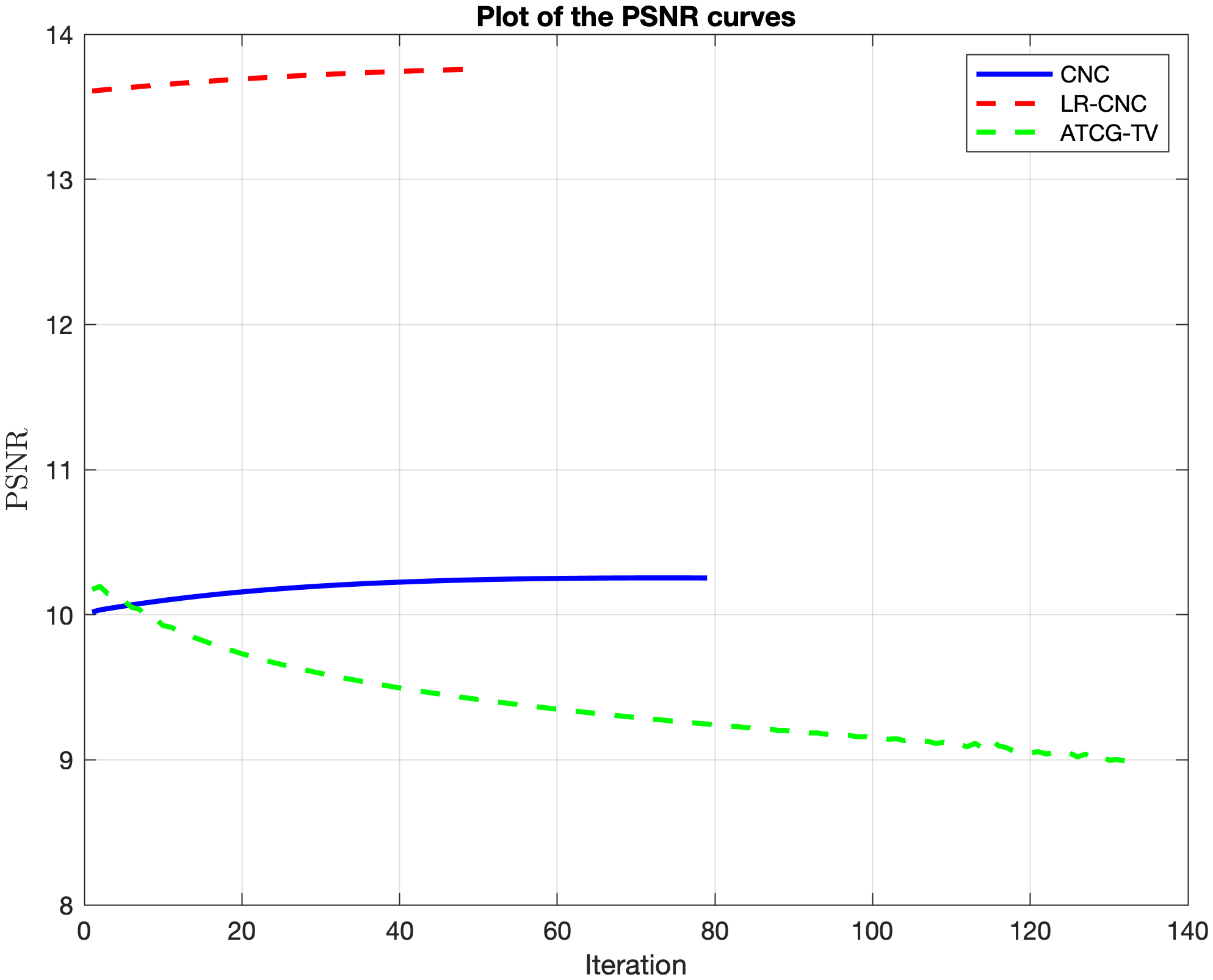} 
\end{tabular}
\caption{The relative differences \eqref{logdiff} and PSNR-values as functions of the
iteration number obtained with CNC, ATCG-TV, and LR-CNC algorithms applied to the 
Millenium image with Gaussian noise level $L=0.3$ and $K=2$.}\label{fig 7}
\end{figure}

Figure \ref{fig 5} shows the segmented images produced by the LR-CNC algorithm to retain 
details with high precision for both images, also when the amount of noise is large. 
Furthermore, the PSNR and SSIM values reported in Table \ref{tab 3} indicate that the
LR-CNC algorithm consistently exhibits greater robustness than the other algorithm, with 
particularly strong performance for larger noise levels. Figures \ref{fig 6} and 
\ref{fig 7} illustrate the smooth and rapid convergence of the iterates determined by the
LR-CNC algorithm.

Our last example is concerned with segmentation of hyperspectral images. Hyperspectral 
images contain more detailed information than RGB images. However, material separation 
based on hyperspectral images remains a challenging task. It involves isolating materials 
from one another. Various approaches have been developed to carry out material 
separation, including non-negative matrix factorization \cite{elha_NFT}.

We consider two hyperspectral images: Samson and Jasper Ridge. Both images are captured by 
the Airborne Visible/Infrared Imaging Spectrometer AVIRIS. Initially, these images 
contained 224 spectral bands. After removing the water absorption bands, the Samson image 
has 156 bands and the Jasper Ridge image has 198 bands. The Samson dataset is a tensor of 
size $95 \times 95 \times 156$, and the Jasper Ridge dataset is a tensor of size 
$100 \times 100 \times 198$. The Samson dataset contains three materials: water, soil, and 
trees, whereas the Jasper Ridge dataset includes four materials: water, soil, trees, and 
roads. Figures \ref{fig 8} and \ref{fig 9} show the different materials of each dataset.
By using the non-negative tensor factorization described in \cite{elha_NFT}, it can be 
shown that the abundance of each material is presented by the color between red and 
yellow.

\begin{figure}[H]
\centering
\begin{tabular}{cccc}
\includegraphics[width=0.2\linewidth]{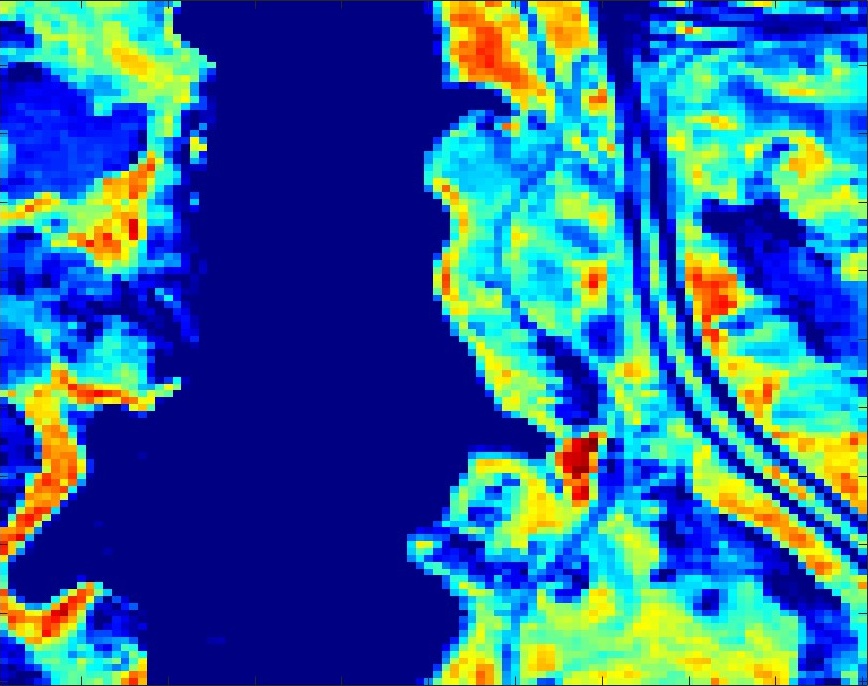}&
\includegraphics[width=0.2\linewidth]{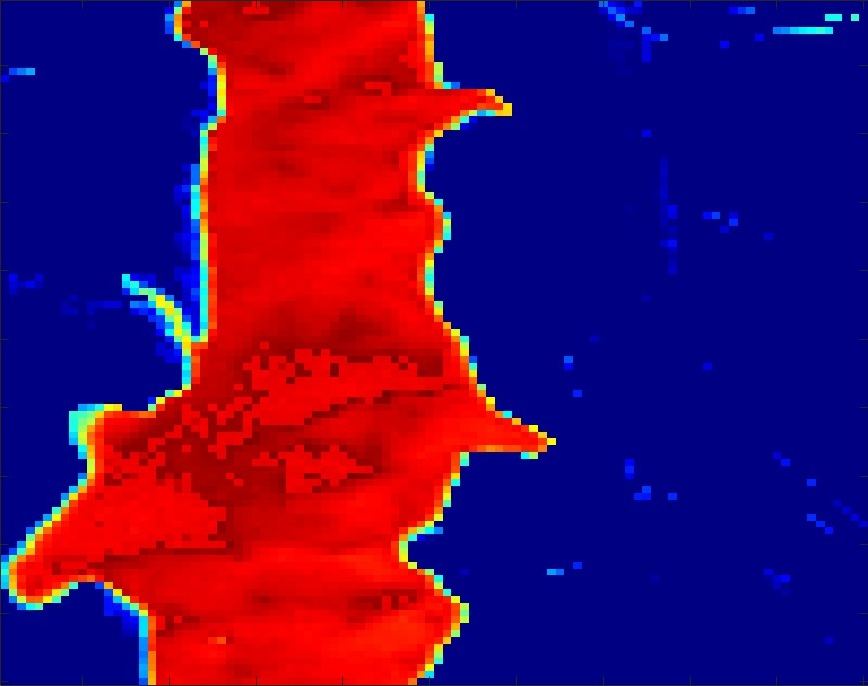}&
\includegraphics[width=0.2\linewidth]{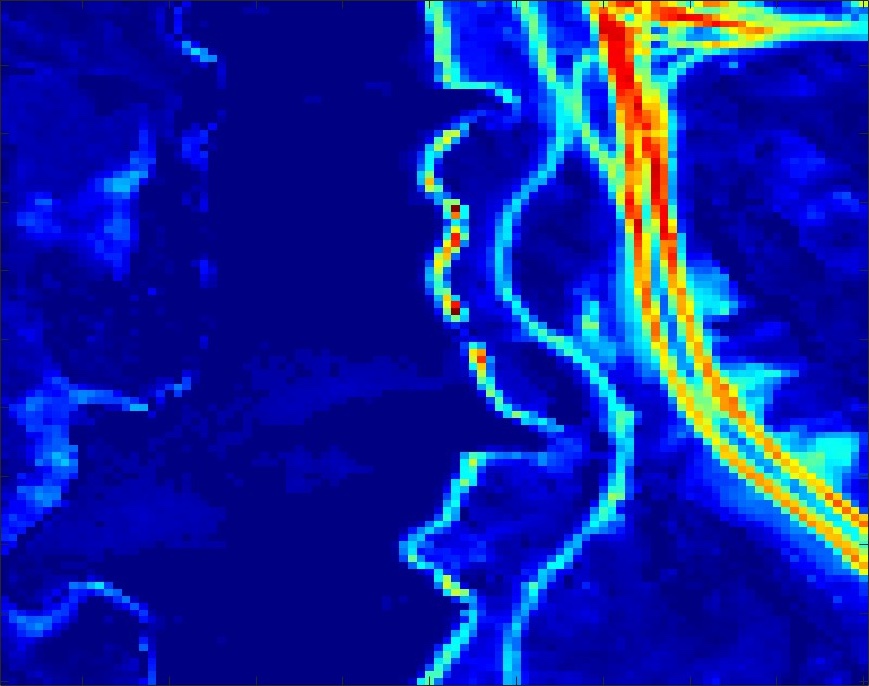}&
\includegraphics[width=0.2\linewidth]{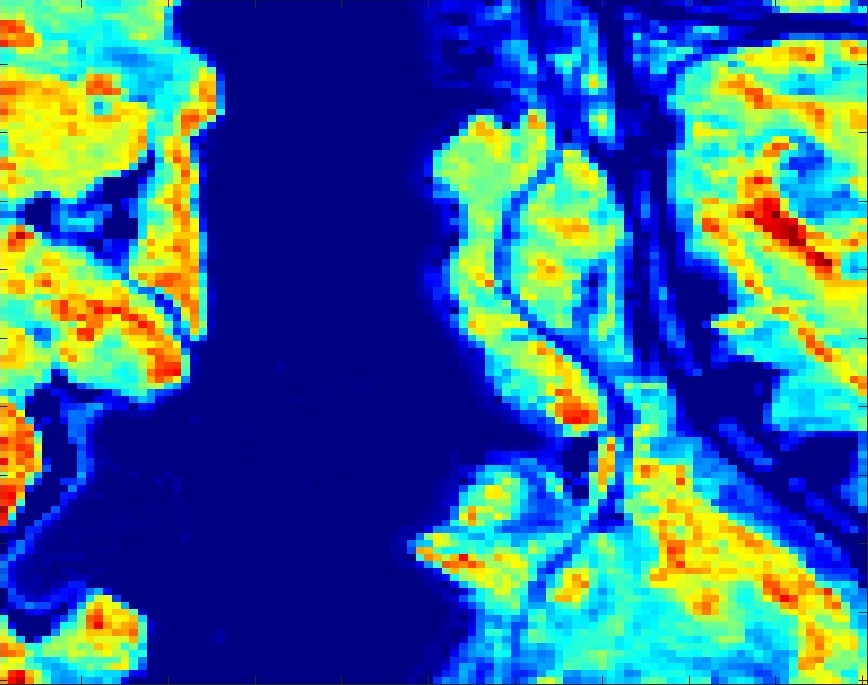}\\
\textbf{Soil} & \textbf{Water} & \textbf{Road} & \textbf{Trees}
\end{tabular}
\caption{The separated materials of the Jasper Ridge dataset.}
\label{fig 8}
\end{figure}

\begin{figure}[H]
\centering
\begin{tabular}{ccc}
\includegraphics[width=0.2\linewidth]{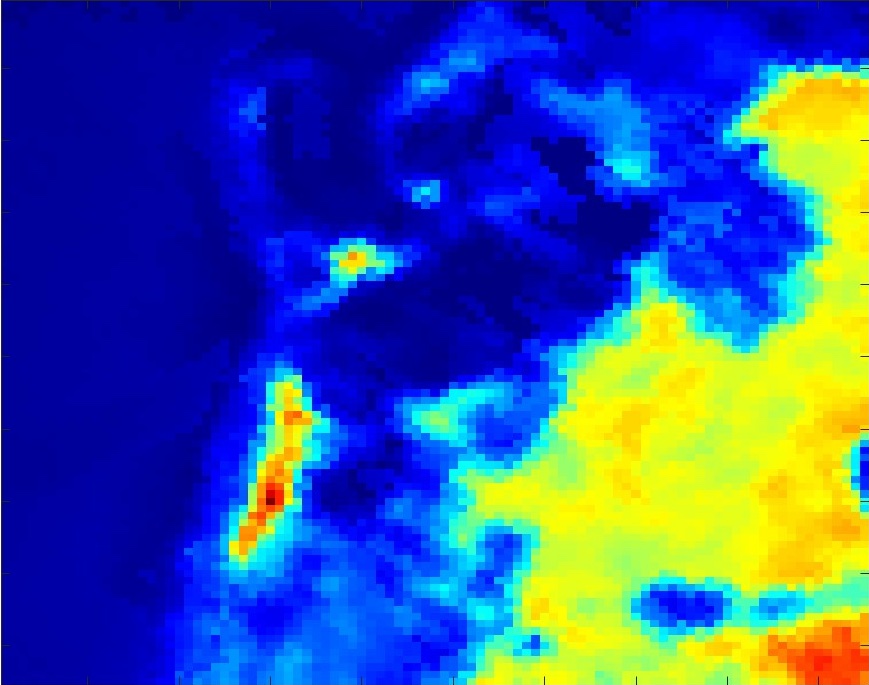}&
\includegraphics[width=0.2\linewidth]{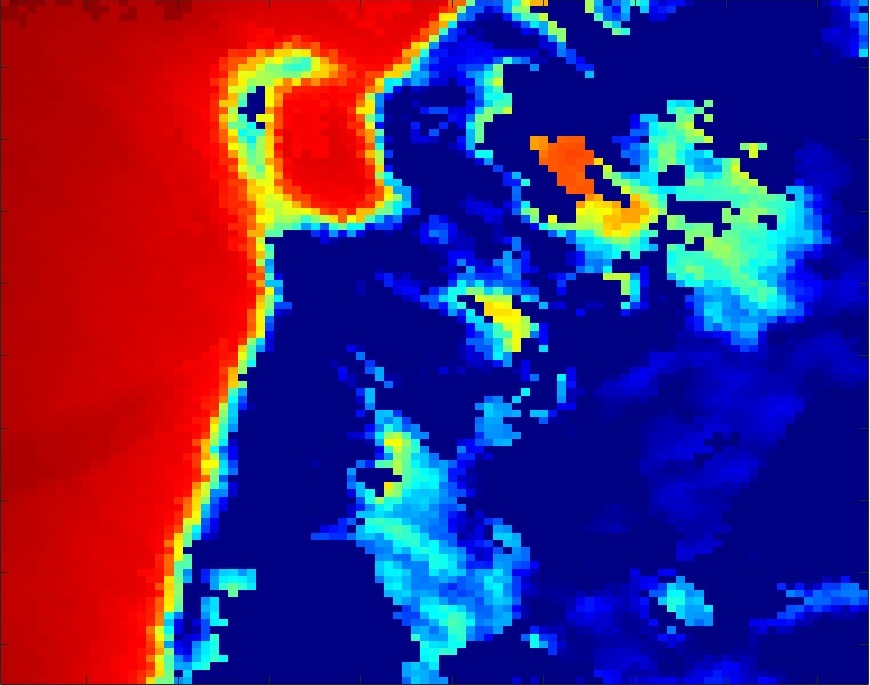}&
\includegraphics[width=0.2\linewidth]{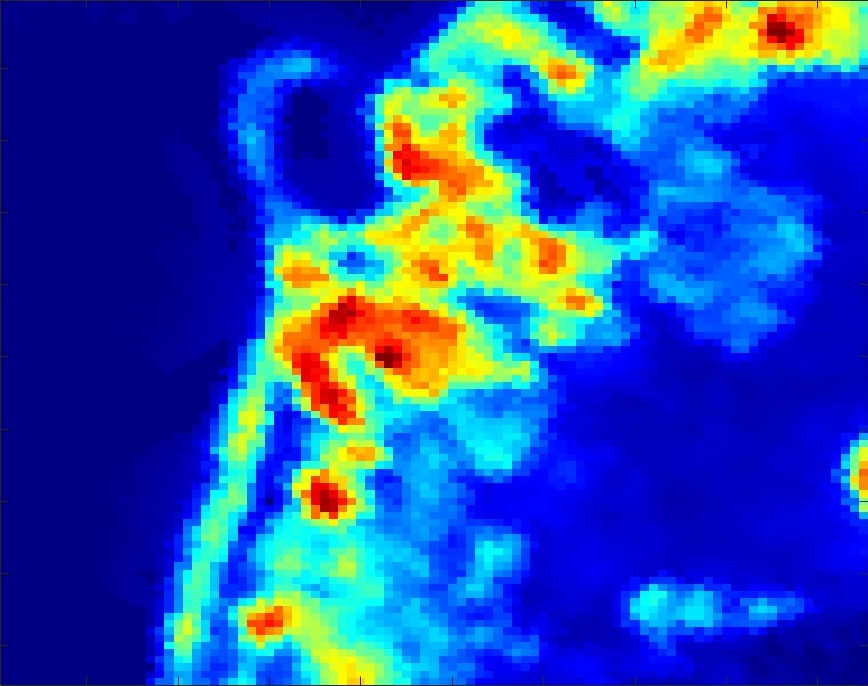}\\
\textbf{Soil} & \textbf{Water} & \textbf{Trees}
\end{tabular}
\caption{The separated materials of the Samson dataset.}
\label{fig 9}
\end{figure}

Figure \ref{fig 10} displays the results of the CNC, Chan \& Vese, ATCG-TV, and LR-CNC 
algorithms applied to the hyperspectral images Samson and Jasper Ridge. Both images are 
corrupted by Gaussian noise of level $0.3$. For the Samson image, we represent water in 
black, trees in gray with an intensity of $0.5$, and soil in white. For the Jasper Ridge 
image, water is represented in black, trees in gray with an intensity of $0.3$, soil in 
gray with an intensity of $0.7$, and roads are shown in white. Given the presence of three 
materials in the Samson test, we use $K=3$, while for the Jasper Ridge image, we use $K=4$ 
to represent the four materials.

\begin{figure}[H]
\centering
\begin{tabular}{cccccc}
\includegraphics[width=0.12\linewidth]{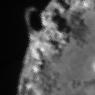}&
\includegraphics[width=0.12\linewidth]{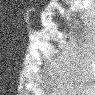}&
\includegraphics[width=0.12\linewidth]{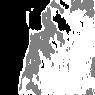}&
\includegraphics[width=0.12\linewidth]{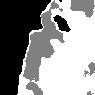}&
\includegraphics[width=0.12\linewidth]{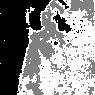}&
\includegraphics[width=0.12\linewidth]{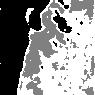}\\
\includegraphics[width=0.12\linewidth]{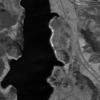}&
\includegraphics[width=0.12\linewidth]{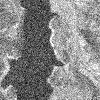}&
\includegraphics[width=0.12\linewidth]{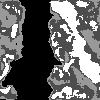}&
\includegraphics[width=0.12\linewidth]{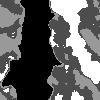}&
\includegraphics[width=0.12\linewidth]{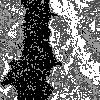}&
\includegraphics[width=0.12\linewidth]{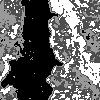}\\
\textbf{Original} &\textbf{Observed} & \textbf{Chan $\&$ Vese} &  \textbf{ATCG-TV} &\textbf{CNC} &\textbf{LR-CNC}
\end{tabular}
\caption{Results of the CNC, Chan \& Vese, ATCG-TV, and LR-CNC algorithms for the Samson 
and Jasper Ridge datasets for Gaussian noise level $0.3$ using $K=3$ and $K=4$, 
respectively.}
\label{fig 10}
\end{figure}

Figure \ref{fig 9} shows that the LR-CNC method identifies each material more distinctly 
than the other algorithms. Table \ref{tab 4} displays PSNR- and SSIM-values for the 
recovered data for the Samson and Jasper Ridge datasets for various Gaussian noise levels. 
Figures \ref{fig 10} and \ref{fig 11} depict relative differences \eqref{logdiff} and 
PSNR-values at each iteration of the algorithms for the Gaussian noise level $0.3$ for the
Samson and Jasper Ridge data sets, respectively.

\begin{table}[h]
\centering
\begin{tabular}{@{}l|l|l|l|l|l|l|l|l|l@{}}
\hline method  & & \multicolumn{2}{c|}{CNC} & \multicolumn{2}{|c|}{ATCG-TV} & \multicolumn{2}{|c}{Chan $\&$ Vese \cite{T.Chan}} & \multicolumn{2}{|c}{LR-CNC}  \\
\hline & $L$ & PSNR & SSIM & PSNR & SSIM & PSNR & SSIM & PSNR & SSIM\\
\hline \multirow{4}{*}{Samson}& 0.01 &26.87 & 0.51 & 29.04 & \bf{0.79} & \bf{29.61} & 
	0.70 & 28.02 & 0.75\\ 
 & 0.1 & 18.81 & 0.30 & 19.54 & 0.60 & 19.66 & 0.52 & \bf{21.38} & \bf{0.63}\\
 & 0.2 & 13.48 & 0.18 & 13.91 & 0.47 & 13.93 & 0.41 & \bf{18.22} & \bf{0.53}\\
 & 0.3 & 10.16 & 0.13 & 10.45 & 0.39 & 10.46 & 0.34 & \bf{15.66} & \bf{0.51} \\
\hline \multirow{4}{*}{Jasper} & 0.01 & 24.88 & 0.51 & 25.28 & 0.60& 26.16 & 0.62 & 
	\bf{27.01}& \bf{0.68}\\
 & 0.1 & 18.40 & 0.39 & 18.96 & 0.50 & 19.30 & 0.49 & \bf{22.54} & \bf{0.57} \\
 & 0.2 & 13.44 & 0.32 & 13.74 & 0.42 & 13.87 & 0.40 & \bf{20.54} & \bf{0.48}\\
 & 0.3 & 10.17 & 0.26 & 10.35 & 0.36 & 10.44 & 0.35  & \bf{18.62}& \bf{0.43}\\
\hline
\end{tabular}
\caption{PSNR- and SSIM-values of segmented images by CNC, Chan \& Vese, ATCG-TV, and 
LR-CNC algorithms.}
\label{tab 4}
\end{table}

\begin{figure}[H]
\centering
\begin{tabular}{cc}
\includegraphics[width=0.4\linewidth]{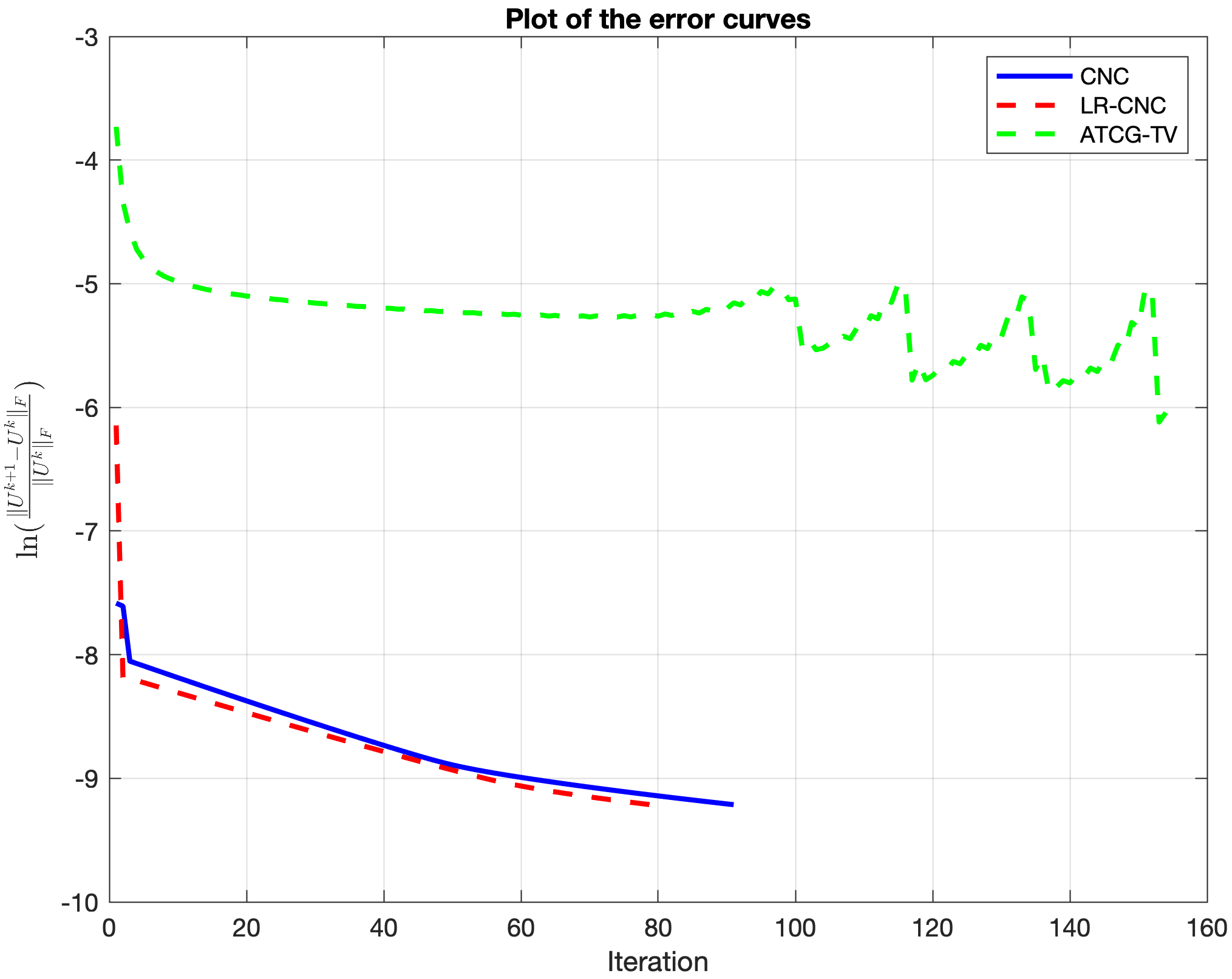}   &
\includegraphics[width=0.4\linewidth]{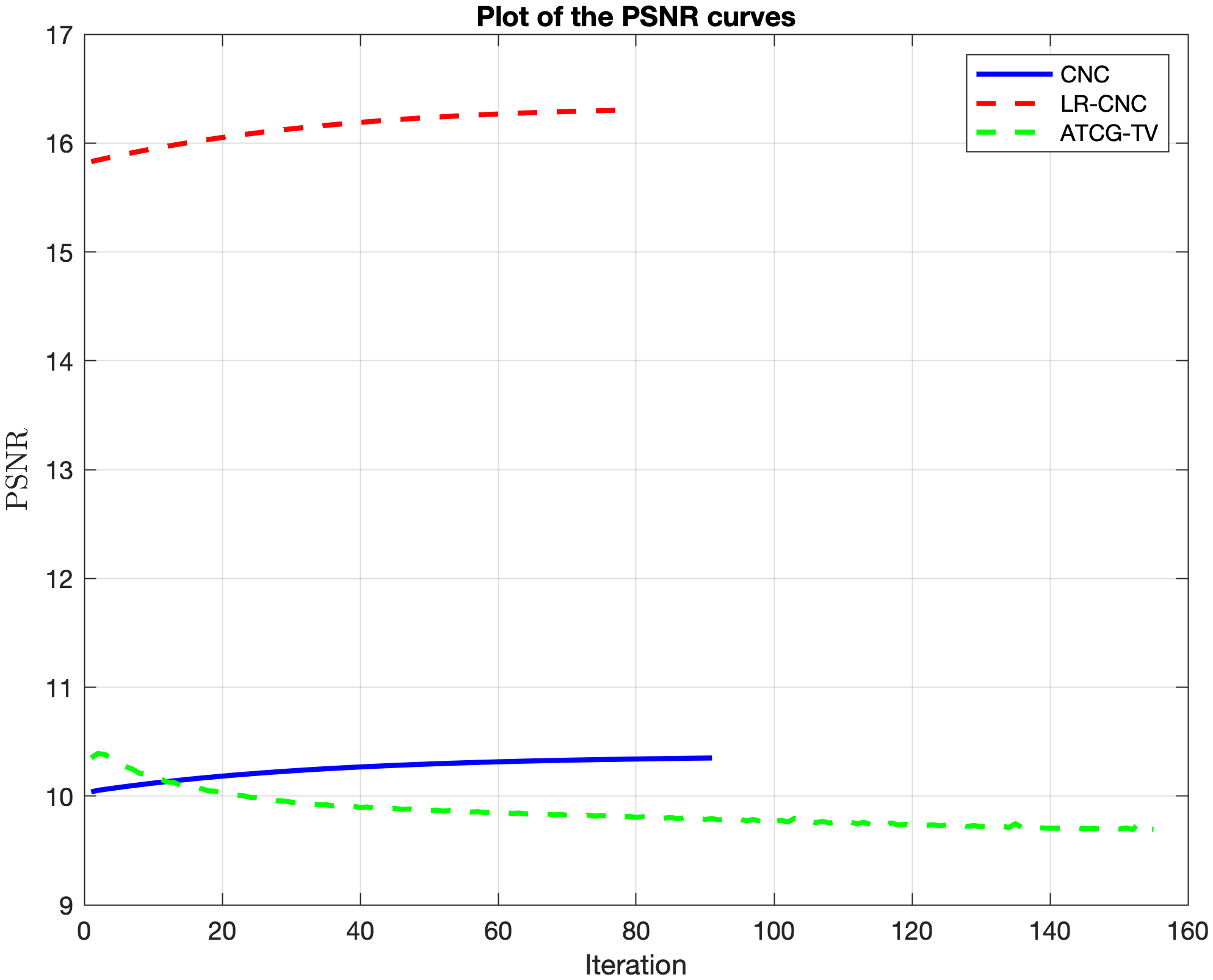}
\end{tabular}
\caption{The relative differences \eqref{logdiff} and PSNR-values as functions of the 
iteration number for the Jasper Ridge image with Gaussian noise level $0.3$.}
\label{fig 11}
\end{figure}

Table \ref{tab 4} illustrates that, at a low Gaussian noise level of $0.01$, the 
performance of all algorithms is comparable, with only slight differences in achieved 
PSNR- and SSIM-values. However, for noise levels $0.1$ and larger, the LR-CNC algorithm 
yields segmentations with significantly larger PSNR- and SSIM-values. Furthermore, the 
graphs of Figures \ref{fig 10} and \ref{fig 11} demonstrate the convergence of the LR-CNC
algorithm to be more regular and faster.

\section{Conclusion}\label{sec 5}
In this work, we proposed a novel approach to consider convex-non-convex optimization 
problems with focus on promoting low-rank solutions. By combining low-rank regularization 
with a non-convex penalty, the method achieves reconstructions that are both accurate and 
computationally efficient. This is supported by measured quality metrics. The optimization
problems are solved with the Alternating Direction Method of Multipliers (ADMM), and we 
provide an analysis that establishes the convergence properties of the LR-CNC algorithm. 
Extensive numerical experiments illustrate the effectiveness of this algorithm. The LR-CNC
framework balances the robustness of convex models and the flexibility of non-convex 
formulations. This makes the LR-CNC algorithm a promising tool for image segmentation and
completion tasks.\\

The authors declare that there is no conflict of interest.

%
\bibliographystyle{model1b-num-names}

\begin{thebibliography}{10}
	
\bibitem{BR}
J. Baglama, L. Reichel, Augmented implicitly restarted Lanczos bidiagonalization methods,
SIAM Journal on Scientific Computing, 27, 19--42 (2005).
	
\bibitem{Amir_denoising}
A. Beck, M. Teboulle, Fast gradient-based algorithms for constrained total variation image
denoising and deblurring problems, IEEE Transactions on Image Processing, 18(11) 
2419--2434 (2009).
	
\bibitem{Completion}
R. Bell, Y. Koren, C. Volinsky, Matrix factorization techniques for recommender systems, 
Computer, 42(8), 30-37 (2009).
	
\bibitem{Benchatou}
O. Benchettou, A. H. Bentbib, A. Bouhamidi, K. Kreit, Constrained tensorial total 
variation problem based on an alternating conditional gradient algorithm, Journal of 
Computational and Applied Mathematics, 451, Art. 116018 (2024).
	
\bibitem{bentbib1}
A. H. Bentbib, M. El Guide, K. Jbilou, E. Onunwor, L. Reichel, Solution methods for linear
discrete ill-posed problems for color image restoration, BIT Numerical Mathematics 58, 
555--576 (2018).
	
\bibitem{elha2}
A. H. Bentbib, A. El Hachimi, K. Jbilou, A. Ratnani, Fast multidimensional completion and
principal component analysis methods via the cosine product, Calcolo, 59(3), Art. 26 
(2022).
	
\bibitem{elha1}
A. H. Bentbib, A. El Hachimi, K. Jbilou, A. Ratnani, A tensor regularized nuclear norm 
method for image and video completion, Journal of Optimization Theory and Applications, 
192(2), 401--425 (2022).

\bibitem{Bj1}
\AA. Bj\"orck, Solving linear least squares problems by Gram-Schmidt orthogonalization,
BIT Numerical Mathematics, 7, 1--210 (1967).

\bibitem{Bj2}
\AA. Bj\"orck, Stability analysis of the method of semi-normal equations for least squares
problems, Linear Algebra and Its Applications, 88/89, 31--48 (1987).

\bibitem{Bj3}
\AA. Bj\"orck, A bidiagonalization algorithm for solving ill-posed systems of linear 
equations, BIT Numerical Mathematics, 28, 659--670 (1988).

\bibitem{Bj4}
\AA. Bj\"orck, Numerical Methods for Least Squares Problems, 2nd ed, SIAM, Philadelphia,
2024.

\bibitem{Bu}
A. Buccini, P. Dell'Acqua, M. Donatelli, A general framework for the acceleration of ADMM,
Numerical Algorithms, 85, 829--848 (2020).
	
\bibitem{MFazel}
S. Boyd, M. Fazel, H. Hindi, A rank minimization heuristic with application to minimum 
order system approximation, Proceedings of the 2001 American Control Conference. (Cat. No.
01CH37148), 6, 4734--4739 (2001).
	
\bibitem{nuclear_norm_prox}
J.-F. Cai, E. J. Cand{\`e}s, Z. Shen, A singular value thresholding algorithm for matrix 
completion, SIAM Journal on Optimization, 20(4), 1956--1982 (2010).
	
\bibitem{two-stage Mumford--Shah model}
X. Cai, R. Chan, T. Zeng, A two-stage image segmentation method using a convex variant of 
the Mumford--Shah model and thresholding, SIAM Journal on Imaging Sciences, 6(1) 368--390 
(2013).
	
\bibitem{Completion_nuclear_norm}
E. J. Cand{\`e}s, B.  Recht, Exact matrix completion via convex optimization, Foundations of 
Computational Mathematics, 9(6), 717--772 (2009).
	
\bibitem{CNC}
R. Chan, A. Lanza, S. Morigi, F. Sgallari, Convex non-convex image segmentation, 
Numerische Mathematik, 138, 635--680 (2018).
	
\bibitem{T.Chan}
T. Chan, L. A. Vese, Active contours without edges, IEEE Transaction on Image 
Processing, 10, 266--277 (2001).
	
\bibitem{HI_denoising2}
Y. Chang, L. Yan, H. Fang, Hyperspectral image denoising using nonlocal low-rank tensor 
decomposition, IEEE Journal of Selected Topics in Applied Earth Observations and Remote 
Sensing, 10(5), 2037--2047 (2017).
	
\bibitem{Clarke}
F. H. Clarke, Optimization and Nonsmooth Analysis, SIAM, Philadelphia, 1990.
	
\bibitem{MRI_compression}
D. Donoho, M. Lustig, J. M. Pauly, Sparse MRI: The application of compressed sensing for 
rapid MR imaging, Magnetic Resonance in Medicine, 58(6), 1182--1195 (2007).
	
\bibitem{elha_NFT}
A. El Hachimi, K. Jbilou, A. Ratnani, Non-negative Einstein tensor factorization for 
unmixing hyperspectral images, Numerical Algorithms (2025). doi.org/10.1007/s11075-025-02011-1

\bibitem{El1}
L. Eld\'en, Algorithms for the regularization of ill-conditioned least squares problems,
BIT Numerical Mathematics, 17, 134--145 (1977).

\bibitem{El2}
L. Eld\'en, Algorithms for the regularization of ill-conditioned least squares problems,
BIT Numerical Mathematics, 17, 134--145 (1977).

\bibitem{El3}
L. Eld\'en, Algorithms for the computation of functionals defined on the solution of a 
discrete ill-posed problem, BIT Numerical Mathematics, 30, 466--483 (1990).

\bibitem{El4}
L. Eld\'en, Matrix Methods in Data Mining and Pattern Recognition, 2nd ed., SIAM,
Philadelphia, 2019.
	
\bibitem{Espindola}
G. Espindola, G. C\^amara, I. Reis, L. Bins, A. Monteiro, Parameter selection for 
region-growing image segmentation algorithms using spatial autocorrelation. International
Journal on Remote Sensing, 27, 3035--3040 (2006).
	
	\bibitem{MRI_2}
	J. A. Fessler, Optimization methods for magnetic resonance image reconstruction: Key 
	models and optimization algorithms, IEEE Signal Processing Magazine, 37(1) 33--40 (2020).
	
	\bibitem{ADMM1}
	M. R. Hestenes, Multiplier and gradient methods. Journal of Optimization Theory and 
	Applications, 5(4), 303--320 (1969).
	
	\bibitem{Iko}
	A. M. Ikotun, A. E. Ezugwu, L. Abualigah, B Abuhaija, J. Heming, $K$-means clustering
	algorithms: a comprehensive review, variants analysis, and advances in the era of big
	data, Information Sciences, 622, 178--210 (2023).
	
	\bibitem{LMRS}
	A. Lanza, S. Morigi, L. Reichel, F. Sgallari, A generalized Krylov subspace method for
	$\ell_p$-$\ell_q$ minimization, SIAM Journal on Scientific Computing, 37, S30--S50 (2015).
	
	\bibitem{LPS}
	A. Lanza, M. Pragliola, F. Sgallari, Parameter-free restoration of piecewise smooth 
	images, Electronic Transactions on Numerical Analysis, 59, 202--229 (2023).
	
	\bibitem{Long}
	J. Long, E. Shelhamer, T. Darrell, Fully convolutional networks for semantic 
	segmentation, Proceedings of the IEEE Conference on Computer Vision and Pattern 
	Recognition (CVPR), 3431--3440 (2015).
	
	\bibitem{Lu}
	C. Lu, J. Tang, S. Yan, Z. Lin, Nonconvex nonsmooth low-rank minimization via 
	iteratively reweighted nuclear norm, IEEE Transactions on Image Processing, 25(2), 
	829--839 (2015).
	
	\bibitem{Mumford--Shah model}
	D. Mumford, J. Shah, Optimal approximations by piecewise smooth functions and 
	associated variational problems, Communications on Pure and Applied Mathematics, 42(5) 
	577--685 (1989).
	
	\bibitem{ADMM2}
	M. J. D. Powell, A method for nonlinear constraints in minimization problems. In R. 
	Fletcher (ed.), Optimization, Academic Press, London, New York, 1969, pp. 283--298.
	
	\bibitem{Fermat}
	R. T. Rockafellar,  R. J. B. Wets, Variational Analysis, 3rd printing, Springer, Berlin,
	2010.  
	
	\bibitem{Rudin}
	L. I. Rudin, S. Osher, E. Fatemi, Nonlinear total variation based noise removal 
	algorithms, Physica D: Nonlinear Phenomena, 60(1-4), 259--268 (1992).
	
	\bibitem{SSIM}
	Z. Wang, A. C. Bovik, H. R. Sheikh, E. P. Simoncelli, Image quality assessment: From error
	visibility to structural similarity, IEEE Transactions on Image Processing, 13(4), 600?612
	(2004).
	
	\bibitem{HI_denoising1}
	Q. Zhang, X. Li, Hyperspectral image denoising with a spatial-spectral sparse 
	representation, IEEE Transactions on Geoscience and Remote Sensing, 52(10), 6753--6764 
	(2014).
	
\end{thebibliography}

\begin{appendices}

\section{Proof of Proposition \ref{penultprop}}. \label{app A}
From Theorem \ref{thm3.12}, we have $U^*=Z^*$ and $DU^*=M^*$. Therefore,
\[
\widetilde{Q}^{k+1}=\widetilde{Q}^{k} +\beta_1 \left(D\widetilde{U}^{k+1} - 
\widetilde{M}^{k+1}\right),\quad 
\widetilde{O}^{k+1}=\widetilde{O}^k + \beta_2 \left( \widetilde{U}^{k+1}-
\widetilde{Z}^{k+1} \right).
\]
Using \eqref{innerprod} we obtain the equations
\begin{eqnarray*}
	\begin{cases}
		\left\Vert\widetilde{Q}^{k}\right\Vert_F^2-\left\Vert\widetilde{Q}^{k+1}\right\Vert_F^2=
		-2\beta_1\langle \widetilde{Q}^k, D\widetilde{U}^{k+1} - \widetilde{M}^{k+1}\rangle -
		\beta_1^2 \left\Vert D\widetilde{U}^{k+1} - \widetilde{M}^{k+1}\right\Vert _F^2,\\
		\left\Vert\widetilde{O}^{k}\right\Vert_F^2-\left\Vert\widetilde{O}^{k+1} \right\Vert_F^2=
		-2\beta_2\langle \widetilde{O}^k, \widetilde{U}^{k+1} - \widetilde{Z}^{k+1}\rangle -
		\beta_2^2 \left\Vert \widetilde{U}^{k+1} - \widetilde{Z}^{k+1}\right\Vert _F^2.
	\end{cases}
\end{eqnarray*}
Adding these equations, we get
\begin{eqnarray*}
	&&\dfrac{1}{2\beta_1}\left( \left\Vert \widetilde{Q}^{k} \right\Vert_F^2- \left\Vert 
	\widetilde{Q}^{k+1} \right\Vert_F^2 \right) + \dfrac{1}{2\beta_1}\left( \left\Vert 
	\widetilde{O}^{k} \right\Vert_F^2- \left\Vert \widetilde{O}^{k+1} \right\Vert_F^2 \right) \\
	&&= -\langle \widetilde{Q}^k, D\widetilde{U}^{k+1} - \widetilde{M}^{k+1}\rangle - 
	\dfrac{\beta_1}{2} \left\Vert D\widetilde{U}^{k+1} - \widetilde{M}^{k+1}\right\Vert _F^2
	-\langle \widetilde{O}^k, \widetilde{U}^{k+1} - \widetilde{Z}^{k+1}\rangle - 
	\dfrac{\beta_2}{2} \left\Vert \widetilde{U}^{k+1} - \widetilde{Z}^{k+1}\right\Vert _F^2.
\end{eqnarray*}
At the $k$th step with $k\geq 1$, Algorithm \ref{Alg 1} yields relations that are 
analogous to \eqref{condition optimality U}, \eqref{condition optimality Z}, and 
\eqref{condition optimality M}, namely, 
\begin{eqnarray}\label{condition optimality Uk}
	& &F(U)-F(U^{k+1}) -\dfrac{\beta_1'}{2}\left\Vert DU -M^{k}\right\Vert_F^2\\
	&+&\dfrac{\beta_1'}{2}\left\Vert DU^{k+1} -M^{k}\right\Vert_F^2 - 
	\dfrac{\beta_2'}{2}\left\Vert Z^{k} -U \right\Vert_F^2 + 
	\dfrac{\beta_2'}{2}\left\Vert Z^{k} -U^{k+1} \right\Vert_F^2\nonumber\\
	&+& \langle D^TQ^{k} + O^{k} + (\beta_1+\beta_1')(D^TDU^{k+1} -D^TM^{k}) + 
	(\beta_2 + \beta_2')(U^{k+1} -Z^{k}), U-U^{k+1}\rangle \geq 0,\;
	\forall U\in \mathbb{R}^{n_1 \times n_2},\nonumber \\[2mm]
	\label{condition optimality Zk}
	& & \left\Vert Z \right\Vert_* - \left\Vert Z^{k+1}\right\Vert_* + 
	\dfrac{\beta_2''}{2}\left\Vert Z-U^{k+1} \right\Vert_F^2 - 
	\dfrac{\beta_2''}{2}\left\Vert Z^{k+1}-U^{k+1} \right\Vert_F^2 \\
	&+&\langle -O^{k} +(\beta_2 -\beta_2'')\left(Z^{k+1}-U^{k+1}\right),Z-Z^{k+1}\rangle
	\geq 0,\; \forall Z\in \mathbb{R}^{n_1 \times n_2},\nonumber \\[2mm]
	\label{condition optimality Mk}
	& & R\left( M \right) - R\left( M^{k+1} \right) + \dfrac{\beta_1''}{2}
	\left\Vert DU^{k+1}-M \right\Vert_F^2 - 
	\dfrac{\beta_1''}{2}\left\Vert DU^{k+1}-M^{k+1} \right\Vert_F^2 \\
	&+& \langle -Q^{k}
	+(\beta_1 -\beta_1'')\left(M^{k+1}-DU^{k+1}\right),M-M^{k+1}\rangle \geq 0,\; 
	\forall M\in \mathbb{R}^{2n_1 \times n_2},\nonumber
\end{eqnarray}
where $\beta_1',\, \beta_2',\, \beta_1''$, and $\beta_2''$ satisfy the relations
\eqref{cond 3.10}.

Substituting $U=U^{k+1}$ into \eqref{condition optimality U} and $U=U^*$ into
\eqref{condition optimality Uk}, and adding the equations so obtained, we get
\begin{eqnarray}\label{eneq 1}
	\beta_1\langle \widetilde{M}^{k}, D\widetilde{U}^{k+1} \rangle + 
	\beta_2 \langle \widetilde{Z}^{k},\widetilde{U}^{k+1} \rangle - 
	\langle \widetilde{Q}^{k}, D\widetilde{U}^{k+1}\rangle - 
	\langle \widetilde{O}^{k}, \widetilde{U}^{k+1}\rangle\\ 
	-(\beta_1 +\beta_1')\left\Vert D\widetilde{U}^{k+1} \right\Vert_F^2 - 
	(\beta_2 +\beta_2')\left\Vert \widetilde{U}^{k+1} \right\Vert_F^2 \geq 0.\nonumber
\end{eqnarray}
Similarly, we substitute $Z=Z^{k+1}$ into \eqref{condition optimality Z}, $Z=Z^*$ into
\eqref{condition optimality Zk}, $M=M^{k+1}$ into \eqref{condition optimality Z}, and 
$M=M^*$ into \eqref{condition optimality Zk}, and obtain
\begin{eqnarray}
	\label{eneq 2}
	&&-(\beta_2 -\beta_2'')\left\Vert \widetilde{Z}^{k+1} \right\Vert_F^2 + 
	\beta_2\langle \widetilde{U}^{k+1},\widetilde{Z}^{k+1} \rangle + 
	\langle \widetilde{O}^{k}, \widetilde{Z}^{k+1}\rangle \geq 0,\\
	\label{eneq 3}
	&& -(\beta_1 -\beta_1'')\left\Vert \widetilde{M}^{k+1} \right\Vert_F^2 + 
	\beta_1\langle D\widetilde{U}^{k+1},\widetilde{M}^{k+1} \rangle + 
	\langle \widetilde{Q}^{k}, \widetilde{M}^{k+1}\rangle \geq 0.
\end{eqnarray}
Adding \eqref{eneq 1} to \eqref{eneq 2} gives
\begin{eqnarray}\label{eneq 3b}
	& & \langle \widetilde{O}^{k},\widetilde{U}^{k+1}-\widetilde{Z}^{k+1} \rangle + 
	\beta_2\langle \widetilde{Z}^{k}-\widetilde{Z}^{k+1},\widetilde{U}^{k+1}\rangle+
	\beta_1\langle \widetilde{M}^{k},D\widetilde{U}^{k+1}\rangle - 
	\langle \widetilde{Q}^{k},D\widetilde{U}^{k+1}\rangle  \\
	&-& \left((\beta_1+\beta_1')\left\Vert D\widetilde{U}^{k+1} \right\Vert_F^2 + 
	(\beta_2+\beta_2')\left\Vert \widetilde{U}^{k+1} \right\Vert_F^2 + (\beta_2 
	-\beta_2')\left\Vert \widetilde{Z}^{k+1} \right\Vert_F^2 
	-2\beta_2 \langle\widetilde{U}^{k+1},\widetilde{Z}^{k+1}\rangle \right)\geq 0.\nonumber
\end{eqnarray}
Moreover, addition of \eqref{eneq 1}, \eqref{eneq 3}, and \eqref{eneq 3b} yields
\begin{eqnarray}\label{eq 3.22}
	&-&\langle \widetilde{Q}^{k},D\widetilde{U}^{k+1}-\widetilde{M}^{k+1} \rangle + 
	\beta_1\langle D\widetilde{U}^{k+1},\widetilde{M}^{k}-\widetilde{M}^{k+1} \rangle - 
	\dfrac{\beta_1+\beta_3}{2}\left\Vert D\widetilde{U}^{k+1}-
	\widetilde{M}^{k+1}\right\Vert_F^2\\
	&-&\langle \widetilde{O}^{k}, \widetilde{U}^{k+1}-\widetilde{Z}^{k+1} \rangle + 
	\beta_2 \langle \widetilde{U}^{k+1},\widetilde{Z}^{k}-\widetilde{Z}^{k+1} \rangle - 
	\dfrac{\beta_2+\beta_4}{2}\left\Vert\widetilde{U}^{k+1}-\widetilde{Z}^{k+1}\right\Vert_F^2
	\nonumber\\
	&-&\left( 
	\left( \dfrac{\beta_1}{2}-\dfrac{\beta_3}{2}+\beta_1' \right)\left\Vert D
	\widetilde{U}^{k+1}\right\Vert_F^2 + \left( \beta_3-\beta_1 \right)
	\langle D\widetilde{U}^{k+1},\widetilde{M}^{k+1} \rangle + 
	\left( \dfrac{\beta_1}{2}-\dfrac{\beta_3}{2}-\beta_1''\right)\left\Vert 
	\widetilde{M}^{k+1}\right\Vert_F^2 \right)\nonumber\\
	&-&\left( 
	\left(\dfrac{\beta_2}{2}-\dfrac{\beta_4}{2}+\beta_2'\right)\left\Vert 
	\widetilde{U}^{k+1}\right\Vert_F^2 + \left( \beta_4 -\beta_2 \right) \langle 
	\widetilde{U}^{k+1},\widetilde{Z}^{k+1} \rangle + \left( \dfrac{\beta_2}{2}-
	\dfrac{\beta_4}{2}-\beta_2'' \right)\left\Vert \widetilde{Z}^{k+1} \right\Vert_F^2
	\right)\geq 0,\nonumber
\end{eqnarray}
where $\beta_3,\beta_4>0$. 

We would like to show that the third and the fourth terms of the above inequality are 
equal to, respectively, 
\[
-\left\Vert c_1 D\widetilde{U}^{k+1}-c_2\widetilde{M}^{k+1} \right\Vert_F^2 
\mbox{~~~and~~~} 
-\left\Vert c_3 \widetilde{U}^{k+1}-c_4 \widetilde{Z}^{k+1} \right\Vert_F^2,
\]
for some nonnegative coefficients $c_1,c_2,c_3,c_4$. This requires that the coefficients 
of $\left\Vert D\widetilde{U}^{k+1} \right\Vert_F^2$, 
$\left\Vert \widetilde{M}^{k+1}\right\Vert_F^2$, 
$\left\Vert \widetilde{U}^{k+1} \right\Vert_F^2$, and 
$\left\Vert \widetilde{Z}^{k+1} \right\Vert_F^2$ are nonnegative, i.e., that
\begin{equation}\label{eq beta''}
	\beta_1'\geq \dfrac{\beta_3}{2}-\dfrac{\beta_1}{2},~~ 
	\beta_1''\leq \dfrac{\beta_1}{2}-\dfrac{\beta_3}{2},~~ 
	\beta_2'\geq \dfrac{\beta_4}{2}-\dfrac{\beta_2}{2},~~ 
	\beta_2''\leq \dfrac{\beta_2}{2}-\dfrac{\beta_4}{2}.
\end{equation}
By also considering the conditions \eqref{cond 3.10} on $\beta_1',\,\beta_2',\,\beta_1''$, 
and $\beta_2''$, we obtain 
\[
0\leq \beta_2''<\dfrac{\beta_2}{2}-\dfrac{\beta_4}{2},\, a\leq\beta_1''<
\dfrac{\beta_1}{2}-\dfrac{\beta_3}{2},\, 4\beta_3 -4\beta_1 +\beta_4 -\beta_2 \leq
8\beta_1'+\beta_2\leq \lambda.
\]
It follows from the first and the second inequalities that 
\[
\beta_2>\beta_4,~~ \beta_1>2a +\beta_3.
\]
We also have the relations
\[
(\beta_3-\beta_1)=2c_1c_2,~~ (\beta_2-\beta_4)=2c_3c_4,
\]
which give
\[
(\beta_1-\beta_3)=2\sqrt{\left(\dfrac{\beta_1}{2}-\dfrac{\beta_3}{2}+
	\beta_1'\right)\left(\dfrac{\beta_1}{2}-\dfrac{\beta_3}{2}-\beta_1''\right)},~~
(\beta_2-\beta_4)=2\sqrt{\left(\dfrac{\beta_2}{2}-\dfrac{\beta_4}{2}+
	\beta_2'\right)\left(\dfrac{\beta_2}{2}-\dfrac{\beta_4}{2}-\beta_2''\right)}.
\]
We obtain
\[
\beta_1=\beta_3- 2\dfrac{\beta_1'\beta_1''}{\beta_1''-\beta_1'},\; 
\beta_2=\beta_4-2\dfrac{\beta_2'\beta_2''}{\beta_2''-\beta_2'}.
\]
Since $\beta_1-\beta_3>2\beta_1''$ and $\beta_2-\beta_4>2\beta_2''$, it follows that we 
can have a solution only if $\beta_1'>\beta_1''$ and $\beta_2'>\beta_2''$.  Imposing the 
conditions
\[
\beta_1''=\dfrac{a}{\rho_1-1},\quad \beta_1'=\dfrac{\rho_1a}{\rho_1-1},\quad 
\beta_2''=\dfrac{\lambda-8a}{\rho_2-1},\quad 
\beta_2'=\dfrac{\rho_2(\lambda-8a)}{\rho_2-1},
\]
we get
\[
\beta_1=\beta_3+2\dfrac{a\rho_1}{(\rho_1-1)^2}, \quad 
\beta_2=\beta_4 +2\dfrac{\rho_2(\lambda-8a)}{(\rho_2-1)^2}.
\]

We now investigate if $\beta_3$ and $\beta_4$ admit some solutions. Considering the above 
two equations and the inequalities \eqref{cond 3.9}, we have 
\begin{equation}\label{eq beta3beta4}
	\beta_3\geq \dfrac{a}{(\rho_1-1)}\dfrac{(-\rho_1-1)}{(\rho_1-1)},\quad
	\beta_4\leq \dfrac{\rho_2(\lambda-8a)}{(\rho_2-1)^2}\left(\rho_2-3\right).
\end{equation}
Since $\beta_1$ and $\beta_2$ are positive, $\beta_1$ accepts solutions for every 
$\beta_3>0$ and $\beta_2$ accepts solutions only if
\[
\rho_2>3.
\]
Letting
\[
\beta_1''=\dfrac{a}{\rho_1-1},\quad \beta_1'=\dfrac{\rho_1a}{\rho_1-1},\quad 
\beta_2''=\dfrac{\lambda-8a}{\rho_2-1},\quad 
\beta_2'=\dfrac{\rho_2(\lambda-8a)}{\rho_2-1},
\]
we find that the third and the fourth terms of \eqref{eq 3.22} can be written as 
\[
-\left\Vert c_1 D\widetilde{U}^{k+1}-c_2\widetilde{M}^{k+1} \right\Vert_F^2,\qquad
-\left\Vert c_3 \widetilde{U}^{k+1}-c_4 \widetilde{Z}^{k+1} \right\Vert_F^2
\]
with
\begin{eqnarray}\label{eq c}
	c_1=\sqrt{\dfrac{\beta_1-\beta_3}{2}+\dfrac{\rho_1 a}{(\rho_1-1)}},~~
	c_2=\sqrt{\dfrac{\beta_1-\beta_3}{2}-\dfrac{a}{(\rho_1-1)}},\\
	c_3=\sqrt{\dfrac{\beta_2-\beta_4}{2}+
		\dfrac{\rho_2(\lambda -8a)}{(\rho_2-1)}},~~ c_4=\sqrt{\dfrac{\beta_2-\beta_4}{2}-
		\dfrac{(\lambda -8a)}{(\rho_2-1)}}.
\end{eqnarray}
It follows that $c_1>c_2$ and $c_3>c_4$. We will use these inequalities below.

Inequality \eqref{eq 3.22} can be written as
\begin{eqnarray*}
	&-&\langle \widetilde{Q}^{k},D\widetilde{U}^{k+1}-\widetilde{M}^{k+1} \rangle + 
	\beta_1\langle D\widetilde{U}^{k+1},\widetilde{M}^{k}-\widetilde{M}^{k+1} \rangle - 
	\dfrac{\beta_1+\beta_3}{2}\left\Vert D\widetilde{U}^{k+1}-
	\widetilde{M}^{k+1}\right\Vert_F^2\\
	&-&\langle \widetilde{O}^{k}, \widetilde{U}^{k+1}-\widetilde{Z}^{k+1} 
	\rangle + \beta_2 \langle \widetilde{U}^{k+1},\widetilde{Z}^{k}-\widetilde{Z}^{k+1} 
	\rangle - \dfrac{\beta_2+\beta_4}{2}\left\Vert \widetilde{U}^{k+1}-\widetilde{Z}^{k+1}
	\right\Vert_F^2\\
	&-&\left\Vert c_1 D\widetilde{U}^{k+1}-c_2\widetilde{M}^{k+1} \right\Vert_F^2-\left\Vert 
	c_3 \widetilde{U}^{k+1}-c_4 \widetilde{Z}^{k+1} \right\Vert_F^2\geq 0,
\end{eqnarray*}
which is equivalent to
\begin{eqnarray*}
	&-&\langle \widetilde{Q}^{k},D\widetilde{U}^{k+1}-\widetilde{M}^{k+1} \rangle  - 
	\dfrac{\beta_1+\beta_3}{2}\left\Vert D\widetilde{U}^{k+1}-\widetilde{M}^{k+1}
	\right\Vert_F^2 -\langle \widetilde{O}^{k}, \widetilde{U}^{k+1}-\widetilde{Z}^{k+1} 
	\rangle\\
	&-& \dfrac{\beta_2+\beta_4}{2}\left\Vert \widetilde{U}^{k+1}-\widetilde{Z}^{k+1}
	\right\Vert_F^2\\
	&\geq&\left\Vert c_1 D\widetilde{U}^{k+1}-c_2\widetilde{M}^{k+1} \right\Vert_F^2+
	\left\Vert c_3 \widetilde{U}^{k+1}-c_4 \widetilde{Z}^{k+1} \right\Vert_F^2 -\beta_1
	\langle D\widetilde{U}^{k+1},\widetilde{M}^{k}-\widetilde{M}^{k+1} \rangle - 
	\beta_2\langle \widetilde{U}^{k+1},\widetilde{Z}^{k}-\widetilde{Z}^{k+1}\rangle.
\end{eqnarray*}
This implies 
\begin{eqnarray*}
	&-&\langle \widetilde{Q}^{k},D\widetilde{U}^{k+1}-\widetilde{M}^{k+1} \rangle  - 
	\dfrac{\beta_1}{2}\left\Vert D\widetilde{U}^{k+1}-\widetilde{M}^{k+1}\right\Vert_F^2
	-\langle \widetilde{O}^{k}, \widetilde{U}^{k+1}-\widetilde{Z}^{k+1} \rangle  - 
	\dfrac{\beta_2}{2}\left\Vert \widetilde{U}^{k+1}-\widetilde{Z}^{k+1}\right\Vert_F^2\\
	&\geq&\left\Vert c_1 D\widetilde{U}^{k+1}-c_2\widetilde{M}^{k+1} \right\Vert_F^2+
	\left\Vert c_3 \widetilde{U}^{k+1}-c_4 \widetilde{Z}^{k+1} \right\Vert_F^2 -
	\beta_1\langle D\widetilde{U}^{k+1},\widetilde{M}^{k}-\widetilde{M}^{k+1} \rangle - 
	\beta_2 \langle \widetilde{U}^{k+1},\widetilde{Z}^{k}-\widetilde{Z}^{k+1} \rangle \\
	&+& \dfrac{\beta_3}{2}\left\Vert D\widetilde{U}^{k+1}-\widetilde{M}^{k+1}\right\Vert_F^2 +
	\dfrac{\beta_4}{2}\left\Vert \widetilde{U}^{k+1}-\widetilde{Z}^{k+1}\right\Vert_F^2,
\end{eqnarray*}
which gives
\begin{eqnarray*}
	& & \dfrac{1}{2\beta_1} \left(\left\Vert \widetilde{Q}^{k} \right\Vert_F^2 - 
	\left\Vert \widetilde{Q}^{k+1} \right\Vert_F^2\right)+\dfrac{1}{2\beta_2} 
	\left(\left\Vert \widetilde{O}^{k} \right\Vert_F^2 - \left\Vert \widetilde{O}^{k+1} 
	\right\Vert_F^2\right)\\
	&\geq&\left\Vert c_1 D\widetilde{U}^{k+1}-c_2\widetilde{M}^{k+1} \right\Vert_F^2+
	\left\Vert c_3 \widetilde{U}^{k+1}-c_4 \widetilde{Z}^{k+1} \right\Vert_F^2 -
	\beta_1\langle D\widetilde{U}^{k+1},\widetilde{M}^{k}-\widetilde{M}^{k+1} \rangle - 
	\beta_2 \langle \widetilde{U}^{k+1},\widetilde{Z}^{k}-\widetilde{Z}^{k+1} \rangle\\
	&+& \dfrac{\beta_3}{2}\left\Vert D\widetilde{U}^{k+1}-\widetilde{M}^{k+1}\right\Vert_F^2 +
	\dfrac{\beta_4}{2}\left\Vert \widetilde{U}^{k+1}-\widetilde{Z}^{k+1}\right\Vert_F^2.
\end{eqnarray*}
This concludes the proof.
\vskip5pt


\section{Proof of Proposition \ref{lastprop}}.
\label{app B}
According Proposition \ref{penultprop}, we have 
\begin{eqnarray*}
	& & \dfrac{1}{2\beta_1} \left(\left\Vert \widetilde{Q}^{k} \right\Vert_F^2 - 
	\left\Vert \widetilde{Q}^{k+1} \right\Vert_F^2\right)+\dfrac{1}{2\beta_2} 
	\left(\left\Vert \widetilde{O}^{k} \right\Vert_F^2 - 
	\left\Vert \widetilde{O}^{k+1} \right\Vert_F^2\right)\\
	&\geq&\left\Vert c_1 D\widetilde{U}^{k+1}-c_2\widetilde{M}^{k+1} \right\Vert_F^2+
	\left\Vert c_3 \widetilde{U}^{k+1}-c_4 \widetilde{Z}^{k+1} \right\Vert_F^2 -
	\beta_1\langle D\widetilde{U}^{k+1},\widetilde{M}^{k}-\widetilde{M}^{k+1} \rangle - 
	\beta_2 \langle \widetilde{U}^{k+1},\widetilde{Z}^{k}-\widetilde{Z}^{k+1} \rangle,
	\nonumber\\
	&+& \dfrac{\beta_3}{2}\left\Vert D\widetilde{U}^{k+1}-\widetilde{M}^{k+1}\right\Vert_F^2 
	+ \dfrac{\beta_4}{2}\left\Vert \widetilde{U}^{k+1}-\widetilde{Z}^{k+1}\right\Vert_F^2.
	\nonumber
\end{eqnarray*}
We now seek to determine lower bounds for 
$-\beta_1\langle D\widetilde{U}^{k+1},\widetilde{M}^{k}-\widetilde{M}^{k+1} \rangle$ and
$ - \beta_2 \langle \widetilde{U}^{k+1},\widetilde{Z}^{k}-\widetilde{Z}^{k+1}\rangle$. We
have 
\begin{eqnarray*}
	\langle \widetilde{M}^{k}-\widetilde{M}^{k+1} ,D\widetilde{U}^{k+1}\rangle&=& 
	\langle \widetilde{M}^{k}-\widetilde{M}^{k+1}, D\widetilde{U}^{k+1} -
	D\widetilde{U}^{k}\rangle +  \langle \widetilde{M}^{k}-\widetilde{M}^{k+1}, 
	D\widetilde{U}^{k} - \widetilde{M}^{k}\rangle\\
	&+&\langle \widetilde{M}^{k}-\widetilde{M}^{k+1}, \widetilde{M}^{k}\rangle
\end{eqnarray*}
and 
\begin{eqnarray*}
	\langle \widetilde{Z}^{k}-\widetilde{Z}^{k+1},\widetilde{U}^{k+1}\rangle &=& 
	\langle \widetilde{Z}^{k}-\widetilde{Z}^{k+1}, \widetilde{U}^{k+1}- \widetilde{U}^{k} 
	\rangle +  \langle \widetilde{Z}^{k}-\widetilde{Z}^{k+1}, \widetilde{U}^{k}- 
	\widetilde{Z}^{k} \rangle\\
	&+&  \langle \widetilde{Z}^{k}-\widetilde{Z}^{k+1}, \widetilde{Z}^{k} \rangle.
\end{eqnarray*}
Moreover,
\begin{equation*}
	\begin{cases}
		\langle \widetilde{M}^{k}-\widetilde{M}^{k+1}, \widetilde{M}^{k}\rangle=
		\dfrac{1}{2}\left( \left\Vert\widetilde{M}^{k} \right\Vert_F^2 - 
		\left\Vert \widetilde{M}^{k+1}\right\Vert_F^2 - \left\Vert \widetilde{M}^{k+1}-
		\widetilde{M}^{k} \right\Vert_F^2 \right),\\
		\langle \widetilde{Z}^{k}-\widetilde{Z}^{k+1}, \widetilde{Z}^{k}\rangle=
		\dfrac{1}{2}\left( \left\Vert\widetilde{Z}^{k} \right\Vert_F^2 - \left\Vert 
		\widetilde{Z}^{k+1}\right\Vert_F^2 - \left\Vert \widetilde{Z}^{k+1}-\widetilde{Z}^{k} 
		\right\Vert_F^2 \right).
	\end{cases}
\end{equation*}
From the construction of $M^{k}$, we get
\begin{eqnarray} \label{eq 3.28}
	&&R\left( M \right) - R\left( M^{k} \right) + 
	\dfrac{\beta_1''}{2}\left\Vert DU^{k}-M \right\Vert_F^2 -
	\dfrac{\beta_1''}{2}\left\Vert DU^{k}-M^{k} \right\Vert_F^2 \\
	&+& \langle -Q^{k-1}+(\beta_1-\beta_1'')\left(M^{k}-DU^{k}\right),M-M^{k}\rangle\geq 0,\; 
	\forall M\in \mathbb{R}^{2n_1 \times n_2},\nonumber
\end{eqnarray}
and from the construction of $Z^k$, we obtain 
\begin{eqnarray}\label{eq 3.29}
	&&\left\Vert Z \right\Vert_* - \left\Vert Z^{k}\right\Vert_* + 
	\dfrac{\beta_2''}{2}\left\Vert Z-U^{k} \right\Vert_F^2 - 
	\dfrac{\beta_2''}{2}\left\Vert Z^{k}-U^{k} \right\Vert_F^2 \\
	&+& \langle -O^{k-1} +(\beta_2 -\beta_2'')\left(Z^{k}-U^{k}\right),Z-Z^{k}\rangle 
	\geq 0,\; \forall Z\in \mathbb{R}^{n_1 \times n_2}.\nonumber
\end{eqnarray}
Replacing $M$ with $M^{k+1}$ in \eqref{eq 3.28}, $M$ with $M^{k}$ in 
\eqref{condition optimality Mk}, summing the two inequalities, and proceeding similarly
when replacing $Z$ with $Z^{k+1}$ in \eqref{eq 3.29} and $Z$ with $Z^{k}$ in 
\eqref{condition optimality Zk}, yields
\begin{equation*}
	\langle \widetilde{Q}^{k}-\widetilde{Q}^{k-1},\widetilde{M}^{k+1}-\widetilde{M}^{k}\rangle
	- (\beta_1-\beta_1'')\left\Vert \widetilde{M}^{k+1}-\widetilde{M}^k\right\Vert_F^2 + 
	\beta_1 \langle D\widetilde{U}^{k+1}-D\widetilde{U}^{k},\widetilde{M}^{k+1}-
	\widetilde{M}^{k} \rangle\geq 0
\end{equation*}
and 
\begin{equation*}
	\langle \widetilde{O}^{k}-\widetilde{O}^{k-1},\widetilde{Z}^{k+1}-\widetilde{Z}^{k}\rangle 
	- (\beta_2-\beta_2'')\left\Vert \widetilde{Z}^{k+1}-\widetilde{Z}^k\right\Vert_F^2 + 
	\beta_2 \langle \widetilde{U}^{k+1}-\widetilde{U}^{k},\widetilde{Z}^{k+1}-
	\widetilde{Z}^{k} \rangle\geq 0.
\end{equation*}
Since 
\[
\begin{cases}
	\widetilde{Q}^{k}-\widetilde{Q}^{k-1}=Q^{k}-Q^{k-1}=\beta_1(DU^k-M^k),\\
	\widetilde{O}^{k}-\widetilde{O}^{k-1}=O^{k}-O^{k-1}=\beta_2(U^k-Z^k),
\end{cases}
\]
it follows that
\begin{equation*}
	\begin{cases}
		\langle D\widetilde{U}^k-\widetilde{M}^k,\widetilde{M}^{k+1}-\widetilde{M}^{k} \rangle  +
		\langle D\widetilde{U}^{k+1}-D\widetilde{U}^{k},\widetilde{M}^{k+1}-\widetilde{M}^{k} 
		\rangle\geq \dfrac{(\beta_1-\beta_1'')}{\beta_1}\left\Vert \widetilde{M}^{k+1}-
		\widetilde{M}^k\right\Vert_F^2,\\
		\langle \widetilde{U}^k-\widetilde{Z}^k,\widetilde{Z}^{k+1}-\widetilde{Z}^{k} \rangle + 
		\langle \widetilde{U}^{k+1}-\widetilde{U}^{k},\widetilde{Z}^{k+1}-\widetilde{Z}^{k} 
		\rangle\geq \dfrac{(\beta_2-\beta_2'')}{\beta_2 }\left\Vert \widetilde{Z}^{k+1}-
		\widetilde{Z}^k\right\Vert_F^2.
	\end{cases}
\end{equation*}
Therefore,
\[
-\langle D\widetilde{U}^{k+1},\widetilde{M}^{k}-\widetilde{M}^{k+1} \rangle\geq 
\dfrac{1}{2}\left( -\left\Vert\widetilde{M}^{k} \right\Vert_F^2 + 
\left\Vert \widetilde{M}^{k+1}\right\Vert_F^2 + 
\left\Vert \widetilde{M}^{k+1}-\widetilde{M}^{k} \right\Vert_F^2 \right) 
+\dfrac{(\beta_1-\beta_1'')}{\beta_1}\left\Vert \widetilde{M}^{k+1}-
\widetilde{M}^k\right\Vert_F^2 
\]
and 
\[
-\langle \widetilde{U}^{k+1},\widetilde{Z}^{k}-\widetilde{Z}^{k+1}\rangle\geq  
\dfrac{1}{2}\left( -\left\Vert\widetilde{Z}^{k} \right\Vert_F^2 + 
\left\Vert \widetilde{Z}^{k+1}\right\Vert_F^2 + \left\Vert \widetilde{Z}^{k+1}-
\widetilde{Z}^{k} \right\Vert_F^2 \right) + \dfrac{(\beta_2-\beta_2'')}{\beta_2 }
\left\Vert \widetilde{Z}^{k+1}-\widetilde{Z}^k\right\Vert_F^2.
\]
Consequently,
\begin{eqnarray*}
	& & \dfrac{1}{2\beta_1} \left(\left\Vert \widetilde{Q}^{k} \right\Vert_F - 
	\left\Vert \widetilde{Q}^{k+1} \right\Vert_F\right)+\dfrac{1}{2\beta_2} 
	\left(\left\Vert \widetilde{O}^{k} \right\Vert_F - \left\Vert \widetilde{O}^{k+1} 
	\right\Vert_F\right)\\
	&\geq&\left\Vert c_1 D\widetilde{U}^{k+1}-c_2\widetilde{M}^{k+1} \right\Vert_F^2+
	\left\Vert c_3 \widetilde{U}^{k+1}-c_4 \widetilde{Z}^{k+1} \right\Vert_F^2 +
	\dfrac{\beta_3}{2}\left\Vert D\widetilde{U}^{k+1}-\widetilde{M}^{k+1}\right\Vert_F^2 + 
	\dfrac{\beta_4}{2}\left\Vert \widetilde{U}^{k+1}-\widetilde{Z}^{k+1}\right\Vert_F^2,\\
	&+& \dfrac{\beta_1}{2}\left( -\left\Vert\widetilde{M}^{k} \right\Vert_F^2 + 
	\left\Vert \widetilde{M}^{k+1}\right\Vert_F^2\right) +(\dfrac{3\beta_1}{2}-\beta_1'')
	\left\Vert \widetilde{M}^{k+1}-\widetilde{M}^k\right\Vert_F^2 \\
	&+& \dfrac{\beta_2}{2}\left( -\left\Vert\widetilde{Z}^{k} \right\Vert_F^2 + 
	\left\Vert \widetilde{Z}^{k+1}\right\Vert_F^2\right) + (\dfrac{3\beta_2}{2}-\beta_2'')
	\left\Vert \widetilde{Z}^{k+1}-\widetilde{Z}^{k} \right\Vert_F^2.  
\end{eqnarray*}
This completes the proof.
\vskip5pt

\section{Proof of Theorem \ref{thmcnv}} 
\label{app C}
It follows from Proposition \ref{lastprop} that
\begin{eqnarray}
	& & \dfrac{1}{2\beta_1} \left(\left\Vert \widetilde{Q}^{k} \right\Vert_F - \left\Vert \widetilde{Q}^{k+1} \right\Vert_F\right)+\dfrac{1}{2\beta_2} \left(\left\Vert \widetilde{O}^{k} \right\Vert_F - \left\Vert \widetilde{O}^{k+1} \right\Vert_F\right)\\
	&\geq&\left\Vert c_1 D\widetilde{U}^{k+1}-c_2\widetilde{M}^{k+1} \right\Vert_F^2+\left\Vert c_3 \widetilde{U}^{k+1}-c_4 \widetilde{Z}^{k+1} \right\Vert_F^2 + \dfrac{\beta_3}{2}\left\Vert D\widetilde{U}^{k+1}-\widetilde{M}^{k+1}\right\Vert_F^2 + \dfrac{\beta_4}{2}\left\Vert \widetilde{U}^{k+1}-\widetilde{Z}^{k+1}\right\Vert_F^2 ,\nonumber\\
	&+& \dfrac{\beta_1}{2}\left( -\left\Vert\widetilde{M}^{k} \right\Vert_F^2 + \left\Vert \widetilde{M}^{k+1}\right\Vert_F^2\right)  +(\dfrac{3\beta_1}{2}-\beta_1'')\left\Vert \widetilde{M}^{k+1}-\widetilde{M}^k\right\Vert_F^2 \nonumber\\
	&+& \dfrac{\beta_2}{2}\left( -\left\Vert\widetilde{Z}^{k} \right\Vert_F^2 + \left\Vert \widetilde{Z}^{k+1}\right\Vert_F^2\right) + (\dfrac{3\beta_2}{2}-\beta_2'')\left\Vert \widetilde{Z}^{k+1}-\widetilde{Z}^{k} \right\Vert_F^2,  \nonumber
\end{eqnarray}
with $\left(\beta_1'',\beta_2''\right)$, $\left(\beta_3,\beta_4\right)$, and 
$(c_1,c_2,c_3,c_4)$ satisfying. respectively, \eqref{eq beta''}, \eqref{eq beta3beta4}, 
and \eqref{eq c}. Thus, we have
\begin{eqnarray}
	& &\left( \dfrac{1}{2\beta_1}\left\Vert \widetilde{Q}^{k} \right\Vert_F^2 + 
	\dfrac{1}{2\beta_2}\left\Vert \widetilde{O}^{k} \right\Vert_F^2 +
	\dfrac{\beta_1}{2}\left\Vert \widetilde{M}^{k} \right\Vert_F^2 +
	\dfrac{\beta_2}{2}\left\Vert \widetilde{Z}^{k} \right\Vert_F^2 \right)\nonumber \\
	&-& \left( \dfrac{1}{2\beta_1}\left\Vert \widetilde{Q}^{k+1} \right\Vert_F^2 + \dfrac{1}{2\beta_2}\left\Vert \widetilde{O}^{k+1} \right\Vert_F^2 +\dfrac{\beta_1}{2}\left\Vert \widetilde{M}^{k+1} \right\Vert_F^2 + \dfrac{\beta_2}{2}\left\Vert \widetilde{Z}^{k+1} \right\Vert_F^2 \right)\\
	&\geq& \left\Vert c_1 D\widetilde{U}^{k+1}-c_2\widetilde{M}^{k+1} \right\Vert_F^2+\left\Vert c_3 \widetilde{U}^{k+1}-c_4 \widetilde{Z}^{k+1} \right\Vert_F^2 + \dfrac{\beta_3}{2}\left\Vert D\widetilde{U}^{k+1}-\widetilde{M}^{k+1}\right\Vert_F^2 + \dfrac{\beta_4}{2}\left\Vert \widetilde{U}^{k+1}-\widetilde{Z}^{k+1}\right\Vert_F^2 \nonumber\\
	&+&  (\dfrac{3\beta_1}{2}-\beta_1'')\left\Vert \widetilde{M}^{k+1}-\widetilde{M}^k\right\Vert_F^2 \nonumber\\
	&+&    (\dfrac{3\beta_2}{2}-\beta_2'')\left\Vert \widetilde{Z}^{k+1}-\widetilde{Z}^{k} \right\Vert_F^2\geq 0. 
\end{eqnarray}
Introduce the sequence 
\[
x_k=\left( \dfrac{1}{2\beta_1}\left\Vert \widetilde{Q}^{k} \right\Vert_F^2 + 
\dfrac{1}{2\beta_2}\left\Vert \widetilde{O}^{k} \right\Vert_F^2 +
\dfrac{\beta_1}{2}\left\Vert \widetilde{M}^{k} \right\Vert_F^2 +
\dfrac{\beta_2}{2}\left\Vert \widetilde{Z}^{k} \right\Vert_F^2 \right),\quad k=1,2,\ldots~.
\]
This sequence is bounded and decreasing, and therefore it converges. This implies that \begin{eqnarray*}
	& &\left\Vert c_1 D\widetilde{U}^{k+1}-c_2\widetilde{M}^{k+1} \right\Vert_F^2+\left\Vert c_3 \widetilde{U}^{k+1}-c_4 \widetilde{Z}^{k+1} \right\Vert_F^2 + \dfrac{\beta_3}{2}\left\Vert D\widetilde{U}^{k+1}-\widetilde{M}^{k+1}\right\Vert_F^2 + \dfrac{\beta_4}{2}\left\Vert \widetilde{U}^{k+1}-\widetilde{Z}^{k+1}\right\Vert_F^2 \\
	& &+  (\dfrac{3\beta_1}{2}-\beta_1'')\left\Vert \widetilde{M}^{k+1}-\widetilde{M}^k\right\Vert_F^2
	+    (\dfrac{3\beta_2}{2}-\beta_2'')\left\Vert \widetilde{Z}^{k+1}-\widetilde{Z}^{k} \right\Vert_F^2
\end{eqnarray*}
converges to $0$. This implies that the sequences
\[
\{\widetilde{Q}^{k}\}_{k\geq 0},~~ \{\widetilde{O}^{k}\}_{k\geq 0},~~
\{\widetilde{M}^{k}\}_{k\geq 0},~~ \{\widetilde{Z}\}_{k\geq 0},
\]
are bounded and, therefore, the sequences
\[
\{Q^{k}\}_{k\geq 0},~~\{O^{k}\}_{k\geq 0},~~\{M^{k}\}_{k\geq 0},~~\{Z\}_{k\geq 0}
\]
also are bounded. Moreover, the sequences
\[
\{D\widetilde{U}^{k}\}_{k\geq 1},~~ \{\widetilde{U}^{k}\}_{k\geq 1},~~ 
(\{DU^{k}\}_{k\geq 1},~~ \{U^{k}\}_{k\geq 1} )
\]
are bounded. We conclude that
\begin{eqnarray}
	&\displaystyle{\lim_{k\to\infty}}& \left\Vert \widetilde{Z}^{k+1}-\widetilde{Z}^{k} \right\Vert_F^2=
	\lim_{k\to\infty} \left\Vert Z^{k+1}-Z^{k} \right\Vert_F^2=0, \nonumber\\
	&\displaystyle{\lim_{k\to\infty}}& \left\Vert \widetilde{M}^{k+1}-\widetilde{M}^{k} \right\Vert_F^2=
	\lim_{k\to\infty} \left\Vert M^{k+1}-M^{k} \right\Vert_F^2=0, \nonumber\\
	&\displaystyle{\lim_{k\to\infty}}& \left\Vert c_1 D\widetilde{U}^{k+1}-c_2\widetilde{M}^{k+1} 
	\right\Vert_F^2=0,\label{end 3}\\
	&\displaystyle{\lim_{k\to\infty}}& \left\Vert c_3 \widetilde{U}^{k+1}-c_4\widetilde{Z}^{k+1} 
	\right\Vert_F^2=0, \label{end 4}\\
	&\displaystyle{\lim_{k\to\infty}}& \left\Vert D\widetilde{U}^{k+1}-\widetilde{M}^{k+1} 
	\right\Vert_F^2=0, \label{end 5}\\
	&\displaystyle{\lim_{k\to\infty}}& \left\Vert \widetilde{U}^{k+1}-\widetilde{Z}^{k+1} \right\Vert_F^2=0. 
	\label{end 6}
\end{eqnarray}

As we have discussed, $c_1,c_2,c_3,c_4>0$, with $c_1\neq c_2$ and $c_3\neq c_4$. Furthermore,
$c_1>c_2$ and $c_3>c_4$. We can see that
\begin{equation}
	(c_1^2-c_2^2)\| D\widetilde{U}^{k+1} \|_F^2 \leq \| c_1 D\widetilde{U}^{k+1}-
	c_2 \widetilde{M}^{k+1}  \|_F^2 + \| c_2 \widetilde{M}^{k+1} - c_2 D\widetilde{U}^{k+1} 
	\|_F^2 \label{last 1}
\end{equation}
and 
\begin{equation}
	(c_3^2-c_4^2)\| \widetilde{U}^{k+1} \|_F^2 \leq \| c_3 \widetilde{U}^{k+1}-c_4 
	\widetilde{Z}^{k+1}  \|_F^2 + \| c_4 \widetilde{Z}^{k+1} - c_4 \widetilde{U}^{k+1} \|_F^2.
	\label{last 2}
\end{equation}
Moreover, we have 
\begin{equation}
	\| \widetilde{Z}^{k+1} \|_F\leq \| \widetilde{Z}^{k+1}-\widetilde{U}^{k+1} \|_F +
	\| \widetilde{U}^{k+1} \|_F, \;\;  \| \widetilde{M}^{k+1} \|_F\leq \| \widetilde{M}^{k+1}-
	D\widetilde{U}^{k+1} \|_F + \| D\widetilde{U}^{k+1} \|_F. \label{last 3}
\end{equation}
Consequently, from the inequalities in \eqref{last 1}, \eqref{last 2}, and \eqref{last 3},
and by using \eqref{end 3} and \eqref{end 5}, as well as from \eqref{end 4} and 
\eqref{end 6}, we obtain 
\begin{eqnarray*}
	\lim_{k\to \infty}M^{k}=M^*,~~\lim_{k\to \infty}Z^{k}=Z^*=U^*,~~
	\lim_{k\to \infty}U^{k}=U^*,~~\lim_{k\to \infty}DU^{k}=DU^*=M^*.
\end{eqnarray*}
This concludes the proof.
\end{appendices}
\end{document}